\documentclass{article}

\usepackage[
    a4paper,
    bindingoffset=0.2in,
    left=1.25in,
    right=1.25in,
    top=1.25in,
    bottom=1.25in,
    footskip=.25in
]{geometry}

\usepackage{microtype}
\usepackage[parfill]{parskip}
\usepackage[shortlabels]{enumitem}
\usepackage{authblk}
\usepackage{fancyhdr}

\usepackage{graphicx}
\usepackage{booktabs}
\usepackage{placeins}

\usepackage{amsmath}
\usepackage{amsfonts}
\usepackage{amssymb}
\usepackage{amsthm}
\usepackage{mathbbol}
\usepackage{prodint}

\usepackage[round,authoryear]{natbib}
\usepackage{xcolor}
\usepackage{hyperref}

\hypersetup{
    colorlinks=true,
    linkcolor=red,
    citecolor=blue,
    urlcolor=blue
}

\DeclareSymbolFontAlphabet{\amsmathbb}{AMSb}

\newcommand{\floor}[1]{\left\lfloor #1 \right\rfloor}
\newcommand{\ceil}[1]{\left\lceil #1 \right\rceil}

\newcommand\independent{\protect\mathpalette{\protect\independenT}{\perp}}
\def\independenT#1#2{\mathrel{\rlap{$#1#2$}\mkern2mu{#1#2}}}

\newtheorem{theorem}{Theorem}[section]
\newtheorem{corollary}[theorem]{Corollary}
\newtheorem{lemma}[theorem]{Lemma}
\newtheorem{proposition}[theorem]{Proposition}
\newtheorem{condition}{Condition}

\theoremstyle{remark}
\newtheorem{remark}{Remark}[section]

\theoremstyle{definition}

\allowdisplaybreaks

\title{Censored Heteroscedastic Extremes}
\author[1]{Martin Bladt\thanks{%
Email: \href{mailto:martinbladt@math.ku.dk}{%
\text{martinbladt@math.ku.dk}}}}

\author[1]{Theodor Henningsen\thanks{%
Email: \href{mailto:th@math.ku.dk}{%
\text{th@math.ku.dk}}}}

\affil[1]{Department of Mathematical Sciences, University of Copenhagen}

\begin{document}
\maketitle

\begin{abstract}
    We study estimation of tail heterogeneity for non-identically distributed extreme observations subject to random right-censoring. In the uncensored setting, such heterogeneity is described by the event scedasis function, which measures the relative contribution of different design points to the upper tail. Under censoring, however, the observed tail heterogeneity is contaminated by the censoring scedasis functions, and applying uncensored techniques targets the wrong object. We propose a Beran-type estimator of the relative event scedasis, which is consistent under mild conditions. To obtain these results, survival analysis representations at an upper order statistics are extended to the non-identically distributed case; specifically, we develop conditional Nelson--Aalen and Beran theory on increasing intervals whose random endpoint is dominated, with probability tending to one, by a deterministic high local quantile. In particular, we derive a martingale array representation of the conditional Nelson--Aalen estimator with explicit error bounds depending only on the sample fraction and the bandwidth.
    Simulations demonstrate the finite-sample performance of the method, and an application to French property-casualty insurance claims illustrates how heterogeneous censoring can distort naive scedasis estimates.
\end{abstract}
\section{Introduction}

Extreme observations are often not identically distributed. In many applications, the contribution to the upper tail varies systematically over time, or according to another deterministic design variable. When the tail index is constant across the design space, this variation can be described by a scedasis function, which measures the relative frequency 
of extremes across design points. Estimating this function is useful for identifying regions that contribute disproportionately to extreme risk.

This paper studies scedasis estimation when the observations are subject to random right-censoring. The setting is motivated by insurance claims data, where the tail behavior of large claims may vary over accident time, while recent claims are often only partially developed. In such data, the observed payment is a right-censored version of the ultimate claim size. Since censoring is itself heterogeneous over accident time, then applying an uncensored scedasis estimator directly to the observed censored losses targets the wrong object. Indeed, under proportional event and censoring tails, the observed tail heterogeneity is governed by the product $c_{F}c_{G}$, where $c_{F}$ is the target event scedasis and $c_{G}$ is the censoring scedasis. Thus the observed extremes only identify $c_{F}c_{G}$ rather than $c_{F}$.
Similar issues can arise in other heavy-tailed settings, such as operational-risk or litigation losses, where ultimate losses may remain unresolved at the observation date. Capacity-constrained demand systems provide a related censoring analogue when demand is observed only up to a time-varying exogenous capacity.

Statistical inference for independent heterogeneous heavy-tailed variables includes the early work of~\cite{mejzler1956problem} and has since been developed in various directions. The paper of~\cite{davisonsmith1990models} considered a linear trend in shape and scale parameters for generalized Pareto distributions , while~\cite{hallTajvidi2000nonparametric} estimated trends in the parameters of generalized extreme-value distributions using non-parametric, local estimators, for which asymptotic results were established under locally constant or linear trends. The paper of~\cite{deHaan2015TailTrend} considered inference of parametric trends for heteroscedastic data and subsequently~\cite{Einmahl_Haan_Zhou_2014} proposed a general framework for inference of heteroscedastic extremes including strong asymptotic results. This work was later extended by~\cite{bucher2024statistics} to a time-series setting, and~\cite{he2026extreme} provided a more general treatment of heterogeneous extreme observations using tail averages.
Statistical inference of independent heavy-tailed data subject to right-censoring has been studied in the independent and identically distributed (iid) setting, among others by~\cite{beirlant2007estimation, Einmahl_Villetard} who proposed estimators for the tail index and extreme quantiles, for which they proved asymptotic Gaussianity, and more recently by~\cite{Bladt_Rodionov_2025} who also proved limit results for general extreme Kaplan--Meier integrals. These contributions address either heterogeneity or censoring, but not the interaction between deterministic tail heterogeneity and heteroscedastic random censoring, which is central here.
Related extreme-value literature that deviates from the fully observed iid setting includes estimation of local and global trends in the tail index~\citep{deHaan_Zhou_2021}, as well as tail index regression using kernel estimators in the presence of covariates~\citep{Daouia2010, GardesStupfler2014}.

Our approach is based on local, or conditional, product-limit estimation. For a fixed design point, we estimate the event distribution locally around it by a Beran-type estimator~\citep{beran1981nonparametric}. This local construction is essential, since under heteroscedastic censoring, global Kaplan--Meier estimators generally estimate censoring-weighted mixture distributions rather than the true event mixture distributions. Classical asymptotic results for the Beran estimator are typically formulated on deterministic intervals with finite upper endpoint; see, for example~\cite{dabrowska1987nonparametric} and~\cite{gonzalezManteiga1994asymptotic}. By contrast, in the present extreme-value framework, the local product-limit estimator is evaluated at a large order statistic that is increasing to the right endpoint as the sample size increases. Consequently the Nelson--Aalen and Beran approximations must hold on a sequence of intervals whose upper endpoint increases with the sample. Such increasing-set representations are only understood for the Kaplan--Meier estimator in the iid random-censorship setting; see~\cite{STUTE_1992} and~\cite{Csorgo_96}. To study censored heteroscedastic extremes, the present paper requires the extension of these two strands by developing local Nelson--Aalen and Beran theory on increasing tail sets for a non-identically distributed deterministic-design array. The resulting survival-theoretic statements are not specific to heteroscedastic tail models, and thus can be considered a contribution to probabilistic representations in survival analysis.

The rest of the paper is organized as follows. Section~\ref{section_model_and_estimators} introduces the censored heteroscedastic extremes model, shows that the observed extremes identify $c_{F}c_{G}$, and defines the local Beran-type scedasis estimator. Section~\ref{section_main_results} contains the main asymptotic results: local empirical convergence, denominator stability, a Nelson--Aalen array representation on increasing tail sets, and consistency of the relative scedasis estimator. Section~\ref{section_simulation_study} presents a simulation study comparing the proposed estimator with a benchmark applied directly to the censored observations. Section~\ref{section_real_data_analysis} applies the method to French property-casualty insurance claim data. Section~\ref{section_discussion} discusses alternative estimators, normalization, and boundary correction. Proofs and additional results are delegated to the appendices.

\section{Model and estimators}\label{section_model_and_estimators}
In this section we introduce the heteroscedastic model of real-valued random variables subject to random (non-informative) censoring and we show that the observed tail heterogeneity is governed by the product of the target and censoring scedasis functions. In addition, we introduce local limit survival estimators from which we construct the estimator of the relative target scedasis function. 

\subsection{Heteroscedastic censoring model}
At design points $i=1,\ldots,n$ we observe independent tuples $(Z_{1}^{(n)}, \delta_{1}^{(n)}),\ldots,(Z_{n}^{(n)}, \delta_{n}^{(n)})$ of right-censored observations and their corresponding concomitants
\begin{align*}
    Z_{i}^{(n)} = X_{i}^{(n)}\wedge C_{i}^{(n)},
    \qquad
    \delta_{i}^{(n)} = 1\{Z_{i}^{(n)} = X_{i}^{(n)}\}.
\end{align*}
The target random variables $X_{1}^{(n)},\ldots, X_{n}^{(n)}$ follow continuous distributions $F_{n,1},\ldots, F_{n,n}$, the censoring variables $C_{1}^{(n)},\ldots, C_{n}^{(n)}$ follow continuous distributions $G_{n,1},\ldots, G_{n,n}$ that share a common right endpoint $x^{*} = \sup\{x:F_{n,i}(x) < 1\}  = \sup\{x:G_{n,i}(x) < 1\} \in (0,\infty]$. We assume there exist continuous distribution functions $F$ and $G$ with the same right endpoint and continuous densities $c_{F}$ and $c_{G}$ on $[0,1]$ that are strictly positive on $(0,1)$, such that
\begin{align}\label{eq_Scedasis_F_and_G}
    \lim_{x \to x^{*}}\frac{1-F_{n,i}(x)}{1-F(x)} = c_{F}\Big(\frac{i}{n}\Big), \quad
    \lim_{x \to x^{*}}\frac{1-G_{n,i}(x)}{1-G(x)} = c_{G}\Big(\frac{i}{n}\Big),
\end{align}
uniformly for all $n \in \amsmathbb{N}$ and all $i=1,\ldots, n$. This assumption is similar to (1.1) in~\cite{Einmahl_Haan_Zhou_2014}.
We impose the standard assumption of entirely random (non-informative) censoring,
\begin{align}\label{eq_entirely_random_censoring}
    X_{i}^{(n)} \independent C_{i}^{(n)},
\end{align}
for all $i\leq n$ and all $n$ (see~\cite{AndersenBorganGillKeiding1993,Stute1995}), and assume that no ties are present in the observations. Denote by $H_{n,1},\ldots, H_{n,n}$ the distribution functions of $Z_{1}^{(n)}, \ldots , Z_{n}^{(n)}$ and note that 
\begin{align*}
     1-H_{n,i}(x) &= \amsmathbb{P}(Z_{i}^{(n)} > x) =  (1-F_{n,i}(x))(1-G_{n,i}(x)),
\end{align*}
for all $i\leq n$, due to~\eqref{eq_entirely_random_censoring}. Since both target and censoring variables are heteroscedastic with the same upper endpoint, the observations $Z_{1}^{(n)},\ldots , Z_{n}^{(n)}$ are also heteroscedastic. This follows from the fact that $1-H(x):= (1-F(x))(1-G(x))$ is continuous and hence
\begin{align}\label{eq_H_scedasis_dens}
    \lim_{x\to x^{*} } \frac{1-H_{n,i}(x)}{1-H(x)}
    =
     \lim_{x\to x^{*}} \frac{(1-F_{n,i}(x))(1-G_{n,i}(x))}{(1-F(x))(1-G(x))} = c_{F}\Big(\frac{i}{n}\Big)c_{G}\Big(\frac{i}{n}\Big) =: c_{H}\Big(\frac{i}{n}\Big).
\end{align}
However, $c_{H}$ does not necessarily integrate to one. Consequently the observations identify $c_{H} = c_{F} c_{G}$, up to a multiplicative constant, rather than the target scedasis density $c_{F}$. In what follows, we concentrate on \emph{relative} estimation of the scedasis function.

In addition, we introduce for every $i\leq n$ the following notation
\begin{align*}
    H_{n,i}^{1}(x) &:=\amsmathbb{P}(Z_{i}^{(n)} \leq x, \delta_{i}^{(n)} = 1)= \int_{0}^{x} (1-G_{n,i}(y-)) F_{n,i}(\mathrm{d} y),
    \\
    \Lambda_{i}(x) &:=\int_{0}^{x} \frac{1}{1-F_{n,i}(y-)} F_{n,i}(\mathrm{d} y)
    =
    \int_{0}^{x} \frac{1}{1-H_{n,i}(y-)} H^{1}_{n,i}(\mathrm{d} y), \quad x < x^{*},
\end{align*}
where the latter corresponds to the target cumulative hazard of $X_{i}^{(n)}$. Notice that we write $G_{n,i}(y-)$ and likewise $F_{n,i}(y-)$, $H_{n,i}(y-)$ to align notation with the existing survival theory, even though these functions are continuous and the minus may be omitted.

\subsection{Local limit estimators}
Inference of the individual survival functions $1-F_{n,i}$ is key for estimating the target scedasis function $c_{F}$. To this end we introduce the following local kernel estimators that subsample observations around a given $s\in (0,1)$,
\begin{align*}
    \amsmathbb{H}_{n}(x \vert s) &= \frac{1}{nh_{n}}\sum_{i=1}^{n} 1\{Z_{i}^{(n)} \leq x\}K\Big(\frac{s - i/n}{h_{n}}\Big), \\
    \amsmathbb{H}^{1}_{n}(x \vert s) &= \frac{1}{nh_{n}}\sum_{i=1}^{n} 1\{Z_{i}^{(n)}\leq x, \delta_{i}^{(n)} = 1 \}K\Big(\frac{s - i/n}{h_{n}}\Big), \\
    \mathbb{\Lambda}^{n}(x \vert s) &= \int^{x}_{0} \frac{1}{1-\amsmathbb{H}_{n}(y- \vert s)} \amsmathbb{H}_{n}^{1}(\mathrm{d} y \vert s), 
    \quad 
    1 -\amsmathbb{F}^{(n)}(x\vert s) = \Prodi_{0 \leq y \leq x}(1-\mathbb{\Lambda}^{n}(\mathrm{d} y \vert s)),
\end{align*}
where $K$ is a symmetric kernel density, which is compactly supported and bounded by some constant $M_{K}$, and $h_{n}$ is a deterministic bandwidth sequence decreasing to zero with $nh_{n}$ diverging to infinity. The former two estimators are based on the basic local-averaging idea of ~\cite{Nadaraya1964,Watson1964}, and the latter two estimators are versions of the famous conditional Nelson--Aalen and Beran estimators, respectively (see~\cite{beran1981nonparametric,dabrowska1987nonparametric}), where the conditioning is done on the deterministic temporal variable $s\in (0,1)$, as opposed to a random covariate. Proving representation and limit results for these estimators is the main ingredient used to state a limit theorem for the relative scedasis estimator introduced below. As stated, these results are of independent interest.

\subsection{Scedasis estimator}
In the case of no censoring, that is, when $C_{i}^{(n)} = \infty$, $\amsmathbb{P}$-a.s., for all $i=1,\ldots, n$, Einmahl, de Haan and Zhou (see~\cite{Einmahl_Haan_Zhou_2014}) proposed the following estimator for the scedasis density 
\begin{align}\label{eq_Einmahl_estimator}
    \hat{c}_{F}^{\text{EHZ}}(s)&:= 
    \frac{1}{k_{n}h_{n}}
    \sum_{i=1}^{n} 1\{X_{i}^{(n)} > X_{n:n-k_{n}}\} K\Big(\frac{s - i/n}{h_{n}}\Big), \quad s\in (0,1),
\end{align}
where $k_{n}\leq n$ is an integer scaling sequence with $k_{n}/n \to 0$, $k_{n}h_{n}\to \infty$ as $n\to \infty$ and $X_{n:1} \leq \ldots \leq X_{n:n}$ are the order statistics of the observations. Recall that in the case of no censoring, the conditional empirical distribution function coincides with the Beran estimator. Consequently 
\begin{align*}
    \hat{c}_{F}^{\text{EHZ}}(s)&= 
     \frac{n}{k_{n}} (1-\amsmathbb{F}^{(n)}(X_{n:n-k_{n}} \vert s))\Big( \frac{1}{nh_{n}} \sum_{i=1}^{n}  K\Big(\frac{s - i/n}{h_{n}}\Big)\Big),
\end{align*}
where the last multiplicative factor is a Riemann sum of the integral of $K$ and thus converges to $1$.

Here $k_{n}$ corresponds to the number of observations above $X_{n:n-k_{n}}$, but this information is no longer available when right-censoring is present. It turns out that approximating this number of exceedances with a global estimator such as $n(1-\amsmathbb{F}^{(n)}(Z_{n:n-k_{n}}))$, where $Z_{n:1} \leq \ldots \leq Z_{n:n}$ are the order statistics of the right-censored observations, yields, in full generality, biased estimators; see Section~\ref{section_discussion} for further details. Consequently, we propose the following relative scedasis estimator for fixed anchor point $s_{0} \in (0,1)$:
\begin{align*}
      \hat{c}_{F}(s \vert s_{0}):=
         \frac{1-\amsmathbb{F}^{(n)}(Z_{n:n-k_{n}} \vert s)}{1-\amsmathbb{F}^{(n)}(Z_{n:n-k_{n}} \vert s_{0})},
         \qquad s \in (0,1).
\end{align*}
By construction $\hat{c}_{F}(s \vert s_{0})$ targets the relative scedasis $c_{F}(s)/c_{F}(s_{0})$, which yields the correct shape of $c_{F}$ but not its correct magnitude. In Section~\ref{section_discussion} we discuss why extending the global cumulative scedasis estimator of~\cite{Einmahl_Haan_Zhou_2014} is not appropriate when censoring is present. This motivates estimating the scedasis function locally. In addition, we also discuss ways to modify the above estimator in order to target $c_{F}(s)$, but we now focus on estimating the relative scedasis function.

\section{Main results}\label{section_main_results}
In this section, we develop uniform local empirical convergence results necessary for our main decomposition of the local Nelson--Aalen estimator up to the random endpoint $Z_{n:n-k_{n}}$. Notably, the decomposition yields that the leading term can be identified as a square-integrable martingale. We establish uniform consistency results for both the leading term and our local Beran estimator. Remarkably, these results do not rely on extreme-value assumptions and are therefore of general interest. Uniform consistency of the local Beran estimator then implies consistency of the relative scedasis estimator. In addition, we quantify which tuning sequences are admissible and show why the observations inherit the target scedasis function in the very special case of homoscedastic censoring.

\subsection{Assumptions}
Smoothness conditions in the covariate direction are classical in kernel regression and control the deterministic smoothing bias induced by the kernel. We impose such conditions directly on the deterministic targets of $\amsmathbb{H}_{n}(\cdot\vert s)$ and $\amsmathbb{H}_{n}^{1}(\cdot\vert s)$. These conditions suffice for the local empirical convergence results in Subsection~\ref{subsection_local_empirical_convergence_results}. The subsequent arguments additionally require mild regularity of the target and censoring distribution functions.
\begin{condition}\label{cond_H_H1_C1_F_G_Lip}
    There exist maps $F,G:[0,x^{\ast})\times[0,1]\to[0,1]$, such that $x\mapsto F(x,u)$, $x\mapsto G(x,u)$ are continuous distribution functions for every $u\in[0,1]$, with  $F(x,i/n) = F_{n,i}(x)$ and $G(x,i/n) = G_{n,i}(x)$, for every $n\in\amsmathbb{N}$, $i=1,\ldots,n$, and $x<x^{\ast}$. Define
    \begin{align*}
        1-H(x,u)
        &:=
        (1-F(x,u))(1-G(x,u)),\\
        H^{1}(x,u)
        &:=
        \int_{0}^{x}
        (1-G(y,u))F(\mathrm{d} y,u).
    \end{align*}
    For every fixed $s\in(0,1)$ there exist $\rho_{s}>0$ such that $I_{s} :=[s-\rho_{s},s+\rho_{s}]\subset(0,1)$.
    Assume that $u\mapsto F(x,u)$ and $u\mapsto G(x,u)$ are Lipschitz on $I_{s}$ uniformly in $x<x^{\ast}$ and that the maps $u\mapsto H(x,u)$ and $u\mapsto H^{1}(x,u)$ are differentiable on $I_{s}$, for every $x < x^{\ast}$, with derivatives that are Lipschitz uniformly in $x$. 
\end{condition}
Note that smoothness assumptions could alternatively be placed on $F(x,\cdot)$ and $G(x,\cdot)$. 
\begin{remark}[Sufficient condition for Condition~\ref{cond_H_H1_C1_F_G_Lip}]\label{remark_sufficient_condition}
    Let $F,G$ be given as in Condition~\ref{cond_H_H1_C1_F_G_Lip}. 
    Fix $s\in(0,1)$ and suppose that, for every $u\in I_{s}$, $F(\cdot,u)$ is absolutely continuous with density $f(\cdot, u)$. Suppose that, for every $x<x^{\ast}$, the maps $u\mapsto f(x,u)$ and $u\mapsto G(x,u)$ are continuously differentiable on $I_{s}$. Assume that the first derivative of $u \mapsto f(x,u)$ admits an integrable envelope in $x$ and is Lipschitz in $u$ with an integrable Lipschitz modulus. Assume, correspondingly, that the first derivative of $u \mapsto G(x,u)$ is uniformly bounded and Lipschitz in $u$ uniformly over $x<x^{\ast}$. Then Condition~\ref{cond_H_H1_C1_F_G_Lip} holds.
\end{remark}
The Lipschitz continuity of the derivatives in Condition~\ref{cond_H_H1_C1_F_G_Lip} yields quadratic Taylor remainders for $H$ and $H^{1}$ that can be uniformly bounded in $x$. Symmetry of the kernel implies a smoothing bias of order $O(h_{n}^{2})$. Closely related smoothness conditions are employed in Lemma~SA.30 of the supplementary material to~\citet{Cattaneo_Yu_2025}, where smoothness of order $p+1$, combined with a local polynomial of order $p$, yields a smoothing bias of order $O(h_{n}^{p+1})$. In order to relate the rate of the bias and variance we impose the usual assumption that the bandwidth satisfies $nh_{n}^{5} \to 0$ (see Corollary 1(b) of~\cite{Calonico_Cattaneo_Farrell_2018} with $S=\varsigma=1$). Here, we also assume the following mild condition: $\log(n)/(k_{n}h_{n}) \to 0$.

\subsection{Local empirical convergence results}\label{subsection_local_empirical_convergence_results}
The next step is to control the local empirical building blocks from which the local Nelson--Aalen and Beran estimators are constructed. The following lemma shows that, on the random increasing interval $[0, Z_{n:n-k_{n}}]$, the local kernel estimators behave as if the observations were drawn from the distribution at the local design point $\floor{ns}/n$. Throughout the paper we consider only $n$ large enough so that $\floor{ns}\geq 1$.
\begin{lemma}\label{lemma_H_n_and_H_n_1_unif_cons}
    Assume Condition~\ref{cond_H_H1_C1_F_G_Lip}. Then, for $s \in (0,1)$ it holds that 
    \begin{align*}
        \sup_{x \leq Z_{n:n-k_{n}}}
        \big\vert \amsmathbb{H}_{n}(x \vert s) - H_{n,\floor{ns}}(x)\big\vert 
        = \mathcal{O}_{\amsmathbb{P}}\Big(\frac{1}{\sqrt{nh_{n}}}\Big), 
        \\
        \sup_{x \leq Z_{n:n-k_{n}}}
        \big\vert \amsmathbb{H}_{n}^{1}(x \vert s) - H_{n,\floor{ns}}^{1}(x)\big\vert  
        = \mathcal{O}_{\amsmathbb{P}}\Big(\frac{1}{\sqrt{nh_{n}}}\Big),
    \end{align*}
    as $n \to \infty$.
\end{lemma}
The following result is closely related to those of~\cite{wellner1978limit} and~\cite{csaki1975some}.

\begin{lemma}\label{lemma_H_tail_fraction}
   Assume Condition~\ref{cond_H_H1_C1_F_G_Lip} and that $nh_{n}^{2}/k_{n}$ and $\sqrt{nh_{n}}/k_{n}$ are decreasing. For any $s \in (0,1)$ it then holds that 
  \begin{align*}
    \sup_{x\leq Z_{n:n-k_{n}}}
    \Big\vert 
        \frac{1-\amsmathbb{H}_{n}(x-\vert s)}
        {1-H_{n,\floor{ns}}(x-)}
        -1
    \Big\vert 
    =
    \mathcal{O}_{\amsmathbb{P}}\Big(
        \frac{n}{k_{n}}\Big(h_{n}^{2} 
        + 
        \frac{\sqrt{h_{n}}}{\sqrt{n}}\Big)
        +
        \frac{1}{\sqrt{k_{n}h_{n}}}
        \Big),
\end{align*}
as $n \to \infty$.
\end{lemma}
The latter term on the right-hand side is directly comparable to the rate $\mathcal{O}_{\amsmathbb{P}}(1/\sqrt{k_{n}})$ in Lemma 1 of~\cite{Csorgo_96} for $\gamma = 0$; however, we have an additional factor of $1/\sqrt{h_{n}}$ in the rate since we use kernel estimators. In fact, if the empirical estimator $\amsmathbb{H}_{n}(\cdot \vert s)$ were based on iid (rather than heterogeneous observations), each having distribution $H_{n,\floor{ns}}$, the latter term above would be the leading term. We prove this in Lemma~\ref{lemma_H_tilde_tail_fraction} in Appendix~\ref{appendix_additional_empirical_processs_results}.\\
The first term on the above right-hand side is the tail-weighted induced bias between $\amsmathbb{H}_{n}(\cdot \vert s)$ and $H_{n,\floor{ns}}$, while the second term is the discrepancy contribution caused by the fact that $\amsmathbb{H}_{n}(\cdot \vert s)$ is based on the observations each with distribution $H_{n,i}$, while the denominator is the tail of $H_{n,\floor{ns}}$.

Having established these local empirical results, we move on to the local Nelson--Aalen and Beran estimators.

\subsection{Local survival decomposition and convergence rates}\label{subsection_local_survival_decomposition_and_convergence_rates}
In this subsection we show that $\mathbb{\Lambda}^{n}(\cdot \vert s)  - \Lambda_{\floor{ns}}(\cdot)$ admits a weak array representation that for fixed $n\in \amsmathbb{N}$ and $s \in (0,1)$ is a square-integrable martingale. This representation leads to consistency of our local survival estimators. The proofs follow those of~\cite{STUTE_1992} using that $\mathbb{\Lambda}^{n}(\cdot \vert s)$ can be identified as a $U$-statistic process $U_{n}(\cdot \vert s)$ and showing that its Hájek projection yields one of the leading terms. However, classical theory of $U$-statistic processes is concerned with iid data and to this end we introduce locally transported random arrays
 \begin{align*}
        &X_{\floor{ns}}^{(n,i)} := F_{n,\floor{ns}}^{\leftarrow}(F_{n,i}(X_{i}^{(n)})),
        \quad
        C_{\floor{ns}}^{(n,i)} := G_{n,\floor{ns}}^{\leftarrow}(G_{n,i}(C_{i}^{(n)})), \\
        & Z_{\floor{ns}}^{(n,i)} := X_{\floor{ns}}^{(n,i)} \wedge C_{\floor{ns}}^{(n,i)}, 
        \qquad \quad \ \ 
        \delta_{\floor{ns}}^{(n,i)} := 1\{Z_{\floor{ns}}^{(n,i)} = X_{\floor{ns}}^{(n,i)}\},
\end{align*}
from which local $U$-statistic processes are constructed. This is the key to establishing maximal inequalities of degenerate local $U$-statistic processes, extending the results of~\cite{U_statistics_Stute_94}. However, introducing estimators based on the locally transported arrays adds a technical layer since their differences to the original estimators must be controlled.
\\
\begin{theorem}\label{thm_Nelson_Aalen_decomposition}
  Assume Condition~\ref{cond_H_H1_C1_F_G_Lip} and that $\log(n)nh_{n}^{2}/k_{n}$ and $\log(n)\sqrt{n}/k_{n}$ are decreasing. 
    For every $s \in (0,1)$, it then holds that
     \begin{align*}
         \mathbb{\Lambda}^{n}(x\vert s)  - \Lambda_{\floor{ns}}(x) 
        =  R_{n}^{0}(x \vert s) +  \mathbb{\Lambda}^{n, *}(x\vert s) - \Lambda_{\floor{ns}}(x),
     \end{align*}
    where 
     \begin{align*}
          &\mathbb{\Lambda}^{n, *}(x\vert s) - \Lambda_{\floor{ns}}(x) 
          = \\
          &R_{n}^{*}(x \vert s)
          +
          \frac{1}{nh_{n}}\sum_{i=1}^{n}K\Big(\frac{s - i/n}{h_{n}}\Big)\Big[ \int_{0}^{x} \frac{1\{ Z_{\floor{ns}}^{(n,i)} < y\} - H_{n,\floor{ns}}(y-)}{(1-H_{n,\floor{ns}}(y-))^{2}}  H_{n,\floor{ns}}^{1}(\mathrm{d} y) \\
        & \quad +
         \frac{1\{ Z_{\floor{ns}}^{(n,i)} \leq x\}\delta_{\floor{ns}}^{(n,i)} - H_{n,\floor{ns}}^{1}(x) }{1-H_{n,\floor{ns}}(x-)} 
       - 
       \int_{0}^{x}
    \frac{1\{Z_{\floor{ns}}^{(n,i)} \leq y\}\delta_{\floor{ns}}^{(n,i)} - H_{n,\floor{ns}}^{1}(y)}{(1-H_{n,\floor{ns}}(y-))^{2}} H_{n,\floor{ns}}(\mathrm{d} y-)\Big],
    \end{align*}
     such that 
      \begin{align*}
        \sup_{x \leq Z_{n:n-k_{n}}} \vert R_{n}^{0}(x \vert s) \vert &= \mathcal{O}_{\amsmathbb{P}}\Big(\frac{n\log(n)}{k_{n}}(h_{n}^{2} \vee 1/\sqrt{n} \vee 1/(nh_{n}))\Big),
        \\ \sup_{x \leq Z_{n:n-k_{n}}} \vert R_{n}^{*}(x \vert s) \vert &= \mathcal{O}\Big(\frac{1}{k_{n}h_{n}}\Big).
    \end{align*} 
\end{theorem}
It turns out that the leading term of $\mathbb{\Lambda}^{n, *}(\cdot\vert s) - \Lambda_{\floor{ns}}(\cdot)$ is a square-integrable martingale, which in turn leads to the consistency rate in Theorem~\ref{thm_Lambda_consistency_speed}.

The above representation is a kernel-adapted, heterogeneous-data analogue version of the representation of Proposition 5 in~\cite{Csorgo_96}; however, our remainder $R_{n}^{*}(x \vert s)$ is a kernel-induced bias and is therefore not present in their result. The uniform negligibility of $R_{n}^{0}(x \vert s)$ on $[0,Z_{n:n-k_{n}}]$ goes back to~\cite{Stute1995} and was later refined by~\cite{Csorgo_96} who obtained a rate of order $\mathcal{O}_{\amsmathbb{P}}(1/k_{n})$. This is directly comparable to our rate in the last argument of the maximum, $\mathcal{O}_{\amsmathbb{P}}(\log(n)/(k_{n}h_{n}))$, where we pay an additional factor of $1/h_{n}$ due to the use of kernels and a technical factor $\log(n)$ which arises from a concentration inequality. The other arguments in the maximum are biases induced by the data being independent but non-identically distributed. To obtain a uniform convergence rate of the local Nelson--Aalen estimator of order $1/\sqrt{k_{n}h_{n}}$ we require the preceding remainder rates to be of smaller order.

\begin{theorem}\label{thm_Lambda_consistency_speed}
     Assume Condition~\ref{cond_H_H1_C1_F_G_Lip}, let $s \in (0,1)$, and assume that $(k_{n},h_{n})$ satisfies
     \begin{align}\label{eq_n_log_n_h2_over_kn_little_o}
        \frac{n\log(n)}{k_{n}}(h_{n}^{2} \vee 1/\sqrt{n})
        = o\Big(\frac{1}{\sqrt{k_{n}h_{n}}}\Big).
    \end{align}    
    It then holds that
     \begin{align*}
         \sup_{x \leq Z_{n:n-k_{n}}} \vert \mathbb{\Lambda}^{n}(x\vert s) - \Lambda_{\floor{ns}}(x) \vert = \mathcal{O}_{\amsmathbb{P}}\Big(\frac{1}{\sqrt{k_{n}h_{n}}}\Big).
     \end{align*}
\end{theorem}
The above theorem is the relevant kernel extension of Theorem 1 of~\cite{Csorgo_96} when the data are heteroscedastic. Transferring this rate to the Beran estimator is usually done by imposing the standard assumption of continuity of the target $F_{n,\floor{ns}}(\cdot)$, see~\citep{Breslow,Stute1995}, but this is already a part of our fundamental setup.
\begin{proposition}\label{prop_Beran_consistency_speed}
     Assume Condition~\ref{cond_H_H1_C1_F_G_Lip} and~\eqref{eq_n_log_n_h2_over_kn_little_o}, and let $s \in (0,1)$. It then holds that
     \begin{align*}
         \sup_{x \leq Z_{n:n-k_{n}}} \Big\vert \frac{\amsmathbb{F}^{(n)}(x \vert s) - F_{n,\floor{ns}}(x)}{1-F_{n,\floor{ns}}(x)} \Big\vert = \mathcal{O}_{\amsmathbb{P}}\Big(\frac{1}{\sqrt{k_{n}h_{n}}}\Big).
     \end{align*}
\end{proposition}

In the results above, we require certain rates of the deterministic sequences $(k_{n})$ and $(h_{n})$ that are subject to tuning. The following result verifies that such requirements can be met and allows for quantification of optimal convergence rate bounds. In what follows, we use for positive sequences the notation $a_n  \asymp b_{n}$ when $\lim_{n\to\infty} a_n/b_n  = d$, for some constant $d>0$.
\begin{lemma}\label{lemma_kn_hn_sequences}
    If $k_{n} \asymp n^{\delta}$ for $\delta \in (1/2,1)$ and $h_{n} \asymp n^{-e}$ for $e\in \big((1-\delta)\vee (2-\delta)/5\big), \delta)$, then
    \begin{align*}
        nh_{n}, \ k_{n}h_{n} \to \infty,
        \qquad
        h_{n}, \ k_{n}/n, \ nh_{n}^{5} \to 0,
        \qquad
          \frac{n\log(n)}{k_{n}}(h_{n}^{2} \vee 1/\sqrt{n})
        = o\Big(\frac{1}{\sqrt{k_{n}h_{n}}}\Big).
    \end{align*}
\end{lemma}
Under the assumptions of Lemma~\ref{lemma_kn_hn_sequences} it clearly also holds that $\log(n)h_{n},\log(n)/(k_{n}h_{n}) \to 0$. As the convergence rates of the leading terms in Theorem~\ref{thm_Lambda_consistency_speed} and Proposition~\ref{prop_Beran_consistency_speed} are of order $\mathcal{O}_{\amsmathbb{P}}((\sqrt{k_{n}h_{n}})^{-1})$, the optimal rate of convergence is obtained by choosing $k_{n}$ and $h_{n}$ as large as possible, i.e. $\delta$ close to $1$ (from below) and $e$ close to $1/5$ (from above). This yields the following optimal convergence rate bound
\begin{align*}
    k_{n}h_{n} = o(n^{4/5}),
\end{align*}
which compares to the bound in the iid case where $k_{n} = o(n)$ as in~\citep{STUTE_1992, Csorgo_96}, noting that $h_{n}$ is often picked as $h_{n} = n^{-1/5}/r_{n}$ for some $(r_{n})$ slowly diverging to infinity such as $\log(n)$, which ensures our assumption of $nh_{n}^{5} \to 0$.

\subsection{Consistency of relative scedasis estimator}
Consistency of the relative scedasis estimator follows directly from the uniform convergence of our local Beran estimator. However, we first present the result with an abstract condition.
\begin{theorem}\label{thm_c_F_consistency}
   Suppose that
     \begin{align}\label{eq_local_beran_tail_negl_in_Z}
         \Big\vert \frac{\amsmathbb{F}^{(n)}(Z_{n:n-k_{n}} \vert s) - F_{n,\floor{ns}}(Z_{n:n-k_{n}})}{1-F_{n,\floor{ns}}(Z_{n:n-k_{n}})} \Big\vert = 
         o_{\amsmathbb{P}}(1),
     \end{align}
     as $n\to \infty$, holds for every $s\in (0,1)$. Then, for any $s\in (0,1)$ and $s_{0} \in (0,1)$ it holds that
     \begin{align*}
         \hat{c}_{F}(s \vert s_{0})
         \overset{\amsmathbb{P}}{\to} \frac{c_{F}(s)}{c_{F}(s_{0})}, 
         \qquad \text{as} \ n \to \infty.
     \end{align*}
\end{theorem}
This result is comparable to Remark 2 of~\cite{Einmahl_Haan_Zhou_2014} in the non-censored setting. However, they directly establish asymptotic Gaussianity for their estimator and target $c_{F}(s)$ directly rather than the relative scedasis. In Section~\ref{section_discussion} we discuss normalization of our estimator. As mentioned, the above minimal assumptions are implied by the assumptions of Proposition~\ref{prop_Beran_consistency_speed}. 
\begin{corollary}
\label{cor_relative_scedasis_verified_conditions}
Assume Condition~\ref{cond_H_H1_C1_F_G_Lip} and the rate condition~\eqref{eq_n_log_n_h2_over_kn_little_o}. Then the conclusion of Theorem~\ref{thm_c_F_consistency} holds.
\end{corollary}

\begin{remark}\label{remark_homoscedastic_censoring}
    There is an interesting setting of right-censored data, namely the case of homoscedastic censoring, i.e. $G_{n,i} \equiv G$. In this case the observations $(Z_{i}^{(n)})$ inherit the scedasis function $c_{F}$ from the target observations $(X_{i}^{(n)})$ due to the fact that~\eqref{eq_H_scedasis_dens} simplifies as follows
\begin{align*}
    \lim_{x\to x^{*} } \frac{1-H_{n,i}(x)}{1-H(x)}
    =
     \lim_{x\to x^{*}} \frac{(1-F_{n,i}(x))(1-G(x))}{(1-F(x))(1-G(x))} 
     = 
     \lim_{x\to x^{*}} \frac{1-F_{n,i}(x)}{1-F(x)}
     = 
     c_{F}\Big(\frac{i}{n}\Big).
\end{align*}
This implies that the possibly right-censored observations $Z_{1}^{(n)}, \ldots , Z_{n}^{(n)}$ can be directly plugged into the estimator of~\cite{Einmahl_Haan_Zhou_2014}, $\hat{c}_{F}^{\text{EHZ}}(\cdot)$, which is then consistent and asymptotically Gaussian. This argument holds under homoscedastic censoring. Consequently comparing our estimator to a relative version of $\hat{c}_{F}^{\text{EHZ}}(\cdot)$ yields a diagnostic tool to quantify whether the censoring mechanism is heteroscedastic or not; if the estimators coincide it suggests homoscedastic censoring but if they disagree it suggests heteroscedastic censoring, at least for the chosen bandwidth $h_{n}$ and threshold $k_{n}$. Hence comparing these estimators provides a useful benchmark, and this is what we investigate in the simulation and real data sections. To our knowledge, no other benchmark estimators exist.
\end{remark}
Extending the above consistency results to asymptotic normality is a promising direction but outside the scope of this paper. It requires foundational work extending central limit theorems for non-identically distributed arrays, which can then be applied to the decomposition of Theorem~\ref{thm_Nelson_Aalen_decomposition}.

\section{Simulation study}\label{section_simulation_study}
In this section we study the finite-sample performance of $\hat{c}_{F}(\cdot\vert s_{0})$ under three heteroscedastic extreme-value models. The purpose is twofold: first, to assess the pointwise and pathwise accuracy, and second, to quantify the
loss incurred by applying the estimator of~\cite{Einmahl_Haan_Zhou_2014} directly to heteroscedastic censored observations. To assess the accuracy we compute Monte Carlo averages of the pointwise mean squared errors and pathwise uniform errors.\\
In the simulations below, we consider the following beta densities,
\begin{align*}
    c_{F}^{\circ}(s) = \frac{s^{2}(1-s)^{2}}{B(3,3)},
    \qquad
    c_{G}^{\circ}(s) = \frac{1-s}{B(1,2)},
    \qquad
    c_{H}^{\circ}(s) = c_{F}^{\circ}(s)c_{G}^{\circ}(s),
\end{align*}
where $B(\alpha, \beta)=\Gamma(\alpha)\Gamma(\beta)/\Gamma(\alpha + \beta)$, see Figure~\ref{fig_theoretical_curve_simulated_frechet_survival}. We note that $c_{H}^{\circ}$ does not integrate to one. For $i=n$ we have that $c_{F}^{\circ}(1)=c_{G}^{\circ}(1) = 0$ and hence to avoid degenerate distributions we use $\varepsilon$-regularized densities; $c_{F}(s) := (c_{F}^{\circ}(s) + \varepsilon)/(1+\varepsilon)$, and likewise $c_{G}(s):=(c_{G}^{\circ}(s) + \varepsilon)/(1+\varepsilon)$ for $\varepsilon = 10^{-4}$. We benchmark $\hat{c}_{F}(\cdot\vert s_{0})$ against the relative Einmahl et al. estimator
\begin{align*}
     \hat{c}_{F}^{\mathrm{EHZ}}(s\vert s_{0})
    =
    \hat{c}_{F}^{\mathrm{EHZ}}(s)/
    \hat{c}_{F}^{\mathrm{EHZ}}(s_{0}),
\end{align*}
computed from the observed $Z_{i}^{(n)}$'s. According to Remark~\ref{remark_homoscedastic_censoring} this benchmark is valid under homoscedastic censoring, but not under the heteroscedastic censoring. 
\begin{figure}[hbt!] 
    \centering
    \includegraphics[width=.45\textwidth]{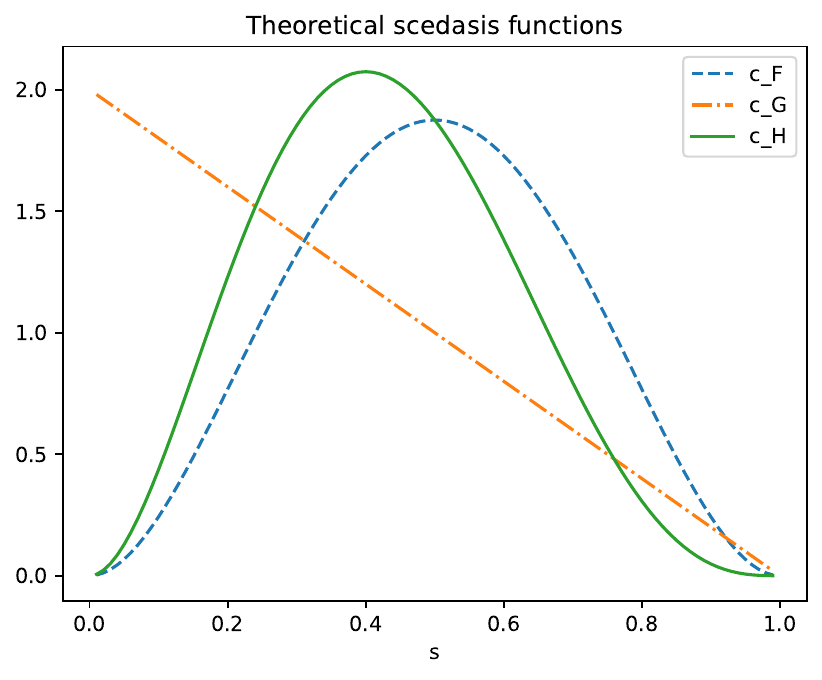} 
    \includegraphics[width=.49\textwidth]{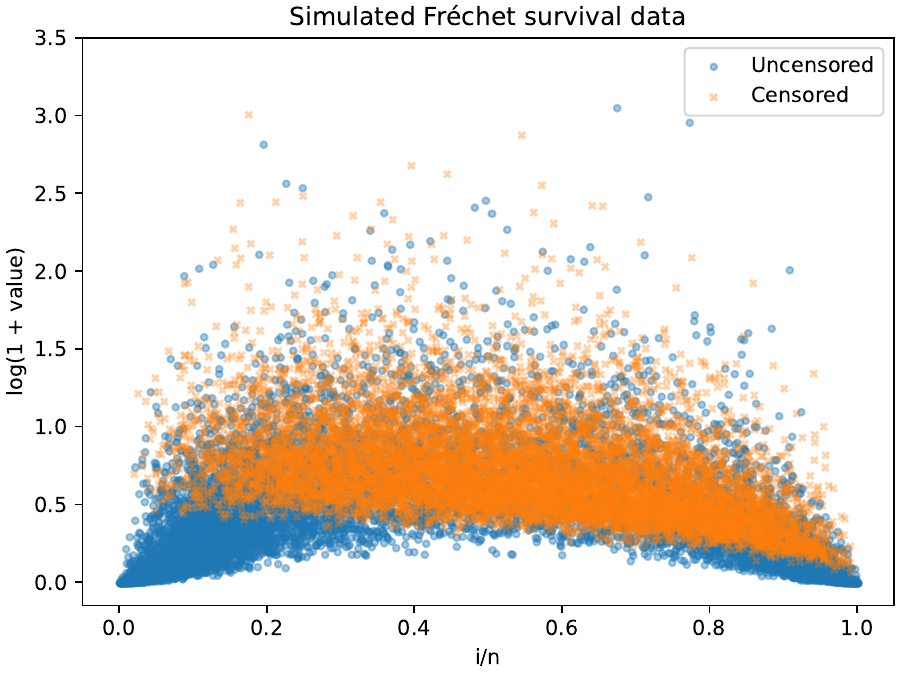} 
    \includegraphics[width=.5\textwidth]{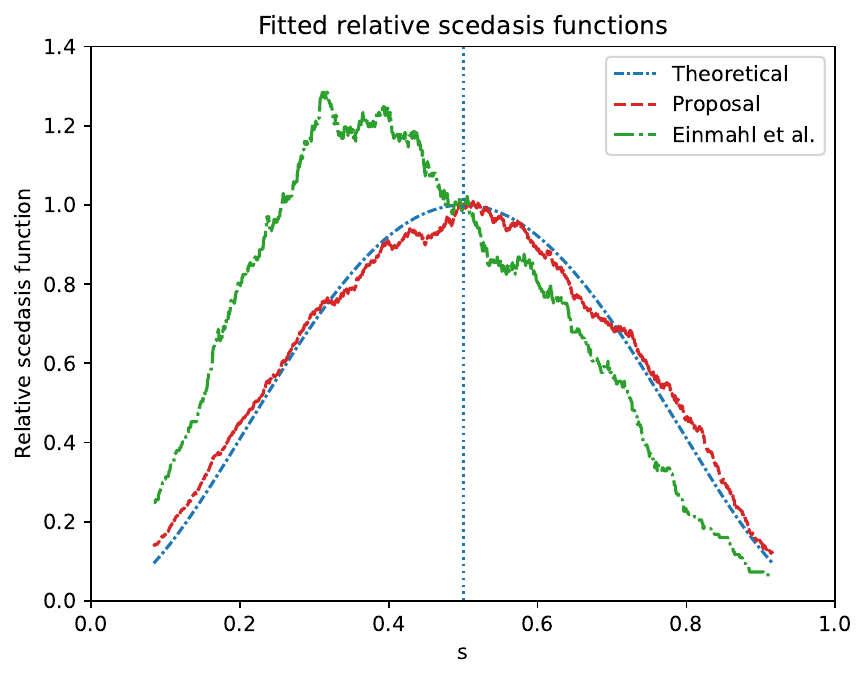} 
    \caption{(Top left) Theoretical scedasis functions $c_{F},c_{G}$ and $c_{H}$ shown as dashed blue, dot-dashed orange and solid green curves, respectively. (Top right) $n=20{,}000$ simulated Fr\'echet survival data points for tail indices $\alpha_{F}=1$ and $\alpha_{G}=2$ on the $\log(1+x)$-scale, where fully observed values $(\delta_{i}^{(n)} = 1)$ are marked with blue dots and censored values $(\delta_{i}^{(n)} = 0)$ are marked with orange x's. (Bottom) Depiction of $\hat{c}_{F}(s \vert s_{0})$ (red dashed line) and $\hat{c}_{F}^{\mathrm{EHZ}}(s\vert s_{0})$ (green dot-dashed line) for $s\in [h_{n},1-h_{n}]$. Here $s_{0} = 1/2$ (marked with a vertical blue dotted line), $k_{n} = \ceil{n/40}$ and $h_{n} = n^{-1/4}$.}
    \label{fig_theoretical_curve_simulated_frechet_survival}
\end{figure}
\begin{figure}[hbt!] 
    \centering
    \includegraphics[width=.82\textwidth]{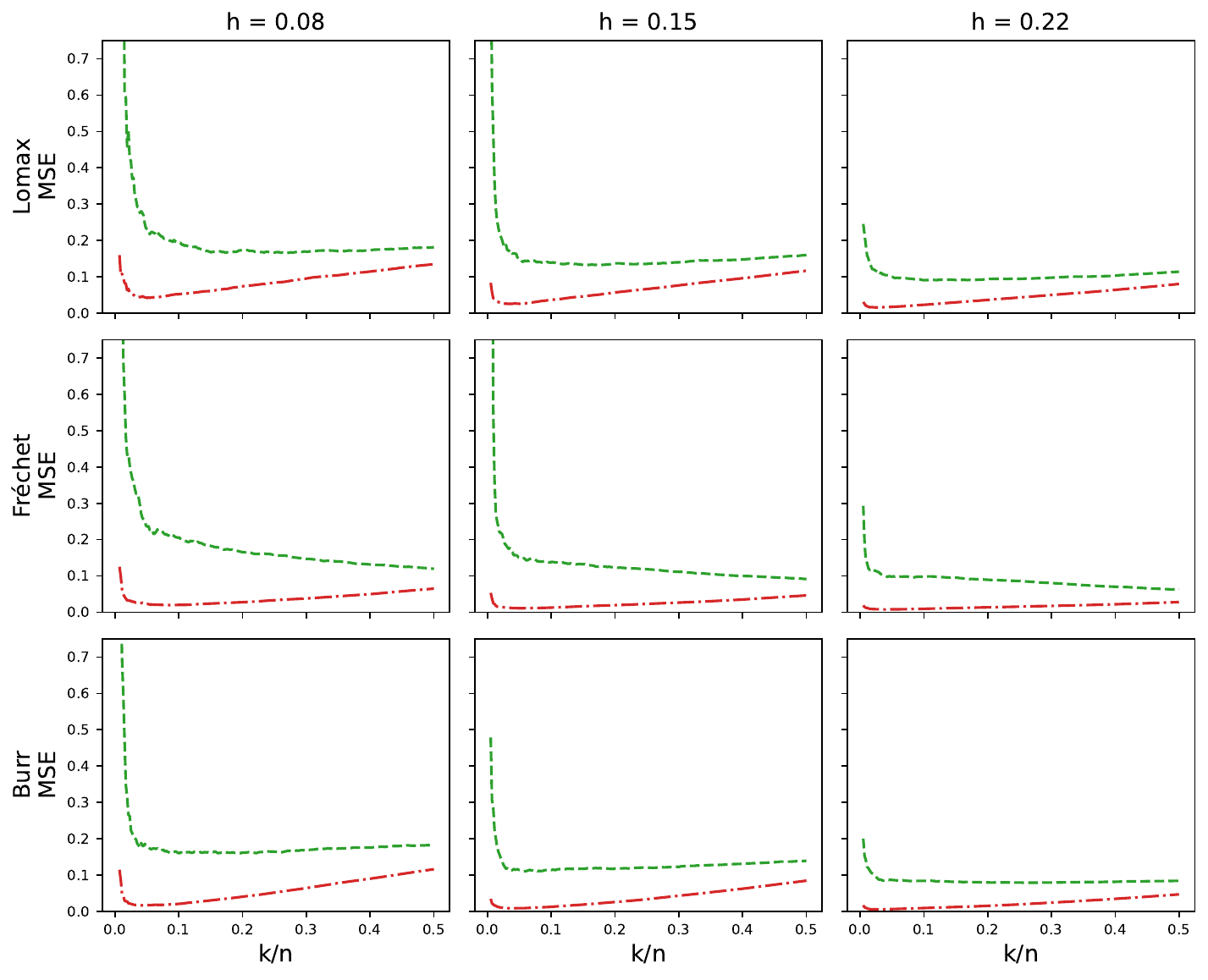}
    \includegraphics[width=.82\textwidth]{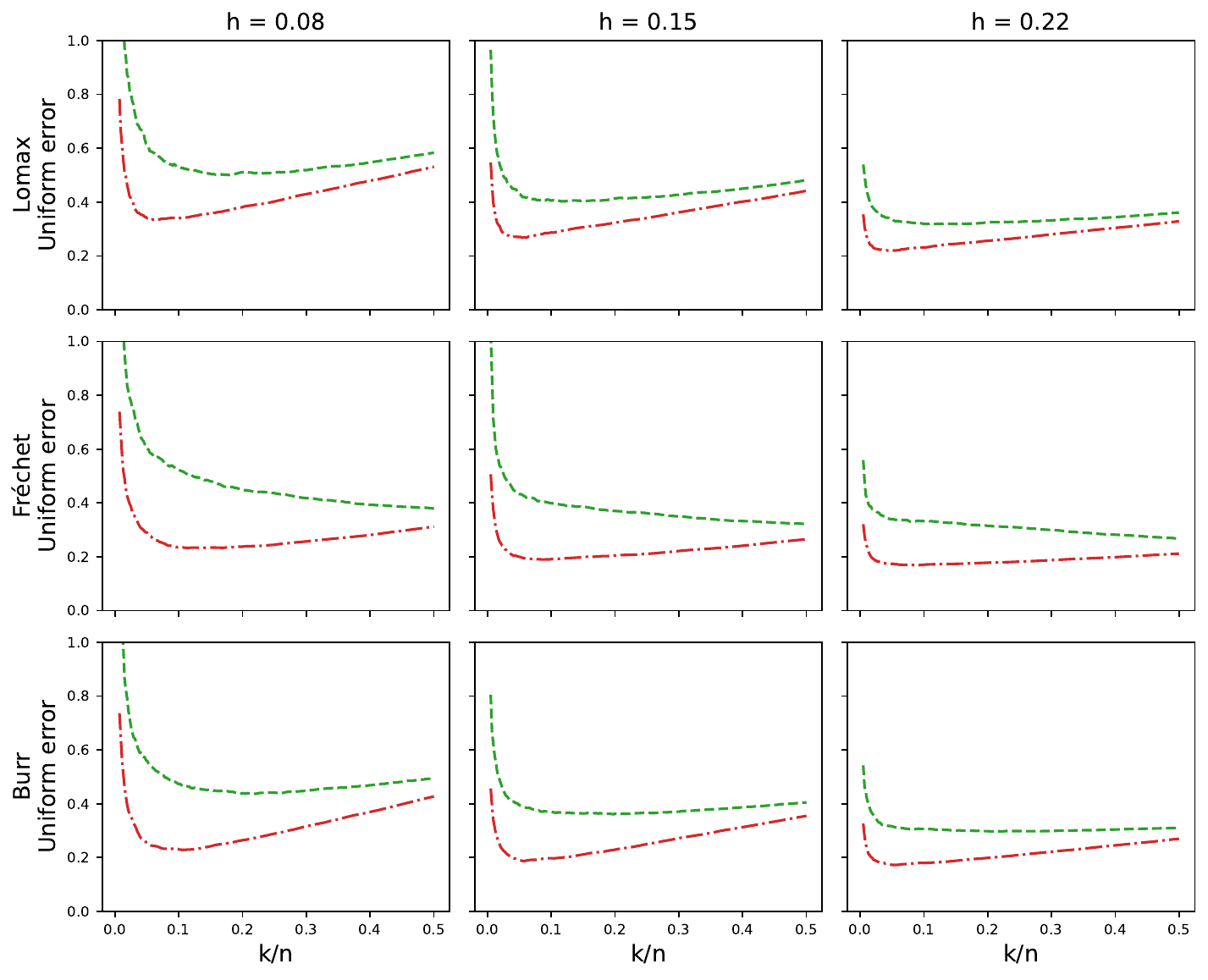}
    \caption{Simulation study: Estimated pointwise mean squared error (top panel) for $s=2/5$, and pathwise uniform error (bottom panel) for $\hat{c}_{F}(\cdot \vert 1/2)$ and $  \hat{c}_{F}^{\mathrm{EHZ}}(\cdot \vert 1/2)$ in red dot-dashed and green dashed lines respectively as functions of $k_{n}/n \in [0.005,1/2]$, where the errors are computed from $B=200$ batches of $n=2{,}000$ simulated observations from heteroscedastic Lomax, Fr\'echet and Burr (in first, second and third row respectively for MSE and fourth, fifth and sixth respectively for uniform error) for bandwidth values $h_{n}\in \{0.08, 0.15, 0.22\}$ (from first to last column respectively) corresponding to the approximate values of $n^{-1/3}, n^{-1/4}, n^{-1/5}$.}
    \label{fig_simulation_mse_unif_error}
\end{figure}

The three parametric distributions from which we draw samples in the simulations are given as follows. Throughout, let $U_{1},\ldots, U_{n}$ and $V_{1},\ldots, V_{n}$ denote iid standard uniform random variables.

\begin{enumerate}
    \item Heteroscedastic Lomax (Pareto II): for $\alpha_{F}=1$ and $\alpha_{G}=2$ 
    \begin{align*}
        X_{i}^{(n)} := c_{F}(i/n)^{1/\alpha_{F}}((1-U_{i})^{-1/\alpha_{F}} -1),
        \qquad
        C_{i}^{(n)} := c_{G}(i/n)^{1/\alpha_{G}}((1-V_{i})^{-1/\alpha_{G}} -1).
    \end{align*}
    \item Heteroscedastic Fr\'echet: for $\alpha_{F}=1$ and $\alpha_{G}=2$,
    \begin{align*}
        X_{i}^{(n)} := (c_{F}(i/n)/\log(1/U_{i}))^{1/\alpha_{F}},
        \qquad
        C_{i}^{(n)} := (c_{G}(i/n)/\log(1/V_{i}))^{1/\alpha_{G}}.
    \end{align*}
    For the above choices of scedasis densities, we depict $20{,}000$ realizations of $(Z_{i}^{(n)},\delta_{i}^{(n)})$ in Figure~\ref{fig_theoretical_curve_simulated_frechet_survival}.
    \item Heteroscedastic Burr: for $a_{F}=a_{G}=2$, $b_{F}=1/2$ and $b_{G}=1$,
    \begin{align*}
        &X_{i}^{(n)} := c_{F}(i/n)^{1/(a_{F}b_{F})}((1-U_{i})^{-1/b_{F}} -1)^{1/a_{F}},\\
        &C_{i}^{(n)} := c_{G}(i/n)^{1/(a_{G}b_{G})}((1-V_{i})^{-1/b_{G}} -1)^{1/a_{G}}.
    \end{align*}
\end{enumerate}
In Appendix~\ref{appendix_simulations} we verify that these constructions admit heteroscedastic distribution functions in the sense of~\eqref{eq_Scedasis_F_and_G} and that they satisfy Condition~\ref{cond_H_H1_C1_F_G_Lip}.

Below we fix $s_{0}=1/2$ and $s=2/5$, draw $B=200$ batches of $n=2{,}000$ observations for each of the distributions, and use the uniform kernel. All losses are plotted for $k_{n}/n \in [0.005,1/2]$ for the bandwidth choices $h_{n}\in \{0.08, 0.15, 0.22\}$ and  corresponding to the approximate values of $n^{-1/3}, n^{-1/4}, n^{-1/5}$ respectively. The pathwise losses are computed over the interior grid $s\in[h_{n},1-h_{n}]$ to retain correct sample sizes. Figure~\ref{fig_simulation_mse_unif_error} reports the pointwise mean squared error and the pathwise uniform error of $\hat{c}_{F}(s \vert s_{0})$.

Across all three simulation designs, the proposed estimator yields substantially smaller losses than the benchmark over the displayed threshold range, as expected. However, we note that the improvement is most pronounced for small and moderate values of $k_{n}/n$, where heterogeneous censoring has the largest effect on the observed tail. For larger bandwidths, the loss curves become smoother and the threshold sensitivity is reduced.
\\
For the pointwise MSE, the proposed estimator attains its smallest values at small threshold levels. The pathwise uniform error exhibits a similar pattern, although its minimum is sometimes shifted slightly toward larger values of $k_{n}/n$. Additional pointwise bias and integrated squared error diagnostics are reported in Appendix~\ref{appendix_simulations}.
\\
Losses are averaged over Monte Carlo replications for which the normalizing denominator of the relative estimators is non-zero. Over the displayed threshold range $k_{n}/n\geq 0.005$, denominator failures are rare; the corresponding maximum failure rates are also reported in Appendix~\ref{appendix_simulations}.

\section{Real data analysis}\label{section_real_data_analysis}
Censored extreme-value methods have recently been applied to non-life insurance claim settlements under a marginal random-censoring framework, see~\cite{BladtGoegebeurGuillou2026}. That analysis provides a benchmark for tail-index estimation under censoring, but treats the observed losses as an exchangeable censored sample, which disregards the clear changes of censoring distribution as one approaches the valuation date. Here we study a different question, under the more realistic assumption of heteroscedastic extremes. Namely, whether the contribution to the extreme tail varies systematically over occurrence time. We therefore retain the temporal ordering of the claims and apply the proposed scedasis estimator to assess temporal tail heterogeneity while accounting for incomplete settlements.

The data\footnote{Available as the \texttt{freclaimset3dam9207} dataset in the \texttt{CASdatasets} R package.}  consist of $109{,}992$ individual commercial property-casualty claim histories from 1992--2007, with occurrence dates and cumulative paid and incurred amounts recorded at daily valuation dates. Since explicit settlement indicators are not available, we treat a claim as fully developed at a given valuation date when the cumulative paid and incurred amounts coincide. For a fixed valuation date and claim $i$, let
\begin{align*}
    P_{i}^{(n)}
    &:=
    \text{cumulative payment of claim } i,
    \qquad
    I_{i}^{(n)}
    :=
    \text{cumulative incurred amount of claim } i.
\end{align*}
We construct the censored observations by setting
\begin{align*}
    Z_{i}^{(n)}
    &:=
    P_{i}^{(n),\mathrm{adj}},
    \qquad
    \delta_{i}^{(n)}
    :=
    1\{P_{i}^{(n)}=I_{i}^{(n)}\},
\end{align*}
where $P_{i}^{(n),\mathrm{adj}}$ denotes the inflation-adjusted cumulative paid amount; see Appendix~\ref{appendix_real_data}. After retaining claims with strictly positive cumulative paid amounts at the final valuation date, the analysis sample contains $n=107{,}728$ observations, of which $2{,}654$ correspond to right-censored data. We plot the corresponding observations in Figure~\ref{fig_scatterplot_and_tail_non_censoring} together with the average proportion of non-censoring among the $k_{n}$ largest order statistics.

\begin{figure}[hbt!] 
    \centering
    \includegraphics[width=.48\textwidth]{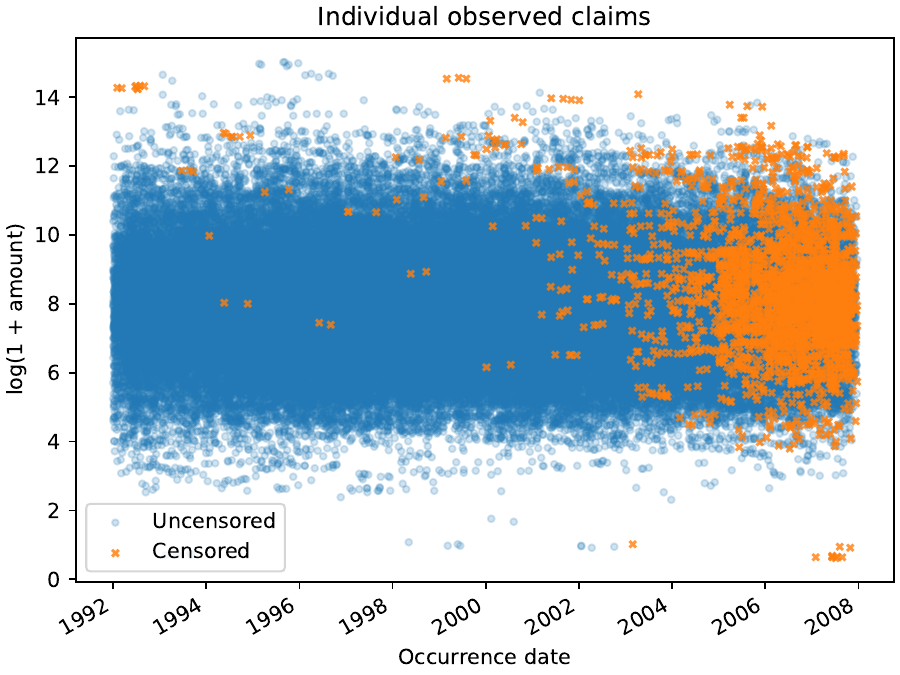}
    \includegraphics[width=.48\textwidth]{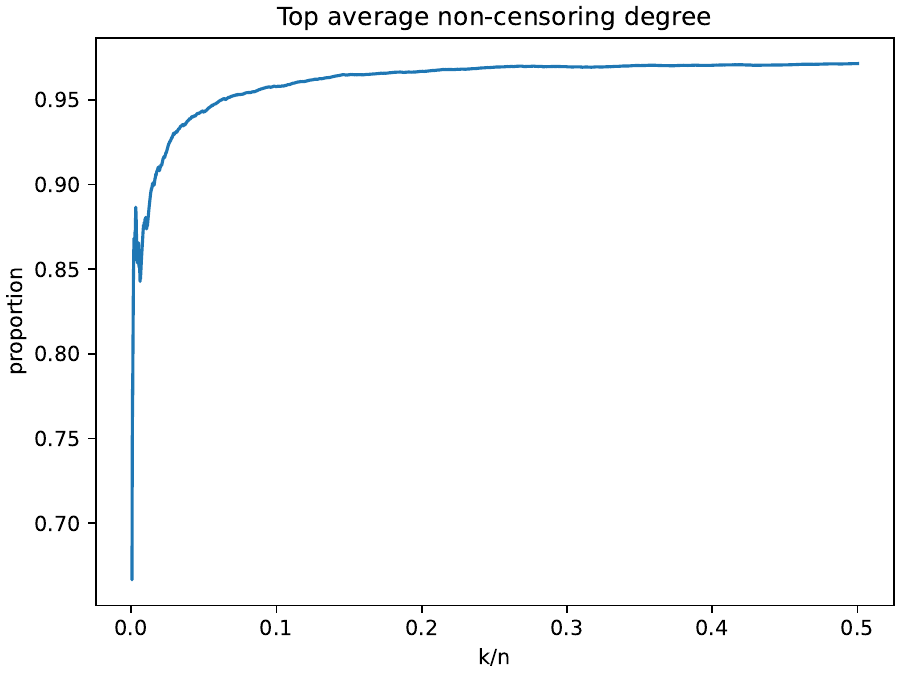}
    \caption{Real data: (Left) Scatter plot of strictly positive, inflation-adjusted claim sizes on the $\log(1+x)$ scale where fully observed claims $(\delta_{i}^{(n)} = 1)$ are marked with blue dots and censored claims $(\delta_{i}^{(n)} = 0)$ are marked with orange x's. (Right) Average degree of non-censoring among the $k_{n}$ largest order statistics as a function of $k_{n}/n \in [0.0005, 0.5]$.}
    \label{fig_scatterplot_and_tail_non_censoring}
\end{figure}
By~\eqref{eq_Scedasis_F_and_G} and~\eqref{eq_H_scedasis_dens},
the model assumes a common, and positive, tail index across occurrence time. We assess
this assumption using censored Hill estimates on temporal subsamples in
Appendix~\ref{appendix_real_data}; the diagnostic does not reveal
large discrepancies over the threshold range used below.
We proceed with our analysis, plotting now the estimated relative scedasis curves.
\begin{figure}[hbt!] 
    \centering
    \includegraphics[width=.48\textwidth]{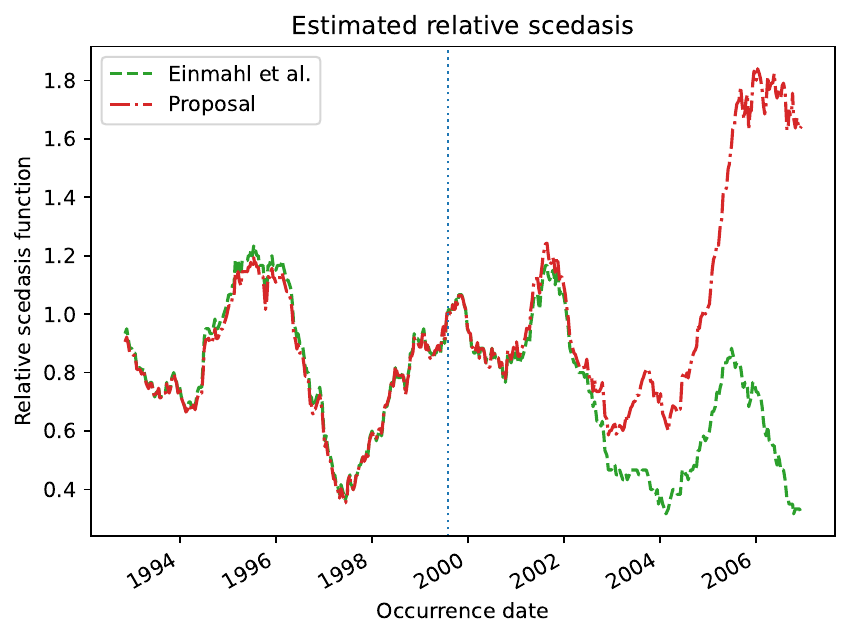}
    \includegraphics[width=.48\textwidth]{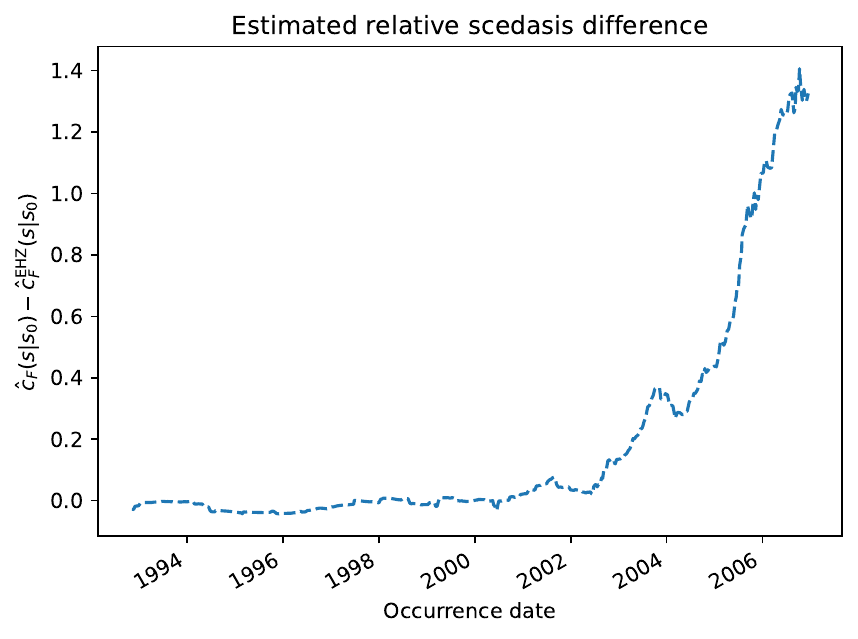}
    \caption{Estimated relative scedasis and corresponding difference. (Left) Estimated scaled scedasis curves from our proposed estimator $\hat{c}_{F}(s \vert s_{0})$ in dot-dashed red compared with $\hat{c}_{F}^{\text{EHZ}}(s\vert s_{0})$ in dashed green. Here $s_{0} = 1/2$ is shown in the vertical blue line on the left, $k_{n} = \floor{n/200}$ and $h_{n} = 2n^{-1/3}$. (Right) Difference between the two estimators $ \hat{c}_{F}(s\vert s_{0}) - \hat{c}_{F}^{\mathrm{EHZ}}(s\vert s_{0})$.}
    \label{fig_freclaim_estimated_scedasis_curves_and_difference}
\end{figure}
In Figure~\ref{fig_freclaim_estimated_scedasis_curves_and_difference} we see that our estimator coincides with the scaled version of the~\cite{Einmahl_Haan_Zhou_2014} estimator in the regions where there is no or little censoring, while in the areas with a high degree of censoring, our estimator is substantially larger than the benchmark. 
At a fixed valuation date, claims with more recent occurrence dates have had less time to develop. Hence the probability that $P_{i}^{(n)}=I_{i}^{(n)}$ is naturally occurrence-time dependent, making homoscedastic censoring structurally impossible in this application. In this case, $\hat{c}_{F}^{\text{EHZ}}(s\vert s_{0})$ is expected to be downward biased because censored observations are treated as fully observed.

To assess the stability of the estimated scedasis curve, we construct a sequence of retrospective valuation datasets. For $j\in\{0,2,5,8\}$, let the valuation year be $2007-j$. We retain only claims with occurrence dates before this valuation year and use the corresponding cumulative payments observed at that valuation date. Thus, claims occurring in the final $j$ years of the full observation window are removed, and the censored observations are reconstructed from the paid/incurred information available at year $2007-j$. More precisely, for each $j\in\{0,2,5,8\}$, we form a sample $(Z_{i}^{(n_{j},j)},\delta_{i}^{(n_{j},j)})_{i=1}^{n_{j}}$, ordered by occurrence date, by setting 
\begin{align*}
    Z_{i}^{(n_{j},j)}
    :=
    P_{i,j}^{\mathrm{adj}}, 
    \qquad
    \delta_{i}^{(n_{j},j)}
    :=
    1\{P_{i,j}=I_{i,j}\},
\end{align*}
where $P_{i,j}$ and $I_{i,j}$ denote the cumulative paid and incurred amounts at valuation year $2007-j$, respectively. The case $j=0$ corresponds to the full dataset at the final valuation date, while the cases $j=2,5,8$ mimic the information that would have been available at earlier valuation horizons. The numbers of retained observations and censored claims are reported for each dataset in Appendix~\ref{appendix_real_data}.

We then apply the same scedasis estimators to each reconstructed sample. If the proposed censoring adjustment is effective, the resulting curves should be broadly stable across valuation horizons on the common occurrence-time region, whereas an estimator that does not account for heterogeneous censoring might show systematic distortion in the more recent, heavily censored part of each sample.

\begin{figure}[hbt!] 
    \centering
    \includegraphics[width=.48\textwidth]{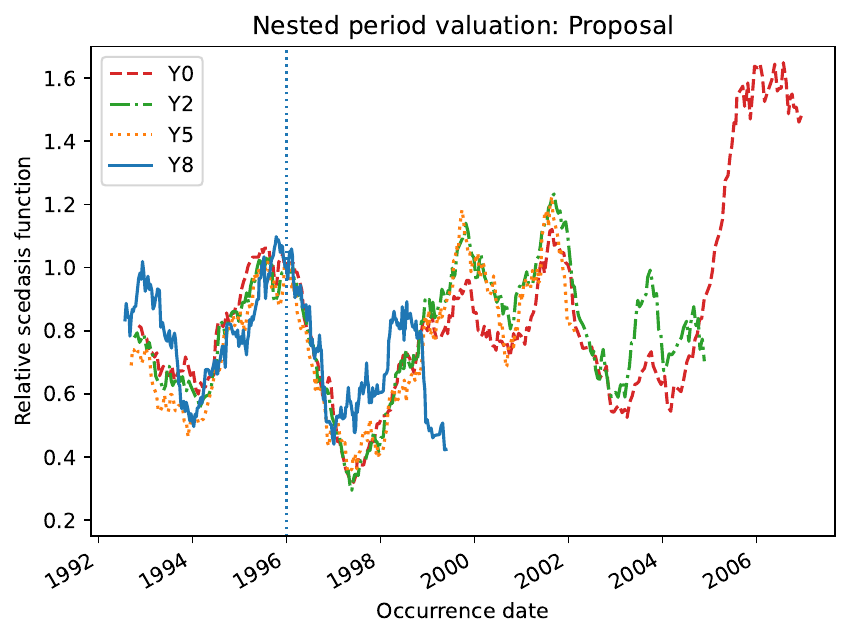}
    \includegraphics[width=.48\textwidth]{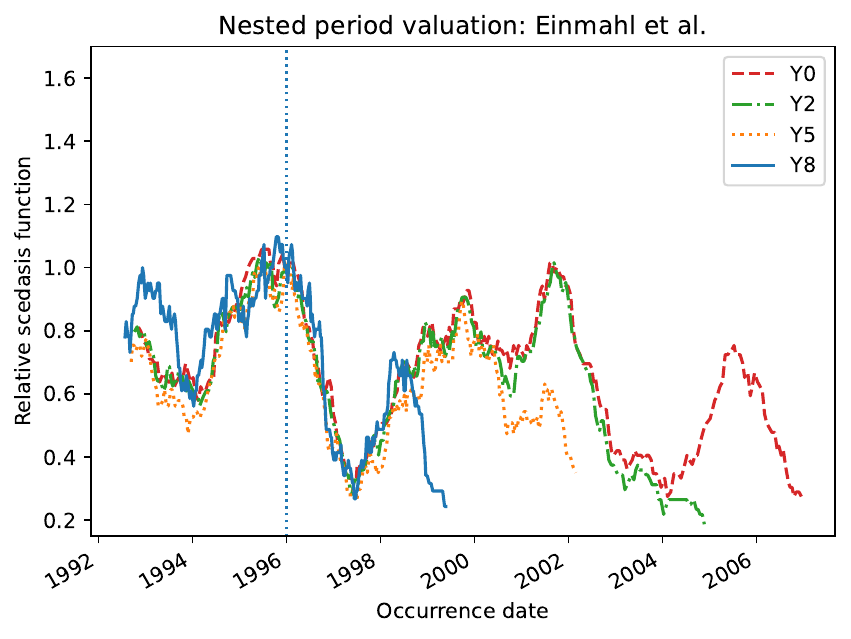}
    \caption{
    Retrospective valuation diagnostic. For $j\in\{0,2,5,8\}$, the dataset $\texttt{Yj}$ is constructed by retaining claims with occurrence dates no later than valuation year $2007-j$, and by reconstructing $(Z_{i}^{(n_{j},j)},\delta_{i}^{(n_{j},j)})$ from the cumulative paid and incurred amounts available at that valuation date. Thus \texttt{Y0} corresponds to the full dataset at the final valuation date, while \texttt{Y2}, \texttt{Y5}, and \texttt{Y8} mimic earlier valuation horizons. The figure compares the resulting relative scedasis estimates using the proposed estimator (left) and the Einmahl et al. benchmark (right). Here $s_{0}$ is set to January 1, 1996 and marked with a vertical dashed line in both panels, and $k_{n} = \floor{n/200}$ and $h_{n} = 2n^{-1/3}$, where $n$ denotes the sample size in each of the subsets, see Table~\ref{tab_real_data_valuation_sample_sizes}.
    }
    \label{fig_freclaim_scedasis_valuation_plots}
\end{figure}

In Figure~\ref{fig_freclaim_scedasis_valuation_plots}, the estimated relative scedasis curves produced by the proposed estimator broadly agree across the successive curves. Their shapes and levels are similar, which further indicates that the estimator is reasonable for the concrete dataset at hand, and robust to the choice of valuation date. However, for the first and shortest period there seems to be a negative bias in the right part where censoring is most severe; upon further investigation, it turns out that the estimated tail index of this period further we see that this period has a significantly higher tail index than the rest of the periods, suggesting weaker compatibility with the common tail index assumption; see Figure~\ref{fig_freclaim_tail_valuation_plots} in Appendix~\ref{appendix_real_data}. As for the benchmark, we consistently see a negative bias in the right part of the curves where censoring is most severe. This supports the conclusion that, for this specific insurance dataset, ignoring temporally heterogeneous censoring leads to downward distortion in the estimated scedasis curve.
\FloatBarrier

\section{Discussion of estimators}\label{section_discussion}
In this final section we provide clarifying remarks on the estimators used in the paper. Specifically, we argue in the first subsection why the cumulative scedasis is not estimable directly in the presence of random right-censoring using Nelson--Aalen-type estimators. In the second subsection we discuss the normalization of the relative scedasis estimator and propose an alternative estimator to obtain the correct normalization.

\subsection{Why not estimate the cumulative scedasis directly?}\label{subsection_discussion_cumul_c_estimator}
In the uncensored setting,~\cite{Einmahl_Haan_Zhou_2014} study the integrated scedasis function
\begin{align*}
    C_{F}(s) := \int_{0}^{s}c_{F}(u)\mathrm{d} u, \quad s \in [0,1],
\end{align*}
which they estimate by
\begin{align*}
    \widehat{C}(s) :&= 
    \frac{1}{k_{n}}\sum_{i=1}^{\floor{ns}}1\{X_{i}^{(n)} > X_{n:n-k_{n}}\}
    =
     \frac{\floor{ns}}{k_{n}}(1-\amsmathbb{F}^{(\floor{ns})}(X_{n:n-k_{n}})),
\end{align*}
where $\amsmathbb{F}^{(\floor{ns})}$ denotes the empirical distribution function based on $(X_{1}^{(n)},\ldots,X_{\floor{ns}}^{(n)})$. This estimator is particularly useful because the process $s\mapsto \widehat C(s)$ admits functional limit theorems.
\\
A seemingly natural extension to censored observations would be to replace the empirical distribution function with the \emph{prefix} Kaplan--Meier estimator, based on the first $\floor{ns}$ samples, while also replacing $\floor{ns}/k_{n}$ by $s/(1-\amsmathbb{F}^{(n)}(Z_{n:n-k_{n}}))$ (disregarding the small rounding error). This amounts to 
\begin{align*}
     \widehat{C}_{\text{KM}}(s):=
     s\frac{1-\amsmathbb{F}^{(\floor{ns})}(Z_{n:n-k_{n}})}{1-\amsmathbb{F}^{(n)}(Z_{n:n-k_{n}})},
\end{align*}
where $\amsmathbb{F}^{(\floor{ns})}$ denotes the prefix Kaplan--Meier estimator constructed from $(Z_{i}^{(n)},\delta_{i}^{(n)})_{i=1}^{\floor{ns}}$, given by
\begin{align*}
    \mathbb{\Lambda}^{(\floor{ns})}(x) &= 
    \int^{x}_{0} \frac{1}{1-\floor{ns}^{-1}\sum_{i=1}^{\floor{ns}}1\{Z_{i}^{(n)} < y\}} \frac{1}{\floor{ns}}\sum_{i=1}^{\floor{ns}} 1\{Z_{i}^{(n)}\leq \mathrm{d} y, \delta_{i}^{(n)} = 1\}, \\
    1 -\amsmathbb{F}^{(\floor{ns})}(x) &= \Prodi_{0 \leq y \leq x}(1-\mathbb{\Lambda}^{\floor{ns}}(\mathrm{d} y)).
\end{align*}
The building blocks in the prefix Nelson--Aalen estimator are sequential empirical processes (see~\cite{Kiefer1961}), partially resembling the structure of the above $\widehat{C}(s)$ which can be identified as a so-called sequential tail empirical process.

The main mathematical difficulty with this construction, as is argued below, is beyond an additional stochastic error caused by censoring. Indeed, under heteroscedastic censoring, pooling the observations in a prefix changes the population object underlying the Nelson--Aalen estimator. Computing the expected value of the numerator and denominator of $\mathbb{\Lambda}^{(\floor{ns})}(x)$ provides the following natural deterministic hazard:
\begin{align*}
        \int_{0}^{x} \frac{1}
        {\floor{ns}^{-1}\sum_{i=1}^{\floor{ns}}(1 - G_{n,i}(y-))(1 - F_{n,i}(y-))}
        \frac{1}{\floor{ns}}\sum_{j=1}^{\floor{ns}}(1 - G_{n,j}(y-)) F_{n,j}(\mathrm{d} y).
    \end{align*}
    A crucial difficulty is that the distribution functions of the censoring mechanisms in the above expression do not cancel out, leading to a censoring-weighted hazard. By contrast, the cumulative hazard of the target mixture distribution of $X_{1}^{(n)},\ldots,X_{\floor{ns}}^{(n)}$ is
    \begin{align*}
        \int_{0}^{x} \frac{1}
        {\floor{ns}^{-1}\sum_{i=1}^{\floor{ns}}(1 - F_{n,i}(y-))} 
       \frac{1}{\floor{ns}}\sum_{j=1}^{\floor{ns}}
       F_{n,j}(\mathrm{d} y),
    \end{align*}
These two quantities coincide when censoring is homoscedastic, i.e. $G_{n,i} \equiv G$, because the common censoring survival factor cancels from numerator and denominator. Under heteroscedastic censoring, however, the factors $(1-G_{n,i}(y-))$ bias the population target. 
In Figure~\ref{fig_distortion} we show in a simple numeric study that $\mathbb{\Lambda}^{\floor{ns}}$ estimates the (biased) censoring-weighted hazard and consequently cannot generally recover the target prefix hazard. This illustrates that, in full generality, $\widehat{C}_{\text{KM}}$ does not recover $C_{F}$.
\begin{figure}[hbt!] 
    \centering
    \includegraphics[width=.48\textwidth]{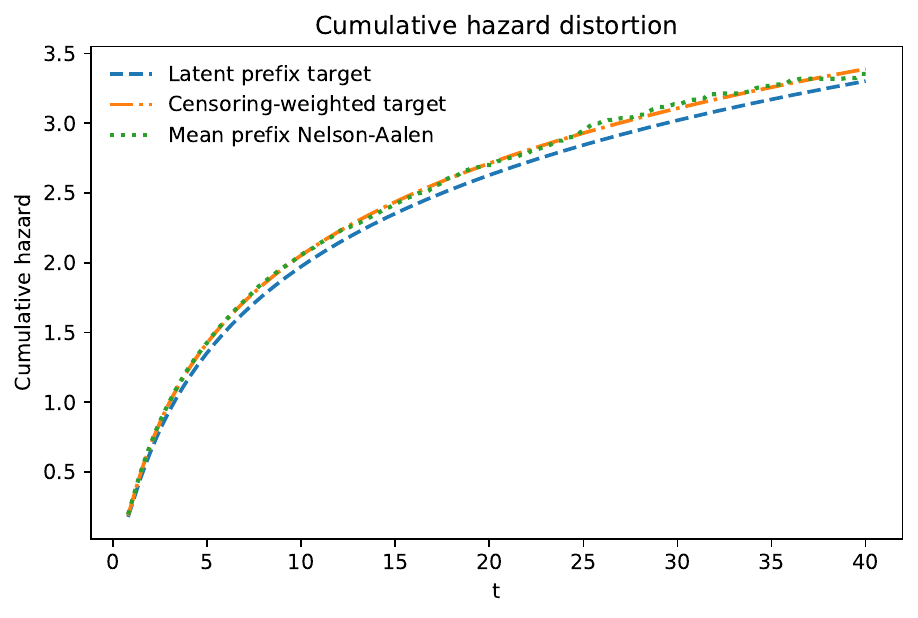}
    \includegraphics[width=.48\textwidth]{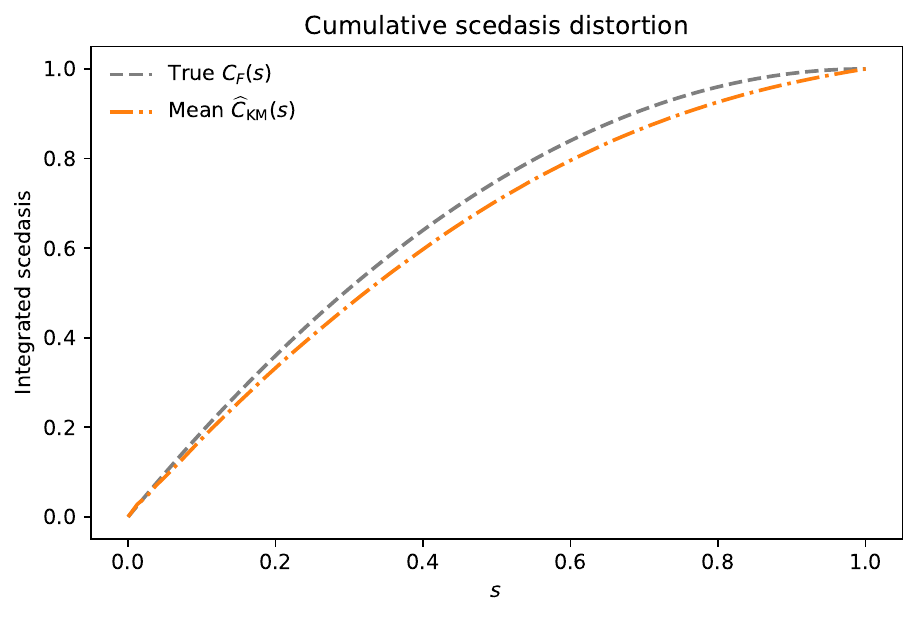}
    \caption{Target distortion of $\mathbb{\Lambda}^{\floor{ns}}$ and $\widehat{C}_{\text{KM}}$. We draw $N=50$ batches of $n=100{,}000$ heteroscedastic Fr\'echet observations and (left) plot the mean $\mathbb{\Lambda}^{\floor{ns}}$ compared with the target prefix target and the censoring-weighted target for $t\leq 40$ and fixed $s=1/2$. In the right panel, we depict the mean $\widehat{C}_{\text{KM}}$ compared with the true $C_{F}(s)$ for $s\in [0,1]$ and $k_{n} = \floor{n/200}$. As above  $\alpha_{F}=1$ and $\alpha_{G}=2$ but now $c_{F}$ is the Beta$(1,2)$ density and $c_{G}$ is the Beta$(3,3)$ density.}
    \label{fig_distortion}
\end{figure}

\subsection{Normalization and boundary correction}
The estimator in Theorem~\ref{thm_c_F_consistency} identifies the
relative scedasis function. For fixed interior points $s,s_{0}\in(0,1)$, we have that $\hat{c}_{F}(s\vert s_{0})$ recovers $c_{F}(s)/c_{F}(s_{0})$, hence the shape of $c_{F}$, but not its normalizing constant.

Since $c_{F}$ integrates to one, a natural scale correction is obtained by normalizing over a uniform grid of size $n-1$. A direct normalization based on the local product-limit estimators is, however, affected by boundary terms. Indeed, although the pointwise arguments used in the proofs below can be extended to interior grids, the grid also contains points within distance $h_{n}$ of the endpoints $\{0,1\}$. At such points the effective kernel support is truncated and the population analogue of the local estimator differs from the interior one.

The boundary effect can be isolated to the normalizing constant, and hence we propose keeping the original local estimator in the numerator and using boundary-corrected local product-limit estimators (see~\cite{Jones1993}) only in the denominator. Denote by $J$ another symmetric kernel that is compactly supported on $[-1,1]$. For $b \in [0,1]$, define 
\begin{align*}
    \mu_{r}^{L}(b)
    &:=
    \int_{-1}^{b}u^{r}J(u)\,\mathrm{d} u,
    \qquad
    \mu_{r}^{R}(b)
    :=
    \int_{-b}^{1}u^{r}J(u)\,\mathrm{d} u,
    \qquad r=0,1,2.
\end{align*}
and the corresponding left and right boundary kernels
\begin{align*}
    J_{b}^{L}(u)
    &:=
    \frac{
        \mu_{2}^{L}(b)-u\mu_{1}^{L}(b)
    }{
        \mu_{0}^{L}(b)\mu_{2}^{L}(b)-\{\mu_{1}^{L}(b)\}^{2}
    }
    J(u)1\{-1\leq u\leq b\}, \\
    J_{b}^{R}(u)
    &:=
    \frac{
        \mu_{2}^{R}(b)-u\mu_{1}^{R}(b)
    }{
        \mu_{0}^{R}(b)\mu_{2}^{R}(b)-\{\mu_{1}^{R}(b)\}^{2}
    }
    J(u)1\{-b\leq u\leq 1\}.
\end{align*}
For $u \in [-1,1]$ and $s \in (0,1)$ consider the piecewise kernel
\begin{align*}
    J_{n,s}(u)&:= 
    J_{s/h_{n}}^{L}(u)1\{s < h_{n}\}
    +
    J(u)1\{h_{n}\leq s \leq 1-h_{n}\}
    +
    J_{(1-s)/h_{n}}^{R}(u)1\{1-h_{n} < s\}.
\end{align*} 
Let $\amsmathbb{F}^{(n),b}(\cdot\vert s)$ denote the above Beran estimator obtained by replacing $K$ with $J_{n,s}$
\begin{align*}
     \hat{c}_{F}^{b}(s) 
    &= \frac{1-\amsmathbb{F}^{(n)}(Z_{n:n-k_{n}} \vert s)}
    {(n-1)^{-1} \sum_{i=1}^{n-1} 1-\amsmathbb{F}^{(n),b}(Z_{n:n-k_{n}} \vert i/n)},
    \qquad 
    s \in (0,1).
\end{align*}
This estimator is a censoring-adjusted analogue of $\hat{c}$ on page 36 in~\cite{Einmahl_Haan_Zhou_2014}. Establishing consistency of $\hat{c}_{F}^{b}(s)$ requires a grid-uniform extension of the local product-limit theory together with control of the boundary-corrected denominator; we leave this as an open problem.

\section*{Acknowledgments}
The authors were supported by the Carlsberg Foundation, grant CF23-1096. The authors would like to thank Jan Beirlant for early discussions regarding inference of censored heteroscedastic extremes.

\newpage
\bibliography{bibliography}

@article{Einmahl_Haan_Zhou_2014,
  author  = {Einmahl, John H. J. and de Haan, Laurens and Zhou, Chen},
  title   = {Statistics of heteroscedastic extremes},
  journal = {Journal of the Royal Statistical Society Series B: Statistical Methodology},
  year    = {2016},
  volume  = {78},
  number  = {1},
  pages   = {31--51},
  doi     = {10.1111/rssb.12099}
}

@article{U_statistics_Stute_94,
  author  = {Stute, Winfried},
  title   = {{$U$-statistic processes: A martingale approach}},
  journal = {The Annals of Probability},
  year    = {1994},
  volume  = {22},
  number  = {4},
  pages   = {1725--1744},
  doi     = {10.1214/aop/1176988480}
}

@article{Csorgo_96,
  author  = {Cs{\"o}rg{\H{o}}, S{\'a}ndor},
  title   = {Universal {Gaussian} approximations under random censorship},
  journal = {The Annals of Statistics},
  year    = {1996},
  volume  = {24},
  number  = {6},
  pages   = {2744--2778},
  doi     = {10.1214/aos/1032181178}
}

@article{STUTE_1992,
  author  = {Stute, Winfried},
  title   = {Strong and weak representations of cumulative hazard function and {Kaplan--Meier} estimators on increasing sets},
  journal = {Journal of Statistical Planning and Inference},
  year    = {1994},
  volume  = {42},
  number  = {3},
  pages   = {315--329},
  doi     = {10.1016/0378-3758(94)00032-8}
}

@article{Einmahl_Villetard,
  author  = {Einmahl, John H. J. and Fils-Villetard, Am{\'e}lie and Guillou, Armelle},
  title   = {Statistics of extremes under random censoring},
  journal = {Bernoulli},
  year    = {2008},
  volume  = {14},
  number  = {1},
  pages   = {207--227},
  doi     = {10.3150/07-BEJ104}
}

@article{Breslow,
  author  = {Breslow, N. and Crowley, J.},
  title   = {A large sample study of the life table and product limit estimates under random censorship},
  journal = {The Annals of Statistics},
  year    = {1974},
  volume  = {2},
  number  = {3},
  pages   = {437--453},
  doi     = {10.1214/aos/1176342705}
}

@book{vanderVaart2023,
  author    = {van der Vaart, A. W. and Wellner, Jon A.},
  title     = {Weak Convergence and Empirical Processes: With Applications to Statistics},
  edition   = {2nd},
  series    = {Springer Series in Statistics},
  publisher = {Springer},
  address   = {Cham},
  year      = {2023},
  isbn      = {978-3-031-29038-1},
  doi       = {10.1007/978-3-031-29040-4}
}

@article{Cairoli_Walsh,
  author  = {Cairoli, R. and Walsh, John B.},
  title   = {Stochastic integrals in the plane},
  journal = {Acta Mathematica},
  year    = {1975},
  volume  = {134},
  pages   = {111--183},
  doi     = {10.1007/BF02392100}
}

@book{ChowTeicher1997,
  author    = {Chow, Yuan Shih and Teicher, Henry},
  title     = {Probability Theory: Independence, Interchangeability, Martingales},
  edition   = {3rd},
  series    = {Springer Texts in Statistics},
  publisher = {Springer},
  address   = {New York},
  year      = {1997},
  isbn      = {978-0-387-98228-1},
  doi       = {10.1007/978-1-4612-1950-7}
}

@techreport{beran1981nonparametric,
  author      = {Beran, Rudolf},
  title       = {Nonparametric Regression with Randomly Censored Survival Data},
  institution = {University of California, Berkeley},
  address     = {Berkeley, CA},
  year        = {1981},
  type        = {Technical Report},
  note        = {Unpublished manuscript}
}

@article{dabrowska1987nonparametric,
  author  = {Dabrowska, Dorota M.},
  title   = {Nonparametric regression with censored survival time data},
  journal = {Scandinavian Journal of Statistics},
  year    = {1987},
  volume  = {14},
  number  = {3},
  pages   = {181--197}
}

@article{mcneil1997,
  author  = {McNeil, Alexander J.},
  title   = {Estimating the tails of loss severity distributions using extreme value theory},
  journal = {ASTIN Bulletin},
  year    = {1997},
  volume  = {27},
  number  = {1},
  pages   = {117--137},
  doi     = {10.2143/AST.27.1.563210}
}

@book{embrechts1997,
  author    = {Embrechts, Paul and Kl{\"u}ppelberg, Claudia and Mikosch, Thomas},
  title     = {Modelling Extremal Events: For Insurance and Finance},
  series    = {Stochastic Modelling and Applied Probability},
  publisher = {Springer},
  address   = {Berlin},
  year      = {1997},
  doi       = {10.1007/978-3-642-33483-2}
}

@article{bornhuetter1972,
  author  = {Bornhuetter, Ronald L. and Ferguson, Ronald E.},
  title   = {The actuary and {IBNR}},
  journal = {Proceedings of the Casualty Actuarial Society},
  year    = {1972},
  volume  = {59},
  pages   = {181--195}
}

@article{mack1994stochastic,
  author  = {Mack, Thomas},
  title   = {Which stochastic model is underlying the {Chain Ladder} method?},
  journal = {Insurance: Mathematics and Economics},
  year    = {1994},
  volume  = {15},
  number  = {2--3},
  pages   = {133--138},
  doi     = {10.1016/0167-6687(94)90789-7}
}

@article{arjas1989claims,
  author  = {Arjas, Elja},
  title   = {The claims reserving problem in non-life insurance: Some structural ideas},
  journal = {ASTIN Bulletin},
  year    = {1989},
  volume  = {19},
  number  = {2},
  pages   = {139--152},
  doi     = {10.2143/AST.19.2.2014905}
}

@article{norberg1993prediction,
  author  = {Norberg, Ragnar},
  title   = {Prediction of outstanding liabilities in non-life insurance},
  journal = {ASTIN Bulletin},
  year    = {1993},
  volume  = {23},
  number  = {1},
  pages   = {95--115},
  doi     = {10.2143/AST.23.1.2005103}
}

@misc{inseeICC,
  author       = {{Insee}},
  title        = {Indice du co{\^u}t de la construction des immeubles {\`a} usage d'habitation ({ICC})},
  howpublished = {S{\'e}rie chronologique, identifiant 000008630},
  year         = {2026},
  note         = {Institut national de la statistique et des {\'e}tudes {\'e}conomiques. Accessed June 24, 2026}
}

@article{BladtGoegebeurGuillou2026,
  author  = {Bladt, Martin and Goegebeur, Yuri and Guillou, Armelle},
  title   = {Asymptotically unbiased estimation of the extreme value index under random censoring},
  journal = {Insurance: Mathematics and Economics},
  year    = {2026},
  volume  = {127},
  pages   = {103225},
  doi     = {10.1016/j.insmatheco.2026.103225}
}

@article{Wuthrich03072018,
  author  = {W{\"u}thrich, Mario V.},
  title   = {Machine learning in individual claims reserving},
  journal = {Scandinavian Actuarial Journal},
  year    = {2018},
  volume  = {2018},
  number  = {6},
  pages   = {465--480},
  doi     = {10.1080/03461238.2018.1428681}
}

@article{deHaan2015TailTrend,
  author  = {de Haan, Laurens and Klein Tank, Albert and Neves, Cl{\'a}udia},
  title   = {On tail trend detection: Modeling relative risk},
  journal = {Extremes},
  year    = {2015},
  volume  = {18},
  pages   = {141--178},
  doi     = {10.1007/s10687-014-0207-8}
}

@book{AndersenBorganGillKeiding1993,
  author    = {Andersen, Per Kragh and Borgan, {\O}rnulf and Gill, Richard D. and Keiding, Niels},
  title     = {Statistical Models Based on Counting Processes},
  series    = {Springer Series in Statistics},
  publisher = {Springer-Verlag},
  address   = {New York},
  year      = {1993},
  isbn      = {978-0-387-97872-7},
  doi       = {10.1007/978-1-4612-4348-9}
}

@article{Stute1995,
  author  = {Stute, Winfried},
  title   = {The central limit theorem under random censorship},
  journal = {The Annals of Statistics},
  year    = {1995},
  volume  = {23},
  number  = {2},
  pages   = {422--439},
  doi     = {10.1214/aos/1176324528}
}

@article{Nadaraya1964,
  author  = {Nadaraya, E. A.},
  title   = {On estimating regression},
  journal = {Theory of Probability and Its Applications},
  year    = {1964},
  volume  = {9},
  number  = {1},
  pages   = {141--142},
  doi     = {10.1137/1109020}
}

@article{Watson1964,
  author  = {Watson, Geoffrey S.},
  title   = {Smooth regression analysis},
  journal = {Sankhy{\=a}: The Indian Journal of Statistics, Series A},
  year    = {1964},
  volume  = {26},
  number  = {4},
  pages   = {359--372}
}

@article{wellner1978limit,
  author  = {Wellner, Jon A.},
  title   = {Limit theorems for the ratio of the empirical distribution function to the true distribution function},
  journal = {Zeitschrift f{\"u}r Wahrscheinlichkeitstheorie und Verwandte Gebiete},
  year    = {1978},
  volume  = {45},
  number  = {1},
  pages   = {73--88},
  doi     = {10.1007/BF00635964}
}

@incollection{csaki1975some,
  author    = {Cs{\'a}ki, Endre},
  title     = {Some notes on the law of the iterated logarithm for empirical distribution function},
  booktitle = {Limit Theorems of Probability Theory},
  editor    = {R{\'e}v{\'e}sz, P{\'a}l},
  series    = {Colloquia Mathematica Societatis J{\'a}nos Bolyai},
  volume    = {11},
  pages     = {47--57},
  publisher = {North-Holland},
  address   = {Amsterdam},
  year      = {1975},
  isbn      = {0-7204-2834-3}
}

@article{Kiefer1961,
  author  = {Kiefer, Jack Carl},
  title   = {On large deviations of the empiric {D.F.} of vector chance variables and a law of the iterated logarithm},
  journal = {Pacific Journal of Mathematics},
  year    = {1961},
  volume  = {11},
  number  = {2},
  pages   = {649--660},
  doi     = {10.2140/pjm.1961.11.649}
}

@article{Jones1993,
  author  = {Jones, M. C.},
  title   = {Simple boundary correction for kernel density estimation},
  journal = {Statistics and Computing},
  year    = {1993},
  volume  = {3},
  pages   = {135--146},
  doi     = {10.1007/BF00147776}
}

@article{beirlant2007estimation,
  author  = {Beirlant, Jan and Guillou, Armelle and Dierckx, Goedele and Fils-Villetard, Am{\'e}lie},
  title   = {Estimation of the extreme value index and extreme quantiles under random censoring},
  journal = {Extremes},
  year    = {2007},
  volume  = {10},
  pages   = {151--174},
  doi     = {10.1007/s10687-007-0039-x}
}

@article{bucher2024statistics,
  author  = {B{\"u}cher, Axel and Jennessen, Tobias},
  title   = {Statistics for heteroscedastic time series extremes},
  journal = {Bernoulli},
  year    = {2024},
  volume  = {30},
  number  = {1},
  pages   = {46--71},
  doi     = {10.3150/22-BEJ1560}
}

@article{Daouia2010,
  author  = {Daouia, Abdelaati and Gardes, Laurent and Girard, St{\'e}phane and Lekina, Alexandre},
  title   = {Kernel estimators of extreme level curves},
  journal = {TEST},
  year    = {2011},
  volume  = {20},
  number  = {2},
  pages   = {311--333},
  doi     = {10.1007/s11749-010-0196-0}
}

@article{GardesStupfler2014,
  author  = {Gardes, Laurent and Stupfler, Gilles},
  title   = {Estimation of the conditional tail index using a smoothed local {Hill} estimator},
  journal = {Extremes},
  year    = {2014},
  volume  = {17},
  number  = {1},
  pages   = {45--75},
  doi     = {10.1007/s10687-013-0174-5}
}

@article{he2026extreme,
  author  = {He, Yi and Einmahl, John H. J.},
  title   = {Extreme value statistics for general heterogeneous data through the average tail},
  journal = {Journal of the American Statistical Association},
  year    = {2026},
  doi     = {10.1080/01621459.2026.2676731},
  note    = {Published online}
}

@article{gonzalezManteiga1994asymptotic,
  author  = {Gonz{\'a}lez-Manteiga, Wenceslao and Cadarso-Su{\'a}rez, Carmen},
  title   = {Asymptotic properties of a generalized {Kaplan--Meier} estimator with some applications},
  journal = {Journal of Nonparametric Statistics},
  year    = {1994},
  volume  = {4},
  number  = {1},
  pages   = {65--78},
  doi     = {10.1080/10485259408832601}
}

@article{mejzler1956problem,
  author  = {Mejzler, D. G.},
  title   = {On the problem of the limit distribution for the maximal term of a variational series},
  journal = {L'vov Politechn. Inst. Naucn. Zp. (Fiz.-Mat.)},
  year    = {1956},
  volume  = {38},
  pages   = {90--109},
  note    = {In Russian}
}

@article{davisonSmith1990models,
  author  = {Davison, A. C. and Smith, R. L.},
  title   = {Models for exceedances over high thresholds},
  journal = {Journal of the Royal Statistical Society. Series B (Methodological)},
  year    = {1990},
  volume  = {52},
  number  = {3},
  pages   = {393--442},
  doi     = {10.1111/j.2517-6161.1990.tb01796.x}
}

@article{hallTajvidi2000nonparametric,
  author  = {Hall, Peter and Tajvidi, Nader},
  title   = {Nonparametric analysis of temporal trend when fitting parametric models to extreme-value data},
  journal = {Statistical Science},
  year    = {2000},
  volume  = {15},
  number  = {2},
  pages   = {153--167},
  doi     = {10.1214/ss/1009212755}
}

@article{csorgo1986weighted,
  author  = {Cs{\"o}rg{\H{o}}, Mikl{\'o}s and Cs{\"o}rg{\H{o}}, S{\'a}ndor and Horv{\'a}th, Lajos and Mason, David M.},
  title   = {Weighted empirical and quantile processes},
  journal = {The Annals of Probability},
  year    = {1986},
  volume  = {14},
  number  = {1},
  pages   = {31--85},
  doi     = {10.1214/aop/1176992617}
}

@article{deHaan_Zhou_2021,
  author  = {de Haan, Laurens and Zhou, Chen},
  title   = {Trends in extreme value indices},
  journal = {Journal of the American Statistical Association},
  year    = {2021},
  volume  = {116},
  number  = {535},
  pages   = {1265--1279},
  doi     = {10.1080/01621459.2019.1705307}
}

@misc{Bladt_Rodionov_2025,
  author        = {Bladt, Martin and Rodionov, Igor},
  title         = {Censored extreme value estimation},
  year          = {2025},
  howpublished  = {arXiv preprint arXiv:2312.10499},
  eprint        = {2312.10499},
  archiveprefix = {arXiv},
  primaryclass  = {math.ST},
  doi           = {10.48550/arXiv.2312.10499}
}

@article{Cattaneo_Yu_2025,
  author  = {Cattaneo, Matias D. and Yu, Ruiqi Rae},
  title   = {Strong approximations for empirical processes indexed by {Lipschitz} functions},
  journal = {The Annals of Statistics},
  year    = {2025},
  volume  = {53},
  number  = {3},
  pages   = {1203--1229},
  doi     = {10.1214/25-AOS2500}
}

@article{Calonico_Cattaneo_Farrell_2018,
  author  = {Calonico, Sebastian and Cattaneo, Matias D. and Farrell, Max H.},
  title   = {On the effect of bias estimation on coverage accuracy in nonparametric inference},
  journal = {Journal of the American Statistical Association},
  year    = {2018},
  volume  = {113},
  number  = {522},
  pages   = {767--779},
  doi     = {10.1080/01621459.2017.1285776}
}

\newpage
\begin{appendix}
\section{Proofs of main results}\label{appendix_proof_of_main_results}
In this appendix we prove the main results stated in Section~\ref{section_main_results}. The Appendix is organized as follows. \\
In Appendix~\ref{appendix_subs_local_empirical_results} we prove the local empirical results of Subsection~\ref{subsection_local_empirical_convergence_results}. These proofs rely on a mean-squared-error decomposition into a bias and variance-like term as is usually done in kernel regression: the bias is handled by Condition~\ref{cond_H_H1_C1_F_G_Lip}, while the variance-like term is handled using a symmetrization inequality. Additional local empirical process results are given in Appendix~\ref{appendix_additional_empirical_processs_results}. In Appendix~\ref{appendix_subs_local_nelson_aalen_decomposition} we provide the decomposition of the local Nelson--Aalen estimator $\mathbb{\Lambda}^{n}(\cdot\vert s)$, partially following the steps in~\cite{STUTE_1992}. However, the proposed decomposition is more involved since we approximate the estimators based on the observations by their locally transported counterparts presented in Subsection~\ref{subsection_local_survival_decomposition_and_convergence_rates}. Crucially we show negligibility of the approximation errors, which ensures that the locally transported estimators constitute the leading terms, on which the $U$-statistic process results of~\cite{U_statistics_Stute_94} can be extended. In Appendix~\ref{appendix_proof_of_nelson_aalen_array_representation} we prove Theorem~\ref{thm_Nelson_Aalen_decomposition} by showing negligibility of all the decomposition remainder terms given in Appendix~\ref{appendix_subs_local_nelson_aalen_decomposition}. The proofs heavily rely on Lemma~\ref{lemma_Z_n_k_n_eventually_contained} and the empirical results of Subsection~\ref{subsection_local_empirical_convergence_results} and Appendix~\ref{appendix_additional_empirical_processs_results}, while also using the $U$-statistic process results of Appendix~\ref{appendix_uniform_bound_on_degenerate_u_statistics_process}. In Appendix~\ref{appendix_proof_of_nelson_aalen_consistency_rate} we prove Theorem~\ref{thm_Lambda_consistency_speed} by following the martingale techniques used in the proof of Proposition 3 in~\cite{Csorgo_96}. In Appendix~\ref{appendix_proof_of_beran_consistency_rate} we prove Proposition~\ref{prop_Beran_consistency_speed} following the argument of the proof of Theorem 2 in~\cite{Csorgo_96}: The result follows immediately from Theorem~\ref{thm_Lambda_consistency_speed} using the fact that $F_{n,\floor{ns}}(\cdot)$ is continuous. In Appendix~\ref{appendix_proof_of_admissible_tuning_sequences} we quantify the admissible tuning sequences $(k_{n})$ and $(h_{n})$ controlling the threshold and bandwidth respectively. Choosing first the threshold sequence $(k_{n})$ and then the bandwidth sequence $(h_{n})$ depending on $(k_{n})$ allows for the most direct interpretation. In Appendix~\ref{appendix_proof_of_scedasis_consistency} we prove consistency of the relative scedasis estimator $\hat{c}_{F}(s \vert s_{0})$. This result follows directly from Proposition~\ref{prop_Beran_consistency_speed}. Finally, in Appendix~\ref{appendix_uniform_bound_on_degenerate_u_statistics_process} we extend the uniform bounds of~\cite{U_statistics_Stute_94} from the iid setting to local degenerate $U$-statistic processes.

Proving limit results on the random interval $[0,Z_{n:n-k_{n}}]$ is done as in~\cite{STUTE_1992} by considering a suitable deterministic and growing interval, involving a high quantile, that eventually contains $[0,Z_{n:n-k_{n}}]$ with probability going to one. To this end define for fixed $s\in(0,1)$ and $p\in (0,1)$,
\begin{align*}
    T_{s,p}^{(n)}
    &:=
    H_{n,\floor{ns}}^{\leftarrow}(1-pk_{n}/n),
\end{align*}
We begin by showing that there exists a $p\in (0,1)$ such that this high threshold eventually dominates $Z_{n:n-k_{n}}$ with probability going to one.
\begin{lemma}\label{lemma_Z_n_k_n_eventually_contained}
   For every $s \in (0,1)$, there exists a $p \in (0,1)$ such that
    \begin{align}\label{eq_Z_n_k_n_eventually_contained}
        \amsmathbb{P}\big(Z_{n:n-k_{n}} \leq T_{s,p}^{(n)} \big)
        \to 1,
        \qquad
        \text{as} \ n \to \infty.
    \end{align}
\end{lemma}
This lemma is the analogue of Lemma 2.4 in~\cite{STUTE_1992} for heteroscedastic data.
We begin by proving Lemma~\ref{lemma_Z_n_k_n_eventually_contained}.

\begin{proof}[Proof of Lemma~\ref{lemma_Z_n_k_n_eventually_contained}]
    Let $s \in (0,1)$ and $n \in \amsmathbb{N}$ be given and let $p_s \in (0,1)$ be arbitrary for now. The continuity of $x \mapsto H_{n,\floor{ns}}(x)$ yields that
    \begin{align*}
       & \amsmathbb{P}\Big(Z_{n:n-k_{n}} \leq H_{n,\floor{ns}}^{\leftarrow}\Big(1-p_{s}\frac{k_{n}}{n}\Big) \Big)
        =
        \amsmathbb{P}\Big(H_{n,\floor{ns}}(Z_{n:n-k_{n}})  \leq 1-p_{s}\frac{k_{n}}{n}\Big) \\
        & =
        \amsmathbb{P}\Big( p_{s}\frac{k_{n}}{n} \leq 1-H_{n,\floor{ns}}(Z_{n:n-k_{n}})\Big)
        = 
        \amsmathbb{P}\Big( p_{s} \leq
        \frac{n}{k_{n}}(1-H(Z_{n:n-k_{n}}))
        \frac{1-H_{n,\floor{ns}}(Z_{n:n-k_{n}})}{1-H(Z_{n:n-k_{n}})}\Big) \\
        &=
        \amsmathbb{P}\Big( p_{s} \leq
        c_{H}(s)/C_{H}(1) +(c_{H}(\floor{ns}/n)-c_{H}(s))/ C_{H}(1) \\
        &\qquad 
        + \Big(\frac{n}{k_{n}}(1-H(Z_{n:n-k_{n}}))
        \frac{1-H_{n,\floor{ns}}(Z_{n:n-k_{n}})}{1-H(Z_{n:n-k_{n}})}
        -
         c_{H}(\floor{ns}/n)/C_{H}(1) \Big)
        \Big).
    \end{align*}
    By continuity of $c_{H}(\cdot)$ it holds that $\vert c_{H}(\floor{ns}/n)-c_{H}(s) \vert = o(1)$. Lemma~\ref{lemma_order_stat_baseline_equiv} yields that $(n/k_{n})(1-H(Z_{n:n-k_{n}})) \overset{\amsmathbb{P}}{\to} 1/C_{H}(1)$ and from~\eqref{eq_H_scedasis_dens} it immediately follows that
    \begin{align*}
        \Big\vert  
        \frac{1-H_{n,\floor{ns}}(Z_{n:n-k_{n}})}{1-H(Z_{n:n-k_{n}})} -
        c_{H}(\floor{ns}/n)
        \Big\vert \overset{\amsmathbb{P}}{\to} 0,
    \end{align*}
    as $n \to \infty$. Consequently 
    \begin{align*}
        \Big\vert\frac{n}{k_{n}}(1-H(Z_{n:n-k_{n}}))
        \frac{1-H_{n,\floor{ns}}(Z_{n:n-k_{n}})}{1-H(Z_{n:n-k_{n}})}
        -
         c_{H}(\floor{ns}/n)/C_{H}(1) \Big\vert
         = o_{\amsmathbb{P}}(1),
    \end{align*}
    as $n \to \infty$. Gathering the above argument yields that, as $n \to \infty$,
    \begin{align*}
        \amsmathbb{P}\Big(Z_{n:n-k_{n}} \leq H_{n,\floor{ns}}^{\leftarrow}\Big(1-p_{s}\frac{k_{n}}{n}\Big) \Big)
        =
        \amsmathbb{P}\Big( p_{s} \leq
        c_{H}(s)/C_{H}(1) + o(1) + o_{\amsmathbb{P}}(1)\Big).
    \end{align*}
   Letting
    \begin{align*}
        p_{s} \in (0, c_{H}(s)/C_{H}(1) \wedge 1), 
    \end{align*}
    and noting that the latter interval is nonempty because $c_{H}(s)>0$ for $s\in(0,1)$, yields that the above probability tends to $1$ as $n\to \infty$. This concludes the proof.
\end{proof}

In all subsequent sections we choose $p_{s}$ such that~\eqref{eq_Z_n_k_n_eventually_contained} holds and define $T_{s}^{(n)}:= T_{s,p_{s}}^{(n)}$.

\subsection{Proof of local empirical convergence results}\label{appendix_subs_local_empirical_results}
\begin{proof}[Proof of Lemma~\ref{lemma_H_n_and_H_n_1_unif_cons}]
    Let $s \in (0,1)$ be given.
    We consider the mean squared error (MSE), which can be decomposed into the squared bias and a variance-like term. Doing so while using the triangle inequality yields that
    \begin{align*}
        &\amsmathbb{E}\Big[\sup_{x \leq Z_{n:n-k_{n}}}\Big(\amsmathbb{H}_{n}(x \vert s) - H_{n,\floor{ns}}(x)\Big)^{2}\Big]  \\
        & \leq
           2\amsmathbb{E}\Big[\sup_{x \leq Z_{n:n-k_{n}}}(\amsmathbb{H}_{n}(x \vert s) -   \amsmathbb{E}[\amsmathbb{H}_{n}(x \vert s)]\Big)^{2}\Big] 
        +
        2\sup_{x \leq Z_{n:n-k_{n}}}\Big(\amsmathbb{E}[\amsmathbb{H}_{n}(x \vert s)] - H_{n,\floor{ns}}(x)\Big)^{2}.
    \end{align*}
    We first bound the bias term on the deterministic interval $[0, T_{s}^{(n)}]$ as follows
    \begin{align*}
         &\sup_{x \leq T_{s}^{(n)}} \vert H_{n,\floor{ns}}(x) -  \amsmathbb{E}[\amsmathbb{H}_{n}(x \vert s)]\vert
         =
          \sup_{x \leq T_{s}^{(n)}} \big\vert H_{n,\floor{ns}}(x) -   
          \frac{1}{nh_{n}}\sum_{i=1}^{n}K\Big(\frac{s -i/n}{h_{n}} \Big) H(x,i/n)
          \big\vert \\ 
          &\leq
          \sup_{x \leq T_{s}^{(n)}} \big\vert 
          H_{n,\floor{ns}}(x) -  
          \frac{1}{h_{n}}\int_{0}^{1}
          K\Big(\frac{s -u}{h_{n}} \Big)H(x,u) \mathrm{d} u
          \big\vert \\
          & \quad +
          \sup_{x \leq T_{s}^{(n)}} \big\vert 
          \int_{0}^{1}\frac{1}{h_{n}}
          K\Big(\frac{s -u}{h_{n}} \Big)H(x,u) \mathrm{d} u
          -
          \frac{1}{n}\sum_{i=1}^{n}\frac{1}{h_{n}}K\Big(\frac{s -i/n}{h_{n}} \Big) H(x,i/n)
          \big\vert,
    \end{align*}
    using the extended map $H(\cdot, \cdot)$ of Condition~\ref{cond_H_H1_C1_F_G_Lip}. The latter supremum is a Riemann error and is consequently bounded by $n^{-1}$ times the total variation of each summand (respectively integrand).
    Each summand is of total variation bounded by a constant times $1/h_{n}$, using that both $K$ and $H$ are bounded. Consequently the latter supremum is bounded by an $\mathcal{O}((nh_{n})^{-1})$. We turn to the first of the latter terms and note that
    \begin{align*}
        &\sup_{x \leq T_{s}^{(n)}} \big\vert 
          H(x, \floor{ns}/n) -  
          \frac{1}{h_{n}}\int_{0}^{1}
          K\Big(\frac{s -u}{h_{n}} \Big)H(x,u) \mathrm{d} u
          \big\vert \\
          &=
          \sup_{x \leq T_{s}^{(n)}} \big\vert 
         H(x, \floor{ns}/n) -  
          \int_{(s-1)/h_{n}}^{s/h_{n}}
          K(w)H(x,s-wh_{n}) \mathrm{d} w
          \big\vert. 
    \end{align*}
    The kernel $K$ is compactly supported and we may assume without loss of generality that it is supported on $[-1, 1]$. For the given $s \in (0,1)$, choose $\rho_{s}$ such that $I_{s} :=[s-\rho_{s},s+\rho_{s}]\subset(0,1)$ and let $t$ satisfy that $s+t \in I_{s}$. Denote by 
    \begin{align*}
        \partial_{u}H(x,w) = \frac{\partial}{\partial u}(H(x,u))\big\vert_{u=w},
    \end{align*}
    the partial derivative of $u\mapsto H(x,u)$ evaluated in $w$ for fixed $x\leq T_{s}^{(n)}$.
    Then by the fundamental theorem of calculus we have that 
    \begin{align*}
        H(x,s+t) - H(x,s) = \int_{0}^{t}  \partial_{u}H(x,s+v) \mathrm{d} v.
    \end{align*}
    and consequently
    \begin{align*}
        \vert H(x,s+t) - H(x,s) - t\partial_{u}H(x,s)\vert  &= \Big\vert \int_{0}^{t}  \partial_{u}H(x,s+v) -\partial_{u}H(x,s) \mathrm{d} v \Big\vert \\ &\leq
        \int_{0}^{\vert t\vert}  Lv \mathrm{d} v
        =\frac{Lt^{2}}{2},
    \end{align*}
    using Condition~\ref{cond_H_H1_C1_F_G_Lip} where, $L$ is some positive Lipschitz constant. Taking $t=-wh_{n}$ we obtain 
    \begin{align*}
        H(x,s -wh_{n}) = H(x,s) - wh_{n}\partial_{u}H(x,s) + R_{n}(x,w),
    \end{align*}
    where
    \begin{align*}
        \sup_{x\leq T_{s}^{(n)}}\vert R_{n}(x,w) \vert \leq \frac{L}{2}h_{n}^{2}w^{2}.
    \end{align*}
    For $n\in \amsmathbb{N}$ sufficiently large it holds that $s/h_{n} > 1$ and $(s-1)/h_{n} < -1$, since $s \in (0,1)$ and it holds that $\floor{ns}/n \in I_{s}$. For such $n$ the latter difference equals
    \begin{align*}
        &\sup_{x \leq T_{s}^{(n)}} \Big\vert  
          \int_{-1}^{1}
          K(w)(H(x, \floor{ns}/n) - H(x,s-wh_{n})) \mathrm{d} w
          \Big\vert \\
          & =
          \sup_{x \leq T_{s}^{(n)}} \Big\vert 
          \int_{-1}^{1}
          K(w)(H(x, \floor{ns}/n) - H(x,s) + wh_{n}\partial_{u}H(x,s) - R_{n}(x,w) \mathrm{d} w
          \Big\vert  \\
          & \leq
          \sup_{x \leq T_{s}^{(n)}} \big\vert H(x, \floor{ns}/n) - H(x,s) \big\vert\\
          & \quad +
          h_{n}
           \sup_{x \leq T_{s}^{(n)}} \Big\vert \partial_{u}H(x,s) \Big\vert
           \int_{-1}^{1}
          K(w) w \mathrm{d} w
          +\int_{-1}^{1}K(w) \sup_{x\leq T_{s}^{(n)}}\vert R_{n}(x,w) \vert \mathrm{d} w
           \\
           &\leq
            \mathcal{O}\Big(\frac{\vert\floor{ns} -s \vert}{n}\Big) 
            +
            h_{n}^{2}\frac{L}{2} \int_{-1}^{1}w^{2}K(w) \mathrm{d} w
          \leq
          \mathcal{O}(n^{-1}) + \mathcal{O}(h_{n}^{2}) = o\Big(\frac{1}{\sqrt{nh_{n}}}\Big),
    \end{align*}
    using Condition~\ref{cond_H_H1_C1_F_G_Lip} to bound the first-order derivative and to bound $R_{n}(x,w)$ as above, using that the kernel $K$ is symmetric and in the last inequality that $nh_{n}^{5} \to 0$. \\
    On the event $\{Z_{n:n-k_{n}} \leq T_{s}^{(n)}\}$ we have that
    \begin{align*}
        \sup_{x \leq Z_{n:n-k_{n}}}\Big(\amsmathbb{E}[\amsmathbb{H}_{n}(x \vert s)] - H_{n,\floor{ns}}(x)\Big)^{2}
        \leq
        \sup_{x \leq T_{s}^{(n)}}\Big(\amsmathbb{E}[\amsmathbb{H}_{n}(x \vert s)] - H_{n,\floor{ns}}(x)\Big)^{2}
        =
        o\Big(\frac{1}{nh_{n}}\Big).
    \end{align*}
    Lemma~\ref{lemma_Z_n_k_n_eventually_contained} yields that the interval $[0, Z_{n:n-k_{n}}]$ can be covered by $[0, T_{s}^{(n)}]$ with probability going to one. 
    For the variance-like term we consider the supremum over the random interval of interest $[0, Z_{n:n-k_{n}}]$, and use symmetrization with a sequence of iid Rademacher variables $(\varepsilon_{i})$ with Lemma 2.3.6 of~\cite{vanderVaart2023}
    \begin{align*}
        &\amsmathbb{E}\Big[\sup_{x \leq Z_{n:n-k_{n}}}\Big(\amsmathbb{H}_{n}(x \vert s) -   \amsmathbb{E}[\amsmathbb{H}_{n}(x \vert s)]\Big)^{2}\Big] \\
        &=
        \amsmathbb{E}\Big[\sup_{x \leq Z_{n:n-k_{n}}}\Big(  \frac{1}{nh_{n}}\sum_{i=1}^{n}\Big( 1\{Z_{i}^{(n)} \leq x\} - H_{n,i}(x)\Big) K\Big(\frac{s-i/n}{h_{n}}\Big)\Big)^{2}\Big] \\
        & \leq
        4 \amsmathbb{E}\Big[\sup_{x \leq Z_{n:n-k_{n}}}\Big(  \frac{1}{nh_{n}}\sum_{i=1}^{n}\varepsilon_{i}1\{Z_{i}^{(n)} \leq x\} K\Big(\frac{s-i/n}{h_{n}}\Big)\Big)^{2}\Big]
        =: 
        \frac{4}{(nh_{n})^{2}} \amsmathbb{E}\Big[\sup_{x \leq Z_{n:n-k_{n}}}\big( S(x) \big)^{2}\Big].
    \end{align*}
    Let $i_{1},\ldots, i_{n}$ be the indices of the order statistics $Z_{n:1},\ldots, Z_{n:n}$. When $x \in (Z_{n:m}, Z_{n:m+1}]$ we have that
    \begin{align*}
        S(x) = S_{m}:= \sum_{j=1}^{m}\varepsilon_{i_{j}}K\Big(\frac{s-i_{j}/n}{h_{n}}\Big), 
        \quad m=1,\ldots,n.
    \end{align*}
    Consequently we may bound the supremum $\amsmathbb{P}$-a.s. as follows
    \begin{align*}
        \sup_{x \leq Z_{n:n-k_{n}}}\big( S(x) \big)^{2}
        \leq
        \max_{1\leq m \leq n}\big( S_{m} \big)^{2}.
    \end{align*}
    Define $\overline{Z}_{n}:=(Z_{1}^{(n)},\ldots, Z_{n}^{(n)})$ and note that $(S_{m})$ is a martingale with respect to $\mathcal{G}_{m} := \sigma(\overline{Z}_{n}, \varepsilon_{i_{1}},\ldots, \varepsilon_{i_{m}})$, since it is adapted, 
    \begin{align*}
        \amsmathbb{E}[S_{m+1} \vert \mathcal{G}_{m}]
        =
        \amsmathbb{E}[S_{m} + \varepsilon_{i_{m+1}}K((s-i_{m+1}/n)/h_{n}) \vert \mathcal{G}_{m}]
        =
        S_{m},
    \end{align*} 
    and clearly, for any fixed $n$, we have $\amsmathbb{E}[S_{n}^{2}]< \infty$. 
    In particular, Doob's inequality yields that
    \begin{align*}
        \amsmathbb{E}\Big[\sup_{x \leq Z_{n:n-k_{n}}}\big( S(x) \big)^{2}\Big]
        \leq
        \amsmathbb{E}\Big[ \max_{1\leq m \leq n}\big( S_{m} \big)^{2}\Big] 
        \leq
        4\amsmathbb{E}[S_{n}^{2}] 
        =
        4\sum_{i=1}^{n} K^{2}\Big(\frac{s-i/n}{h_{n}}\Big).
    \end{align*}
    Summarizing, we have that
    \begin{align*}
        &\amsmathbb{E}\Big[\sup_{x \leq Z_{n:n-k_{n}}}\Big(\amsmathbb{H}_{n}(x \vert s) -   \amsmathbb{E}[\amsmathbb{H}_{n}(x \vert s)]\Big)^{2}\Big]
        \leq
           \frac{16}{(nh_{n})^{2}}\sum_{i=1}^{n}K^{2}\Big(\frac{s-i/n}{h_{n}}\Big)
          =
          \mathcal{O}\Big(\frac{1}{nh_{n}}\Big).
    \end{align*}
    This concludes the two terms of the MSE. An application of Chebyshev's inequality finishes the proof. The second statement of the Lemma is proven identically to the first statement.
\end{proof}

\begin{proof}[Proof of Lemma~\ref{lemma_H_tail_fraction}]
We note that
\begin{align*}
    \amsmathbb{E}\Big[
    \Big(
    \sup_{x \leq Z_{n:n-k_{n}}}\Big(\frac{1-\amsmathbb{H}_{n}(x- \vert s)}{1-H_{n,\floor{ns}}(x-)} \Big)
    -
    1
    \Big)^{2}
    \Big]
    =
    \amsmathbb{E}\Big[
    \Big(
    \sup_{x \leq Z_{n:n-k_{n}}}\Big(\frac{H_{n,\floor{ns}}(x-)-\amsmathbb{H}_{n}(x- \vert s)}{1-H_{n,\floor{ns}}(x-)} \Big)
    \Big)^{2}
    \Big],
\end{align*}
which is the mean squared error of the fraction considered in Lemma~\ref{lemma_H_minus_H_gamma_fraction_rate} with $\gamma = 0$. In particular, the Lemma yields that the latter fraction converges in probability to 0, since $nh_{n}^{2}/k_{n}$, $\sqrt{nh_{n}}/k_{n}$ and $1/(k_{n}h_{n})$ are all decreasing.
\end{proof}

\subsection{Local Nelson--Aalen decomposition}\label{appendix_subs_local_nelson_aalen_decomposition}
First we introduce the auxiliary estimators based on the locally transformed data.
\begin{align*}
     &\widetilde{\amsmathbb{H}}_{n}(x \vert s) 
     : = \frac{1}{nh_{n}} \sum_{i=1}^{n} 1\{Z_{\floor{ns}}^{(n,i)}\leq x\} K\Big(\frac{s - i/n}{h_{n}}\Big) , 
     \\
     &\widetilde{\amsmathbb{H}}^{1}_{n}(x \vert s) : = 
     \frac{1}{nh_{n}} \sum_{i=1}^{n} 1\{Z_{\floor{ns}}^{(n,i)}\leq x, \delta_{\floor{ns}}^{(n,i)} = 1\} K\Big(\frac{s - i/n}{h_{n}}\Big) ,
\end{align*}
On $y \leq Z_{n,n}$, for $n \in \amsmathbb{N}$ large enough, we have that
\begin{align*}
    \frac{1}{1-\amsmathbb{H}_{n}(y-\vert s)}
    &=
    \frac{2}{1-H_{n,\floor{ns}}(y-)} \\
   & \qquad +
    \frac{(\amsmathbb{H}_{n}(y-\vert s) - H_{n,\floor{ns}}(y-))^{2}}{(1-H_{n,\floor{ns}}(y-))^{2}(1-\amsmathbb{H}_{n}(y-\vert s))}
    -\frac{1-\amsmathbb{H}_{n}(y-\vert s)}{(1-H_{n,\floor{ns}}(y-))^{2}}.
\end{align*}
We can then integrate both sides with respect to $\amsmathbb{H}^{1}_{n}$ for each $x \leq Z_{n,n}$. This yields
\begin{align*}
    &\mathbb{\Lambda}^{n}(x\vert s) =
    2\int_{0}^{x}
    \frac{1}{1-H_{n,\floor{ns}}(y-)} \amsmathbb{H}^{1}_{n}(\mathrm{d} y \vert s) \\
    & \quad 
    +
    \int_{0}^{x}
    \frac{(\amsmathbb{H}_{n}(y- \vert s) - H_{n,\floor{ns}}(y-))^{2}}{(1-H_{n,\floor{ns}}(y-))^{2}(1-\amsmathbb{H}_{n}(y- \vert s))}\amsmathbb{H}^{1}_{n}(\mathrm{d} y \vert s) -
    \int_{0}^{x}\frac{1-\amsmathbb{H}_{n}(y- \vert s)}{(1-H_{n,\floor{ns}}(y-))^{2}}\amsmathbb{H}^{1}_{n}(\mathrm{d} y \vert s).
\end{align*}
The latter term can be approximated as follows
\begin{align*}
     &\int_{0}^{x}\frac{1-\amsmathbb{H}_{n}(y- \vert s)}{(1-H_{n,\floor{ns}}(y-))^{2}}\amsmathbb{H}^{1}_{n}(\mathrm{d} y \vert s)=
      \int_{0}^{x}\frac{1-\widetilde{\amsmathbb{H}}_{n}(y- \vert s)}{(1-H_{n,\floor{ns}}(y-))^{2}}\widetilde{\amsmathbb{H}}^{1}_{n}(\mathrm{d} y \vert s) \\
      & \qquad +
      \Big(\int_{0}^{x}\frac{1-\amsmathbb{H}_{n}(y- \vert s)}{(1-H_{n,\floor{ns}}(y-))^{2}}\amsmathbb{H}^{1}_{n}(\mathrm{d} y \vert s)
    -
    \int_{0}^{x}\frac{1-\widetilde{\amsmathbb{H}}_{n}(y- \vert s)}{(1-H_{n,\floor{ns}}(y-))^{2}}\widetilde{\amsmathbb{H}}^{1}_{n}(\mathrm{d} y \vert s)\Big) \\
    &=:
     \int_{0}^{x}\frac{1-\widetilde{\amsmathbb{H}}_{n}(y- \vert s)}{(1-H_{n,\floor{ns}}(y-))^{2}}\widetilde{\amsmathbb{H}}^{1}_{n}(\mathrm{d} y \vert s) + \widetilde{R}_{n}(x\vert s),
\end{align*}
and the leading term can be decomposed as
\begin{align*}
    &\int_{0}^{x}\frac{1-\widetilde{\amsmathbb{H}}_{n}(y- \vert s)}{(1-H_{n,\floor{ns}}(y-))^{2}}\widetilde{\amsmathbb{H}}^{1}_{n}(\mathrm{d} y \vert s) \\
    & =
    \int_{0}^{x}\frac{\sum_{j=1}^{n}1\{Z_{\floor{ns}}^{(n,j)} \geq y\}K((s -j/n)/h_{n})}{(nh_{n})^{2}(1-H_{n,\floor{ns}}(y-))^{2}}
    \sum_{i=1}^{n} 1\{Z_{\floor{ns}}^{(n,i)}\leq \mathrm{d} y, \delta_{\floor{ns}}^{(n,i)} = 1 \}K\Big(\frac{s - i/n}{h_{n}}\Big) \\
    &\qquad - \Big(1- \frac{1}{nh_{n}}\sum_{i=1}^{n}K\Big(\frac{s-i/n}{h_{n}}\Big) \Big)
    \int_{0}^{x}\frac{1}{(1-H_{n,\floor{ns}}(y-))^{2}}\widetilde{\amsmathbb{H}}^{1}_{n}(\mathrm{d} y \vert s) \\
    & =
    \frac{1}{(nh_{n})^{2}}
    \sum_{i=1}^{n}\sum_{j=1}^{n} \frac{1\{ Z_{\floor{ns}}^{(n,j)} \geq Z_{\floor{ns}}^{(n,i)} \}1\{Z_{\floor{ns}}^{(n,i)} \leq x, \delta_{\floor{ns}}^{(n,i)} = 1\}}{(1-H_{n,\floor{ns}}(Z_{\floor{ns}}^{(n,i)}-))^{2}}K\Big(\frac{s -j/n}{h_{n}}\Big)K\Big(\frac{s -i/n}{h_{n}}\Big) \\
    & \qquad - \Big(1- \frac{1}{nh_{n}}\sum_{i=1}^{n}K\Big(\frac{s-i/n}{h_{n}}\Big) \Big)
    \int_{0}^{x}\frac{1}{(1-H_{n,\floor{ns}}(y-))^{2}}\widetilde{\amsmathbb{H}}^{1}_{n}(\mathrm{d} y \vert s)\\
    &= 
    \frac{1}{(nh_{n})^{2}}
    \sum_{i \neq j}^{n} \frac{1\{ Z_{\floor{ns}}^{(n,j)} \geq Z_{\floor{ns}}^{(n,i)} \}1\{Z_{\floor{ns}}^{(n,i)} \leq x, \delta_{\floor{ns}}^{(n,i)} = 1\}}{(1-H_{n,\floor{ns}}(Z_{\floor{ns}}^{(n,i)}-))^{2}}K\Big(\frac{s -j/n}{h_{n}}\Big)K\Big(\frac{s -i/n}{h_{n}}\Big) \\
    & \qquad 
    +
    \frac{1}{(nh_{n})^{2}}
    \sum_{i =1}^{n} \frac{1\{Z_{\floor{ns}}^{(n,i)} \leq x, \delta_{\floor{ns}}^{(n,i)} = 1\}}{(1-H_{n,\floor{ns}}(Z_{\floor{ns}}^{(n,i)}-))^{2}}K^{2}\Big(\frac{s -i/n}{h_{n}}\Big) \\
     & \qquad - \Big(1- \frac{1}{nh_{n}}\sum_{i=1}^{n}K\Big(\frac{s-i/n}{h_{n}}\Big) \Big)
    \int_{0}^{x}\frac{1}{(1-H_{n,\floor{ns}}(y-))^{2}}\widetilde{\amsmathbb{H}}^{1}_{n}(\mathrm{d} y \vert s) \\
    & =:
    \frac{1}{(nh_{n})^{2}} U_{n}(x \vert s)
    +
    \Big[ \Big(-1 + \frac{1}{nh_{n}}\sum_{i=1}^{n}K\Big(\frac{s-i/n}{h_{n}}\Big) \Big)
    \int_{0}^{x}\frac{1}{(1-H_{n,\floor{ns}}(y-))^{2}}\widetilde{\amsmathbb{H}}^{1}_{n}(\mathrm{d} y \vert s) \\
     & \qquad + \frac{1}{n}\int_{0}^{x}\frac{1}{(1-H_{n,\floor{ns}}(y-))^{2}} 
    \Big( \frac{1}{nh_{n}^{2}}\sum_{i=1}^{n} 1\{Z_{\floor{ns}}^{(n,i)}\leq \mathrm{d} y, \delta_{\floor{ns}}^{(n,i)} = 1 \}K^{2}\Big(\frac{s - i/n}{h_{n}}\Big) \Big) \Big] \\
    & =:  \frac{1}{(nh_{n})^{2}} U_{n}(x \vert s) + R_{n1}(x\vert s).
\end{align*}
Then $\mathbb{\Lambda}^{n}(\cdot \vert s)$ can be written as
\begin{align*}
      \mathbb{\Lambda}^{n}(x\vert s) &=
    \int_{0}^{x}
    \frac{2}{1-H_{n,\floor{ns}}(y-)} \amsmathbb{H}^{1}_{n}(\mathrm{d} y \vert s)
    +
    \int_{0}^{x}
    \frac{(\amsmathbb{H}_{n}(y- \vert s) - H_{n,\floor{ns}}(y-))^{2}}{(1-H_{n,\floor{ns}}(y-))^{2}(1-\amsmathbb{H}_{n}(y- \vert s))}\amsmathbb{H}^{1}_{n}(\mathrm{d} y \vert s) \\
    & \qquad \qquad -
    \big(\widetilde{R}_{n}(x\vert s) + \frac{1}{(nh_{n})^{2}} U_{n}(x \vert s) + R_{n1}(x\vert s)\big) \\
    &=:
    2S_{n1}(x \vert s) + R_{n2}(x \vert s) - \big(\widetilde{R}_{n}(x\vert s) - \frac{1}{(nh_{n})^{2}} U_{n}(x \vert s) - R_{n1}(x\vert s)\big).
\end{align*}
Following the idea of~\cite{STUTE_1992} we're interested in adding and subtracting the Hájek projection of $U_{n}(\cdot \vert s)$, and using the uniform error bounds of Appendix~\ref{appendix_uniform_bound_on_degenerate_u_statistics_process} on the corresponding degenerate $U$-statistic process. To this end, define the variables
\begin{align*}
    W_{\floor{ns}}^{(n,i)} = \begin{cases}
        Z_{\floor{ns}}^{(n,i)}, \quad \delta_{\floor{ns}}^{(n,i)} =1, \\
        \infty, \quad \ \ \ \ \delta_{\floor{ns}}^{(n,i)}=0,
    \end{cases}
\end{align*}
for all $i \leq n$. Note that it then holds that
\begin{align*}
    \amsmathbb{P}(W_{\floor{ns}}^{(n,i)} \leq y) = \amsmathbb{P}(Z_{\floor{ns}}^{(n,i)} \leq y, \delta_{\floor{ns}}^{(n,i)} =1) = H_{n,\floor{ns}}^{1}(y),
\end{align*}
by definition of $H_{n,\floor{ns}}^{1}$. Similarly it holds that $\delta_{\floor{ns}}^{(n,i)}1\{Z_{\floor{ns}}^{(n,i)} \leq y\} = 1\{W_{\floor{ns}}^{(n,i)} \leq y\}$. Define 
\begin{align*}
    h_{n,\floor{ns}}(y,z) = \frac{1\{z\geq y\}}{(1-H_{n,\floor{ns}}(y-))^{2}}, \quad y \leq x^{*},
\end{align*}
and note that $U_{n}(\cdot \vert s)$ can now be written as
\begin{align*}
    U_{n}(x\vert s)
    &=
    \sum_{i \neq j}^{n} \frac{1\{ Z_{\floor{ns}}^{(n,j)} \geq Z_{\floor{ns}}^{(n,i)} \}}{(1-H_{n,\floor{ns}}(Z_{\floor{ns}}^{(n,i)}-))^{2}}
    1\{Z_{\floor{ns}}^{(n,i)} \leq x, \delta_{\floor{ns}}^{(n,i)} = 1\}
    K\Big(\frac{s -j/n}{h_{n}}\Big)K\Big(\frac{s -i/n}{h_{n}}\Big) \\
    &= 
    \sum_{i \neq j}^{n} h_{n,\floor{ns}}(W_{\floor{ns}}^{(n,i)},Z_{\floor{ns}}^{(n,j)})
    1\{W_{\floor{ns}}^{(n,i)} \leq x\}
    K\Big(\frac{s -j/n}{h_{n}}\Big)K\Big(\frac{s -i/n}{h_{n}}\Big).
\end{align*}
On the set $\{i \neq j\}$ we have $W_{\floor{ns}}^{(n,i)}$ and $Z_{\floor{ns}}^{(n,j)}$ are independent. We compute the Hájek projection of each $(W_{\floor{ns}}^{(n,i)}, Z_{\floor{ns}}^{(n,j)}) \mapsto  h_{n,\floor{ns}}(W_{\floor{ns}}^{(n,i)}, Z_{\floor{ns}}^{(n,j)}) 1\{W_{\floor{ns}}^{(n,i)} \leq x\}$ to obtain the Hájek projection of $U_{n}(\cdot \vert s)$. Below we use the notation $K_{n,s}(i,j) := K((s -j/n)/h_{n})K((s -i/n)/h_{n})$.
\begin{align*}
    &\widehat{U}_{n}(x\vert s)
      :=
     \sum_{i\neq j}^{n}  K_{n,s}(i,j) \Big\{ \int_{0}^{x^{*}} h_{n,\floor{ns}}(W_{\floor{ns}}^{(n,i)},z) 1\{W_{\floor{ns}}^{(n,i)} \leq x\} H_{n,\floor{ns}}(\mathrm{d} z) \\
      &\qquad +
     \int_{0}^{x} h_{n,\floor{ns}}(y,Z_{\floor{ns}}^{(n,j)})  H_{n,\floor{ns}}^{1}(\mathrm{d} y) 
     -
     \int_{0}^{x^{*}} \int_{0}^{x}h_{n,\floor{ns}}(y,z)  H_{n,\floor{ns}}^{1}(\mathrm{d} y)  H_{n,\floor{ns}}(\mathrm{d} z)  \Big\} \\
     &=
     \sum_{i\neq j}^{n}  K_{n,s}(i,j) \Big\{ \int_{0}^{x^{*}} \frac{1\{z\geq W_{\floor{ns}}^{(n,i)}\}}{(1-H_{n,\floor{ns}}(W_{\floor{ns}}^{(n,i)}-))^{2}} 1\{W_{\floor{ns}}^{(n,i)} \leq x\} H_{n,\floor{ns}}(\mathrm{d} z) \\
      &\qquad \qquad \qquad \qquad  +
     \int_{0}^{x} \frac{1\{Z_{\floor{ns}}^{(n,j)}\geq y\}}{(1-H_{n,\floor{ns}}(y-))^{2}}  H_{n,\floor{ns}}^{1}(\mathrm{d} y)  \\
     & \qquad \qquad \qquad \qquad  -
     \int_{0}^{x^{*}} \int_{0}^{x \wedge z} \frac{1}{(1-H_{n,\floor{ns}}(y-))^{2}}  H_{n,\floor{ns}}^{1}(\mathrm{d} y)  H_{n,\floor{ns}}(\mathrm{d} z) \Big\} \\
     &=
     \sum_{i\neq j}^{n}  K_{n,s}(i,j) \Big\{ \int_{0}^{x^{*}} \frac{1\{Z_{\floor{ns}}^{(n,i)} \leq x \wedge z\}}{(1-H_{n,\floor{ns}}(Z_{\floor{ns}}^{(n,i)}-))^{2}} 1\{\delta_{\floor{ns}}^{(n,i)} = 1\} H_{n,\floor{ns}}(\mathrm{d} z) \\
      &\qquad \qquad \qquad \qquad  +
     \int_{0}^{x} \frac{1\{Z_{\floor{ns}}^{(n,j)}\geq y\}}{(1-H_{n,\floor{ns}}(y-))^{2}}  H_{n,\floor{ns}}^{1}(\mathrm{d} y) \\
     & \qquad \qquad \qquad \qquad  -
     \int_{0}^{x} \int_{y}^{x^{*}} H_{n,\floor{ns}}(\mathrm{d} z) \frac{1}{(1-H_{n,\floor{ns}}(y-))^{2}}  H_{n,\floor{ns}}^{1}(\mathrm{d} y)  \Big\} \\
     &=
     \sum_{i\neq j}^{n}  K_{n,s}(i,j) \Big\{ \frac{1\{Z_{\floor{ns}}^{(n,i)} \leq x \}1\{\delta_{\floor{ns}}^{(n,i)} = 1\}}{(1-H_{n,\floor{ns}}(Z_{\floor{ns}}^{(n,i)}-))^{2}}  \int_{Z_{\floor{ns}}^{(n,i)}}^{x^{*}}  H_{n,\floor{ns}}(\mathrm{d} z) \\
      &\qquad +
     \int_{0}^{x} \frac{1\{Z_{\floor{ns}}^{(n,j)}\geq y\} - (1 - H_{n,\floor{ns}}(y -))}{(1-H_{n,\floor{ns}}(y-))^{2}}  H_{n,\floor{ns}}^{1}(\mathrm{d} y) \Big\}
     \\
     &=
     \sum_{i\neq j}^{n}  K_{n,s}(i,j) \Big\{ \frac{1\{Z_{\floor{ns}}^{(n,i)} \leq x \}1\{\delta_{\floor{ns}}^{(n,i)} = 1\}}{1-H_{n,\floor{ns}}(Z_{\floor{ns}}^{(n,i)}-)}     \\
      & \qquad \qquad +
     \int_{0}^{x} \frac{1\{Z_{\floor{ns}}^{(n,j)}\geq y\} - (1 - H_{n,\floor{ns}}(y -))}{(1-H_{n,\floor{ns}}(y-))^{2}}  H_{n,\floor{ns}}^{1}(\mathrm{d} y) \Big\}
     \\
     &=
     \Big(\sum_{j=1}^{n} K\Big(\frac{s-j/n}{h_{n}}\Big)\Big) \sum_{i=1}^{n} K\Big(\frac{s-i/n}{h_{n}}\Big) \Big\{ \frac{1\{Z_{\floor{ns}}^{(n,i)} \leq x \}1\{\delta_{\floor{ns}}^{(n,i)} = 1\}}{1-H_{n,\floor{ns}}(Z_{\floor{ns}}^{(n,i)}-)}     \\
      & \qquad \qquad \qquad \qquad +
     \int_{0}^{x} \frac{1\{Z_{\floor{ns}}^{(n,i)}\geq y\} - (1 - H_{n,\floor{ns}}(y -))}{(1-H_{n,\floor{ns}}(y-))^{2}}  H_{n,\floor{ns}}^{1}(\mathrm{d} y) \Big\} \\
     & \qquad -
     \sum_{i=1}^{n} K^{2}\Big(\frac{s-i/n}{h_{n}}\Big)\Big\{ \frac{1\{Z_{\floor{ns}}^{(n,i)} \leq x \}1\{\delta_{\floor{ns}}^{(n,i)} = 1\}}{1-H_{n,\floor{ns}}(Z_{\floor{ns}}^{(n,i)}-)}     \\
      & \qquad \qquad \qquad \qquad +
     \int_{0}^{x} \frac{1\{Z_{\floor{ns}}^{(n,i)}\geq y\} - (1 - H_{n,\floor{ns}}(y -))}{(1-H_{n,\floor{ns}}(y-))^{2}}  H_{n,\floor{ns}}^{1}(\mathrm{d} y) \Big\}\\
     &=: (nh_{n}) \Big(\frac{1}{nh_{n}}\sum_{j=1}^{n} K\Big(\frac{s-j/n}{h_{n}}\Big)\Big)\widehat{V}_{n}(x \vert s) - R_{n3}(x \vert s)\\
     &=
     (nh_{n})\widehat{V}_{n}(x \vert s) +
     \Big(\frac{1}{nh_{n}}\sum_{j=1}^{n} K\Big(\frac{s-j/n}{h_{n}}\Big) -1 \Big)\widehat{V}_{n}(x \vert s) - R_{n3}(x \vert s)
\end{align*}
In addition, we approximate $S_{n1}(\cdot \vert s)$ by its locally transported counterpart 
\begin{align*}
    \widetilde{S}_{n1}(x \vert s):=
    \int_{0}^{x}
    \frac{1}{1-H_{n,\floor{ns}}(y-)}  \widetilde{\amsmathbb{H}}^{1}_{n}(\mathrm{d} y \vert s).
\end{align*}
This yields that
\begin{align*}
    &\mathbb{\Lambda}^{n}(x\vert s) 
    =
    2S_{n1}(x \vert s)  - \frac{1}{nh_{n}} U_{n}(x \vert s) - \widetilde{R}_{n}(x \vert s) - R_{n1}(x\vert s) + R_{n2}(x \vert s) \\
    &=
    2\widetilde{S}_{n1}(x \vert s)   - \frac{1}{nh_{n}} \widehat{V}_{n}(x \vert s) \\
    & \quad -\widetilde{R}_{n}(x \vert s) - R_{n1}(x\vert s) + R_{n2}(x \vert s) + \frac{1}{(nh_{n})^{2}} R_{n3}(x \vert s) - 2(\widetilde{S}_{n1}(x \vert s)  - S_{n1}(x \vert s) ) \\
    & \quad 
    -  \frac{1}{nh_{n}}\Big(1 -\frac{1}{nh_{n}}\sum_{j=1}^{n} K\Big(\frac{s-j/n}{h_{n}}\Big) \Big)\widehat{V}_{n}(x \vert s) 
    - \frac{1}{(nh_{n})^{2}} (U_{n}(x \vert s) - \widehat{U}_{n}(x \vert s)) \\
    &=: 2\widetilde{S}_{n1}(x \vert s)  - \frac{1}{nh_{n}} \widehat{V}_{n}(x \vert s) + R_{n}^{0}(x \vert s),
\end{align*}
where the terms first two terms yields the representation
\begin{align*}
    &\mathbb{\Lambda}^{n, *}(x\vert s) := 2\widetilde{S}_{n1}(x \vert s)  - \frac{1}{nh_{n}} \widehat{V}_{n}(x \vert s) \\
    &=
    2\int_{0}^{x}
    \frac{1}{1-H_{n,\floor{ns}}(y-)} \widetilde{\amsmathbb{H}}^{1}_{n}(\mathrm{d} y \vert s) -
    \frac{1}{nh_{n}}\sum_{i=1}^{n}  K\Big(\frac{s-i/n}{h_{n}}\Big) 
    \Big\{ \frac{1\{Z_{\floor{ns}}^{(n,i)} \leq x \}1\{\delta_{\floor{ns}}^{(n,i)} = 1\}}{1-H_{n,\floor{ns}}(Z_{\floor{ns}}^{(n,i)}-)}     \\
      & \qquad \qquad \qquad \qquad \qquad \qquad \qquad \qquad  +
     \int_{0}^{x} \frac{1\{Z_{\floor{ns}}^{(n,i)}\geq y\} - (1 - H_{n,\floor{ns}}(y -))}{(1-H_{n,\floor{ns}}(y-))^{2}}  H_{n,\floor{ns}}^{1}(\mathrm{d} y) \Big\}\\
    &=
    \int_{0}^{x}
    \frac{1}{1-H_{n,\floor{ns}}(y-)}  \widetilde{\amsmathbb{H}}^{1}_{n}(\mathrm{d} y \vert s)\\
     & \qquad -
     \int_{0}^{x} \frac{ (nh_{n})^{-1}\sum_{i=1}^{n}K((s-i/n)/h_{n})(1\{Z_{\floor{ns}}^{(n,i)}\geq y\} - (1 - H_{n,\floor{ns}}(y -)))}{(1-H_{n,\floor{ns}}(y-))^{2}} H_{n,\floor{ns}}^{1}(\mathrm{d} y). 
\end{align*}
We state the following useful fact from~\cite{STUTE_1992} Lemma 2.1, which is used throughout the paper: for any distribution function $J$ it holds that $J(J^{-1}(u)-)\leq u$ for all $u \in (0,1)$ and 
\begin{align}\label{eq_integral_bound}
    \int_{0}^{J^{-1}(u)} \frac{1}{(1-J(x-))^{{\beta}}} \mathrm{d} J(x) \leq \frac{\beta}{(\beta -1)(1-u)^{\beta -1}}, \quad 0<u<1, \ \text{for all} \ \beta >1.
\end{align}

\subsection{Proof of Theorem~\ref{thm_Nelson_Aalen_decomposition} (Array representation; Local Nelson--Aalen estimator)}\label{appendix_proof_of_nelson_aalen_array_representation}
After establishing a technical lemma we prove Theorem~\ref{thm_Nelson_Aalen_decomposition} by first bounding the difference $\mathbb{\Lambda}^{n}(\cdot \vert s) - \mathbb{\Lambda}^{n,\ast}(\cdot \vert s)$ and subsequently the difference $\mathbb{\Lambda}^{n,\ast}(\cdot \vert s) - \Lambda_{n,\floor{ns}}(\cdot)$ uniformly on the interval $[0,T_{s}^{(n)}]$. The corresponding results on the random interval $[0,Z_{n:n-k_{n}}]$ follow then directly by Lemma~\ref{lemma_Z_n_k_n_eventually_contained}.

\begin{lemma}\label{lemma_integral_wrt_H_tilde_size_n_over_k}
It holds that
\begin{align*}
    \int_{0}^{T_{s}^{(n)}}\frac{1}{(1-H_{n,\floor{ns}}(y-))^{2}}\widetilde{\amsmathbb{H}}^{1}_{n}(\mathrm{d} y \vert s)
    =
    \mathcal{O}_{\amsmathbb{P}}(n/k_{n}).
\end{align*}
\end{lemma}
\begin{proof}[Proof of Lemma~\ref{lemma_integral_wrt_H_tilde_size_n_over_k}]
Note that 
\begin{align*}
     \int_{0}^{T_{s}^{(n)}}\frac{1}{(1-H_{n,\floor{ns}}(y-))^{2}}\widetilde{\amsmathbb{H}}^{1}_{n}(\mathrm{d} y \vert s) 
     =
     \frac{1}{nh_{n}}\sum_{i=1}^{n}K\Big(\frac{s-i/n}{h_{n}}\Big) \frac{1\{Z_{\floor{ns}}^{(n,i)} \leq T_{s}^{(n)} \}1\{\delta_{\floor{ns}}^{(n,i)} = 1\}}{(1-H_{n,\floor{ns}}(Z_{\floor{ns}}^{(n,i)}-))^{2}},
\end{align*}
which is a positive stochastic process with mean 
\begin{align*}
  &\frac{1}{nh_{n}}\sum_{i=1}^{n}K\Big(\frac{s-i/n}{h_{n}}\Big) \int_{0}^{T_{s}^{(n)}}\frac{1}{(1-H_{n,\floor{ns}}(y-))^{2}}H_{n,\floor{ns}}^{1}(\mathrm{d} y) \\
  &\leq
   \frac{1}{nh_{n}}\sum_{i=1}^{n}K\Big(\frac{s-i/n}{h_{n}}\Big)
   \int_{0}^{T_{s}^{(n)}}\frac{1}{(1-H_{n,\floor{ns}}(y-))^{2}}H_{n,\floor{ns}}(\mathrm{d} y)
   \leq
    \frac{2n}{p_{s}k_{n}}\frac{1}{nh_{n}}\sum_{i=1}^{n}K\Big(\frac{s-i/n}{h_{n}}\Big),
\end{align*}
which is asymptotically of size $(2n/p_{s}k_{n})$. Consequently
\begin{align*}
    &\amsmathbb{P}\Big( \frac{p_{s}k_{n}}{2n} \int_{0}^{T_{s}^{(n)}}\frac{1}{(1-H_{n,\floor{ns}}(y-))^{2}}\widetilde{\amsmathbb{H}}^{1}_{n}(\mathrm{d} y \vert s) > u \Big)
    \leq
    \frac{1}{u}\frac{1}{nh_{n}}\sum_{i=1}^{n}K\Big(\frac{s-i/n}{h_{n}}\Big) \\
    &=
    1/u - 1/u\Big(1 -\frac{1}{nh_{n}}\sum_{i=1}^{n}K\Big(\frac{s-i/n}{h_{n}}\Big)\Big).
\end{align*}
For $n\in \amsmathbb{N}$ sufficiently large the latter term can be made smaller than any given $\delta \in (0,1)$. 
For such $n$
\begin{align*}
    \amsmathbb{P}\Big( \frac{p_{s}k_{n}}{2n} \int_{0}^{T_{s}^{(n)}}\frac{1}{(1-H_{n,\floor{ns}}(y-))^{2}}\widetilde{\amsmathbb{H}}^{1}_{n}(\mathrm{d} y \vert s) > u \Big)
    \leq
    \frac{1-\delta}{u},
\end{align*}
which vanishes for $u \to \infty$. This concludes the proof.
\end{proof}

\begin{theorem}\label{thm_NA_weak_repr}
Assume Condition~\ref{cond_H_H1_C1_F_G_Lip}, assume that $\log(n)nh_{n}^{2}/k_{n}$ and $\log(n)\sqrt{n}/k_{n}$ are decreasing in $n$, and let $s \in (0,1)$.
 It then holds that
\begin{align*}
    &\mathbb{\Lambda}^{n}(x\vert s)
    =
     R_{n}^{0}(x \vert s) + 
    \int_{0}^{x}
\frac{1}{1-H_{n,\floor{ns}}(y-)}  \widetilde{\amsmathbb{H}}^{1}_{n}(\mathrm{d} y \vert s)\\
 & \qquad -
 \int_{0}^{x} \frac{ (nh_{n})^{-1}\sum_{i=1}^{n}K((s-i/n)/h_{n})(1\{Z_{\floor{ns}}^{(n,i)}\geq y\} - (1 - H_{n,\floor{ns}}(y -)))}{(1-H_{n,\floor{ns}}(y-))^{2}} H_{n,\floor{ns}}^{1}(\mathrm{d} y)   \\
     & =: R_{n}^{0}(x \vert s) +  \mathbb{\Lambda}^{n, *}(x\vert s) ,
\end{align*}
where 
\begin{align}\label{eq_error_term_rate_NA_repr}
    \sup_{x \leq Z_{n:n-k_{n}}} \vert R_{n}^{0}(x \vert s) \vert = \mathcal{O}_{\amsmathbb{P}}\Big(\frac{n\log(n)}{k_{n}}(h_{n}^{2} \vee 1/\sqrt{n}\vee 1/(nh_{n}))\Big)
\end{align}
\end{theorem}

\begin{proof}[Proof of Theorem~\ref{thm_NA_weak_repr}]
We show that each of the remainder terms of $R_{n}^{0}(\cdot \vert s)$ satisfies the uniform weak bound~\eqref{eq_error_term_rate_NA_repr}. First, consider the difference 
\begin{align*}
    &\widetilde{R}_{n}(x \vert s) = \int_{0}^{x}\frac{1-\amsmathbb{H}_{n}(y- \vert s)}{(1-H_{n,\floor{ns}}(y-))^{2}}\amsmathbb{H}^{1}_{n}(\mathrm{d} y \vert s)
    -
    \int_{0}^{x}\frac{1-\widetilde{\amsmathbb{H}}_{n}(y- \vert s)}{(1-H_{n,\floor{ns}}(y-))^{2}}\widetilde{\amsmathbb{H}}^{1}_{n}(\mathrm{d} y \vert s) \\
    &=
    \int_{0}^{x}\frac{1-\amsmathbb{H}_{n}(y- \vert s)}{(1-H_{n,\floor{ns}}(y-))^{2}}(\amsmathbb{H}^{1}_{n}(\mathrm{d} y \vert s)- \widetilde{\amsmathbb{H}}^{1}_{n}(\mathrm{d} y \vert s))
    +
    \int_{0}^{x}\frac{\widetilde{\amsmathbb{H}}_{n}(y- \vert s)-\amsmathbb{H}_{n}(y- \vert s)}{(1-H_{n,\floor{ns}}(y-))^{2}}\widetilde{\amsmathbb{H}}^{1}_{n}(\mathrm{d} y \vert s).
\end{align*}
Note that
\begin{align*}
    &\sup_{x\leq T_{s}^{(n)}} \Big\vert \int_{0}^{x}\frac{1-\amsmathbb{H}_{n}(y- \vert s)}{(1-H_{n,\floor{ns}}(y-))^{2}}(\amsmathbb{H}^{1}_{n}(\mathrm{d} y \vert s)- \widetilde{\amsmathbb{H}}^{1}_{n}(\mathrm{d} y \vert s)) \Big\vert\\
    &\leq
    \sup_{x\leq T_{s}^{(n)}} \Big\vert \frac{1-\amsmathbb{H}_{n}(x- \vert s)}{1-H_{n,\floor{ns}}(x-)} \Big\vert
    \sup_{x\leq T_{s}^{(n)}} \Big\vert \int_{0}^{x}\frac{1}{1-H_{n,\floor{ns}}(y-)}(\amsmathbb{H}^{1}_{n}(\mathrm{d} y \vert s)- \widetilde{\amsmathbb{H}}^{1}_{n}(\mathrm{d} y \vert s)) \Big\vert \\
    &=
    \mathcal{O}_{\amsmathbb{P}}(1)
    \sup_{x\leq T_{s}^{(n)}} \big\vert 
    S_{n}(x\vert s) - \widetilde{S}_{n}(x\vert s)
    \big\vert,
\end{align*}
using Lemma~\ref{lemma_H_tail_fraction}. Below we bound the latter supremum.
Then we return to the second term of $\widetilde{R}_{n}(x \vert s)$, which can be uniformly bounded on $[0,T_{s}^{(n)}]$ as follows
\begin{align*}
    &\sup_{x\leq T_{s}^{(n)}} \Big\vert \ \int_{0}^{x}\frac{\widetilde{\amsmathbb{H}}_{n}(y- \vert s)-\amsmathbb{H}_{n}(y- \vert s)}{(1-H_{n,\floor{ns}}(y-))^{2}}\widetilde{\amsmathbb{H}}^{1}_{n}(\mathrm{d} y \vert s)\Big\vert\\
    & \leq
    \sup_{x \leq T_{s}^{(n)}}\big\vert \widetilde{\amsmathbb{H}}^{1}_{n}(x \vert s)
     -
     \amsmathbb{H}^{1}_{n}(x \vert s) \big\vert
     \int_{0}^{T_{s}^{(n)}}\frac{1}{(1-H_{n,\floor{ns}}(y-))^{2}}\widetilde{\amsmathbb{H}}^{1}_{n}(\mathrm{d} y \vert s) \\
     & =
     \sup_{x \leq T_{s}^{(n)}}\big\vert \widetilde{\amsmathbb{H}}^{1}_{n}(x \vert s)
     -
     \amsmathbb{H}^{1}_{n}(x \vert s) \big\vert
     \mathcal{O}_{\amsmathbb{P}}(n/k_{n})\\
     &\leq
      \mathcal{O}_{\amsmathbb{P}}\Big(\log(n)(h_{n}^{2} \vee 1/\sqrt{n} \vee 1/(nh_{n}))\Big)\mathcal{O}_{\amsmathbb{P}}(n/k_{n})= 
        \mathcal{O}_{\amsmathbb{P}}\Big(\frac{n\log(n)}{k_{n}}(h_{n}^{2} \vee 1/\sqrt{n} \vee 1/(nh_{n}))\Big),
\end{align*}
using Lemma~\ref{lemma_H_n_and_H_n_1_unif_cons}, Lemma~\ref{lemma_integral_wrt_H_tilde_size_n_over_k}, Lemma~\ref{lemma_H_minus_H_tilde} with the fact that $H^{1}_{n,\floor{ns}}\leq H_{n,\floor{ns}}$ and the definition $T_{s}^{(n)}$. 
Moving on to $R_{n1}(\cdot\vert s)$, recall that
\begin{align*}
   &R_{n1}(x\vert s)
    =
    \frac{1}{n}\int_{0}^{x}\frac{1}{(1-H_{n,\floor{ns}}(y-))^{2}} 
    \Big( \frac{1}{nh_{n}^{2}}\sum_{i=1}^{n} 1\{Z_{i}^{(n)}\leq \mathrm{d} y, \delta_{i}^{(n)} = 1 \}K^{2}\Big(\frac{s - i/n}{h_{n}}\Big) \Big) \\
     & \quad - \Big(1- \frac{1}{nh_{n}}\sum_{i=1}^{n}K\Big(\frac{s-i/n}{h_{n}}\Big) \Big)
    \int_{0}^{x}\frac{1}{(1-H_{n,\floor{ns}}(y-))^{2}}\amsmathbb{H}^{1}_{n}(\mathrm{d} y \vert s) =: R_{n1}^{I}(x\vert s) - R_{n1}^{II}(x\vert s).
\end{align*}
It holds that 
\begin{align*}
    &\amsmathbb{E}\Big[ \int_{0}^{T_{s}^{(n)}}\frac{1}{(1-H_{n,\floor{ns}}(y-))^{2}}\amsmathbb{H}^{1}_{n}(\mathrm{d} y \vert s)\Big]\\
    &=
    \frac{1}{nh_{n}}\sum_{i=1}^{n}K\Big(\frac{s-i/n}{h_{n}}\Big)
    \int_{0}^{T_{s}^{(n)}}\frac{1}{(1-H_{n,\floor{ns}}(y-))^{2}}H_{n,i}^{1}(\mathrm{d} y) \\
    &=
     \frac{1}{nh_{n}}\sum_{i=1}^{n}K\Big(\frac{s-i/n}{h_{n}}\Big)
    \int_{0}^{T_{s}^{(n)}}\frac{1}{(1-H_{n,\floor{ns}}(y-))^{2}}H_{n,\floor{ns}}^{1}(\mathrm{d} y)\\
    & \quad+
     \frac{1}{nh_{n}}\sum_{i=1}^{n}K\Big(\frac{s-i/n}{h_{n}}\Big)
    \int_{0}^{T_{s}^{(n)}}\frac{1}{(1-H_{n,\floor{ns}}(y-))^{2}}(H_{n,i}^{1}(\mathrm{d} y) -H_{n,\floor{ns}}^{1}(\mathrm{d} y)) \\
    & \leq
    \frac{2n}{p_{s}k_{n}} \frac{1}{nh_{n}}\sum_{i=1}^{n}K\Big(\frac{s-i/n}{h_{n}}\Big)
    +
    \Big(\frac{n}{k_{n}}\Big)^{2}\Big\vert\frac{1}{nh_{n}}\sum_{i=1}^{n}K\Big(\frac{s-i/n}{h_{n}}\Big)(H_{n,i}^{1}(y) -H_{n,\floor{ns}}^{1}(y)) \Big\vert\\
    &=
    \mathcal{O}\Big(\frac{n}{k_{n}}\Big(1+\frac{nh_{n}^{2}}{k_{n}}\Big)\Big) =   \mathcal{O}\Big(\frac{n}{k_{n}}\Big), 
\end{align*}
by assumption, where we used~\eqref{eq_integral_bound}, the triangle inequality and the bias calculations of the proof of Lemma~\ref{lemma_H_n_and_H_n_1_unif_cons}.
Since the random integral in the first line above is positive, the argument of the Lemma~\ref{lemma_integral_wrt_H_tilde_size_n_over_k} can be repeated to conclude that it itself is an $\mathcal{O}_{\amsmathbb{P}}(n/k_{n})$ term. Hence
\begin{align*}
   &\vert R_{n1}^{I}(x\vert s) \vert
   =
   \frac{1}{n}\int_{0}^{x}\frac{1}{(1-H_{n,\floor{ns}}(y-))^{2}} 
    \Big( \frac{1}{nh_{n}^{2}}\sum_{i=1}^{n} 1\{Z_{i}^{(n)}\leq \mathrm{d} y, \delta_{i}^{(n)} = 1 \}K^{2}\Big(\frac{s - i/n}{h_{n}}\Big) \Big) \\
    & \leq
    \frac{M_{K}}{nh_{n}}
    \int_{0}^{x}\frac{1}{(1-H_{n,\floor{ns}}(y-))^{2}} 
    \amsmathbb{H}_{n}^{1}(\mathrm{d} y \vert s) = 
    \frac{M_{K}}{nh_{n}}
    \mathcal{O}_{\amsmathbb{P}}\Big(\frac{n}{k_{n}}\Big) = \mathcal{O}_{\amsmathbb{P}}\Big(\frac{1}{k_{n}h_{n}}\Big),
    \end{align*}
    using Lemma~\ref{lemma_H_n_and_H_n_1_unif_cons} in the second equality and~\eqref{eq_integral_bound} in the last inequality. In addition
\begin{align*}
    &\vert R_{n1}^{II}(x\vert s) \vert
   \leq
   \Big\vert 1- \frac{1}{nh_{n}}\sum_{i=1}^{n}K\Big(\frac{s-i/n}{h_{n}}\Big) \Big\vert
   \int_{0}^{T_{s}^{(n)}}\frac{1}{(1-H_{n,\floor{ns}}(y-))^{2}}\amsmathbb{H}^{1}_{n}(\mathrm{d} y \vert s) \\
   &\leq
   \mathcal{O}_{\amsmathbb{P}}\Big( \frac{1}{nh_{n}}\frac{n}{k_{n}} \Big)
   =
   \mathcal{O}_{\amsmathbb{P}}\Big(\frac{1}{k_{n}h_{n}}\Big),
\end{align*}
using similar arguments as above with the calculations from the proof of Lemma~\ref{lemma_H_n_and_H_n_1_unif_cons}. This concludes $R_{n1}(\cdot \vert s)$.
Let $\gamma \in [0,1/2)$ and $x \leq T_{s}^{(n)}$. Then 
\begin{align*}
    \left \vert R_{n2}(x \vert s) \right\vert
     &=
     \int_{0}^{x}
    \frac{(\amsmathbb{H}_{n}(y- \vert s) - H_{n,\floor{ns}}(y-))^{2}}{(1-H_{n,\floor{ns}}(y-))^{2}(1-\amsmathbb{H}_{n}(y- \vert s))}\amsmathbb{H}^{1}_{n}(\mathrm{d} y \vert s) \\
    &
    \leq  \sup_{x \leq T_{s}^{(n)}}\Big(\frac{\amsmathbb{H}_{n}(x- \vert s) - H_{n,\floor{ns}}(x-)}{(1-H_{n,\floor{ns}}(x-))^{1-\gamma}}\Big)^{2}
    \sup_{x \leq T_{s}^{(n)}}\Big(\frac{1- H_{n,\floor{ns}}(x-)}{1-\amsmathbb{H}_{n}(x- \vert s)}\Big) \\
    & \qquad \qquad \qquad \qquad \times 
    \int_{0}^{x}
    \frac{1}{(1-H_{n,\floor{ns}}(y-))^{1+2\gamma}}\amsmathbb{H}_{n}(\mathrm{d} y \vert s). 
\end{align*}
Lemma~\ref{lemma_H_minus_H_gamma_fraction_rate} yields that  
\begin{align*}
    &\sup_{x \leq T_{s}^{(n)}}\Big(\frac{\amsmathbb{H}_{n}(x- \vert s) - H_{n,\floor{ns}}(x-)}{(1-H_{n,\floor{ns}}(x-))^{1-\gamma}}\Big)^{2}=
   \mathcal{O}_{\amsmathbb{P}}\Big(
    \Big(\frac{n}{k_{n}}\Big)^{1-\gamma}(h_{n}^{2} 
    + 
    \frac{\sqrt{h_{n}}}{\sqrt{n}})
    +
    \frac{1}{\sqrt{nh_{n}}}\Big(\frac{n}{k_{n}}\Big)^{(1-2\gamma)/2}
    \Big),
\end{align*}
taking care of the first factor. As for the second factor, note that
\begin{align*}
   \sup_{x \leq Z_{n:n-k_{n}}}\Big\vert\frac{1-\amsmathbb{H}_{n}(x- \vert s)}{1-H_{n,\floor{ns}}(x-)} 
    -
    1\Big\vert
    =
    o_{\amsmathbb{P}}(1),
\end{align*}
using Lemma~\ref{lemma_H_minus_H_gamma_fraction_rate} with $\gamma = 0$. Consequently $(1-\amsmathbb{H}_{n}(x- \vert s))/(1-H_{n,\floor{ns}}(x-))$ is bounded away from zero in probability, uniformly on the random interval $[0,Z_{n:n-k_{n}}]$ which is contained in $[0,T_{s}^{(n)}]$ with probability tending to one. Consequently
\begin{align}\label{eq_inverse_tail_fraction}
    \sup_{x \leq T_{s}^{(n)}}\Big(\frac{1- H_{n,\floor{ns}}(x-)}{1-\amsmathbb{H}_{n}(x- \vert s)}\Big)
    =
    \sup_{x \leq T_{s}^{(n)}}1/\Big(\frac{1-\amsmathbb{H}_{n}(x- \vert s)}{1-H_{n,\floor{ns}}(x-)} \Big)
    =
    \mathcal{O}_{\amsmathbb{P}}(1).
\end{align}
This concludes the second factor.
Repeating the argument from the term $R_{n1}^{I}(x\vert s)$, the Markov inequality yields that 
the integral from the bound of $R_{n2}(x \vert s)$ is of order 
\begin{align*}
  \int_{0}^{x}
    \frac{1}{(1-H_{n,\floor{ns}}(y-))^{1+2\gamma}}\amsmathbb{H}_{n}(\mathrm{d} y \vert s)
    =
    \mathcal{O}_{\amsmathbb{P}}\Big(\Big(\frac{n}{k_{n}}\Big)^{2\gamma}\Big).
\end{align*} 
Gathering the above yields
\begin{align*}
    \left \vert R_{n2}(x \vert s) \right\vert
     &\leq 
      \mathcal{O}\Big(
    \Big(\frac{n}{k_{n}}\Big)^{1-\gamma}(h_{n}^{2} 
    + 
    \frac{\sqrt{h_{n}}}{\sqrt{n}})
    +
    \frac{1}{\sqrt{nh_{n}}}\Big(\frac{n}{k_{n}}\Big)^{(1-2\gamma)/2}
    \Big)^{2}\mathcal{O}_{\amsmathbb{P}}(1)  \mathcal{O}_{\amsmathbb{P}}\Big(\Big(\frac{n}{k_{n}}\Big)^{2\gamma}\Big) \\
    & =
    \mathcal{O}_{\amsmathbb{P}}\Big(\Big(\frac{n\log(n)}{k_{n}}(h_{n}^{2} \vee 1/\sqrt{n} \vee 1/(nh_{n}))\Big)^{2}\Big),
\end{align*}
where the right-hand side does not depend on $x$. This concludes $R_{n2}(x \vert s)$. Moving on to $R_{n3}(x \vert s)$, note that 
\begin{align*}
    &\frac{1}{(nh_{n})^{2}} \left \vert R_{n3}(x \vert s) \right\vert
    =
     \Big\vert \frac{1}{(nh_{n})^{2}}\sum_{i=1}^{n} K^{2}\Big(\frac{s-i/n}{h_{n}}\Big)\Big\{ \frac{1\{Z_{\floor{ns}}^{(n,i)} \leq x \}1\{\delta_{\floor{ns}}^{(n,i)} = 1\}}{1-H_{n,\floor{ns}}(Z_{\floor{ns}}^{(n,i)}-)}     \\
      & \qquad \qquad \qquad \qquad \qquad \qquad  +
     \int_{0}^{x} \frac{1\{Z_{\floor{ns}}^{(n,i)}\geq y\} - (1 - H_{n,\floor{ns}}(y -))}{(1-H_{n,\floor{ns}}(y-))^{2}}  H_{n,\floor{ns}}^{1}(\mathrm{d} y) \Big\} \Big\vert \\
     & \leq
     \frac{M_{K}}{nh_{n}} \int_{0}^{T_{s}^{(n)}} \frac{1}{1 - H_{n,\floor{ns}}(y -)} \widetilde{\amsmathbb{H}}_{n}^{1}(\mathrm{d} y \vert s) \\
     & \qquad +
     \frac{2}{nh_{n}} \sum_{i=1}^{n} K^{2}\Big(\frac{s-i/n}{h_{n}}\Big)
     \frac{1}{nh_{n}} \int_{0}^{T_{s}^{(n)}} \frac{1}{1 - H_{n,\floor{ns}}(y -)^{2}} H_{n,\floor{ns}}^{1}(\mathrm{d} y) \\
     &= 
      \frac{M_{K}}{nh_{n}}\mathcal{O}_{\amsmathbb{P}}\Big(\frac{n}{k_{n}}\Big) +
     \frac{\mathcal{O}(1)}{nh_{n}} \int_{0}^{T_{s}^{(n)}} \frac{1}{1 - H_{n,\floor{ns}}(y -)^{2}} H_{n,\floor{ns}}^{1}(\mathrm{d} y) 
     \leq
     \mathcal{O}_{\amsmathbb{P}}\Big(\frac{n}{k_{n}}\frac{1}{nh_{n}}\Big)
     = \mathcal{O} \Big( \frac{1}{k_{n}h_{n}}\Big),
\end{align*}
where we used the crude bound $\vert 1\{Z_{\floor{ns}}^{(n,i)}\geq y\} - (1 - H_{n,\floor{ns}}(y -)) \vert \leq 2$ in the first inequality, used Lemma~\ref{lemma_H_n_and_H_n_1_unif_cons} in the second equality and the fact that $H^{1}_{n,\floor{ns}} \leq H_{n,\floor{ns}}$ with~\eqref{eq_integral_bound} in the last line. This concludes $R_{n3}(x \vert s)$. In addition, we have that
\begin{align*}
   &\frac{1}{nh_{n}}\Big(1 -\frac{1}{nh_{n}}\sum_{j=1}^{n} K\Big(\frac{s-j/n}{h_{n}}\Big) \Big)\widehat{V}_{n}(x \vert s) \\
   &=
   \Big(1 -\frac{1}{nh_{n}}\sum_{j=1}^{n} K\Big(\frac{s-j/n}{h_{n}}\Big) \Big) \\
   & \quad \times \Big(\int_{0}^{x}
    \frac{1}{1-H_{n,\floor{ns}}(y-)}  \widetilde{\amsmathbb{H}}^{1}_{n}(\mathrm{d} y \vert s)
      +
     \int_{0}^{x} \frac{(1-\widetilde{\amsmathbb{H}}_{n}(y-)) - (1 - H_{n,\floor{ns}}(y -)))}{(1-H_{n,\floor{ns}}(y-))^{2}} H_{n,\floor{ns}}^{1}(\mathrm{d} y) \Big) \\
     & \leq
     \mathcal{O}\Big(\frac{1}{nh_{n}}\Big)
     \Big( \mathcal{O}_{\amsmathbb{P}}\Big(\frac{n}{k_{n}}\Big) +
     \sup_{x < T_{s}^{(n)}}\vert \widetilde{\amsmathbb{H}}_{n}(x)- H_{n,\floor{ns}}(x) \vert
     \int_{0}^{x} \frac{1}{(1-H_{n,\floor{ns}}(y-))^{2}} H_{n,\floor{ns}}^{1}(\mathrm{d} y) \Big) \\
     &\leq
      \mathcal{O}\Big(\frac{1}{nh_{n}}\Big)
     \Big( \mathcal{O}_{\amsmathbb{P}}\Big(\frac{n}{k_{n}}\Big) +
     \mathcal{O}_{\amsmathbb{P}}\Big(\frac{1}{\sqrt{nh_{n}}}\frac{n}{k_{n}}\Big) \Big)
    =
    \mathcal{O}_{\amsmathbb{P}}\Big(\frac{1}{k_{n}h_{n}}\Big),
\end{align*}
uniformly on $[0,T_{s}^{(n)}]$, where we used the calculations and statement of Lemma~\ref{lemma_H_n_and_H_n_1_unif_cons} in the first and second inequality with~\eqref{eq_integral_bound}. The term $2(\widetilde{S}_{n1}(x \vert s)  - S_{n1}(x \vert s) )$ can be bounded as follows
\begin{align*}
   &\sup_{x \leq T_{s}^{(n)}} \vert S_{n1}(x \vert s) - \widetilde{S}_{n1}(x \vert s) \vert 
   \leq
    \sup_{x \leq T_{s}^{(n)}}\big\vert \widetilde{\amsmathbb{H}}^{1}_{n}(x \vert s)
     -
     \amsmathbb{H}^{1}_{n}(x \vert s) \big\vert /(1-H_{n,\floor{ns}}(T_{s}^{(n)}-)) \\
     &= 
     \frac{n}{p_{s}k_{n}} \sup_{x \leq T_{s}^{(n)}}\big\vert \widetilde{\amsmathbb{H}}^{1}_{n}(x \vert s)
     -
     \amsmathbb{H}^{1}_{n}(x \vert s) \big\vert  
     \leq \frac{n}{p_{s}k_{n}} \mathcal{O}_{\amsmathbb{P}}\big(\sqrt{\log(n)} (h_{n}^{2} \vee 1/\sqrt{n} \vee 1/(nh_{n}))\big)\\
     &=  \mathcal{O}_{\amsmathbb{P}}\Big(\frac{n\log(n)}{k_{n}}(h_{n}^{2} \vee 1/\sqrt{n} \vee 1/(nh_{n}))\Big),
\end{align*}
using Lemma~\ref{lemma_H_minus_H_tilde}. Finally we turn to the last remainder term which is the local degenerate $U$-statistic process,
\begin{align*}
    \frac{1}{(nh_{n})^{2}}(U_{n}(x \vert s) - \widehat{U}_{n}(x \vert s)).
\end{align*}
Split both $U_{n}$ and $\widehat{U}_{n}$ into two sums corresponding to the indices $1 \leq i < j \leq n$ and $1 \leq j < i \leq n$. Then Theorem~\ref{thm_U_stat_repr_Xi_Yj_remainder_unif_bound} yields 
\begin{align*}
    \frac{1}{(nh_{n})^{2}} \vert U_{n}(x \vert s) - \widehat{U}_{n}(x \vert s) \vert \leq
     \frac{1}{(nh_{n})^{2}} \vert R_{s}^{(n)}(x, \infty ) \vert, 
\end{align*}
where 
\begin{align*}
   \amsmathbb{E}\Big[ \sup_{x \leq T_{s}^{(n)}} \vert R_{s}^{(n)}(x,\infty) \vert^{2} \Big]
   \leq (a_{s}^{(n)})^{2} \mathcal{O}(n^{2}h_{n}^{2}),
\end{align*}
and the constant $a_{s}^{(n)}$ satisfies
\begin{align*}
    a_{s}^{(n)}
    \leq
    \widetilde{A} \Big[\int_{0}^{\infty}\int_{0}^{T_{s}^{(n)}} \big\vert h_{n,\floor{ns}}(y,z) \big\vert^{2}H_{n,\floor{ns}}^{1}(\mathrm{d} y) H_{n,\floor{ns}}(\mathrm{d} z)\Big]^{1/2},
\end{align*}
with constant $\widetilde{A}$. The double integral in the above bracket can be rewritten and bounded as follows
\begin{align*}
    &\int_{0}^{\infty}\int_{0}^{T_{s}^{(n)}} \big\vert h_{n,\floor{ns}}(y,z) \big\vert^{2}H_{n,\floor{ns}}^{1}(\mathrm{d} y) H_{n,\floor{ns}}(\mathrm{d} z)\\
    &=
    \int_{0}^{\infty}\int_{0}^{T_{s}^{(n)} \wedge z} \frac{1}{(1-H_{n,\floor{ns}}(y-))^{4}} H_{n,\floor{ns}}^{1}(\mathrm{d} y) H_{n,\floor{ns}}(\mathrm{d} z) \\
    &=
    \int_{0}^{T_{s}^{(n)}} \int_{y}^{\infty}H_{n,\floor{ns}}(\mathrm{d} z) \frac{1}{(1-H_{n,\floor{ns}}(y-))^{4}} H_{n,\floor{ns}}^{1}(\mathrm{d} y) \\
    &=
    \int_{0}^{T_{s}^{(n)}} \frac{1}{(1-H_{n,\floor{ns}}(y-))^{3}} H_{n,\floor{ns}}^{1}(\mathrm{d} y) \\
    & \leq
    \int_{0}^{T_{s}^{(n)}} \frac{1}{(1-H_{n,\floor{ns}}(y-))^{3}} H_{n,\floor{ns}}(\mathrm{d} y)  \leq
    \frac{3}{2(p_{s}k_{n}/n)^{2}},
\end{align*}
using again~\eqref{eq_integral_bound}. Summarizing
\begin{align*}
    &\amsmathbb{E}\Big[\frac{1}{(nh_{n})^{2}} \sup_{x \leq T_{s}^{(n)}} \Big\vert  R_{s}^{(n)}(x,\infty) \Big\vert^{2} \Big]  
    \leq
    \frac{1}{n^{4}h_{n}^{4}} (a_{s}^{(n)})^{2} \mathcal{O}(n^{2}h_{n}^{2}) \\
    & =
    \mathcal{O}\Big(\frac{1}{n^{2}h_{n}^{2}}\Big)\widetilde{A}^{2}\int_{0}^{\infty}\int_{0}^{T_{s}^{(n)}} \big\vert h_{n,\floor{ns}}(y,z) \big\vert^{2}H_{n,\floor{ns}}^{1}(\mathrm{d} y) H_{n,\floor{ns}}(\mathrm{d} z)
    \leq
    \mathcal{O}\Big(\frac{1}{k_{n}^{2}h_{n}^{2}}\Big),
\end{align*}
using the above bound. Lemma~\ref{lemma_Z_n_k_n_eventually_contained} yields that the interval $[0, Z_{n:n-k_{n}}]$ can be covered by $[0, T_{s}^{(n)}]$ with probability going to one. This finishes the proof.
\end{proof}
\begin{proposition}\label{prop_array_repr_NA}
Assume Condition~\ref{cond_H_H1_C1_F_G_Lip}, and let $s \in (0,1)$.
 It then holds that 
\begin{align*}
      &\mathbb{\Lambda}^{n, *}(x\vert s) - \Lambda_{\floor{ns}}(x) =
      \frac{1}{nh_{n}}\sum_{i=1}^{n}K\Big(\frac{s - i/n}{h_{n}}\Big)l_{n,i}(x\vert s)
      +
      R_{n}^{*}(x \vert s),
\end{align*}
   where the latter sum is a normalized sum of independent, centered summands with
\begin{align*}
   &l_{n,i}(x\vert s) = 
    \int_{0}^{x} \frac{1\{ Z_{\floor{ns}}^{(n,i)} < y\} - H_{n,\floor{ns}}(y-)}{(1-H_{n,\floor{ns}}(y-))^{2}}  H_{n,\floor{ns}}^{1}(\mathrm{d} y) \\
    & \quad +
     \frac{1\{ Z_{\floor{ns}}^{(n,i)} \leq x\}\delta_{\floor{ns}}^{(n,i)} - H_{n,\floor{ns}}^{1}(x) }{1-H_{n,\floor{ns}}(x-)} 
   - 
   \int_{0}^{x}
\frac{1\{Z_{\floor{ns}}^{(n,i)} \leq y\}\delta_{\floor{ns}}^{(n,i)} - H_{n,\floor{ns}}^{1}(y)}{(1-H_{n,\floor{ns}}(y-))^{2}}H_{n,\floor{ns}}(\mathrm{d} y-),
\end{align*}
and where 
\begin{align*}
    \sup_{x \leq Z_{n:n-k_{n}}} \vert R_{n}^{*}(x \vert s) \vert = \mathcal{O}\Big(\frac{1}{k_{n}h_{n}}\Big),
\end{align*} 
\end{proposition}

\begin{proof}[Proof of Proposition~\ref{prop_array_repr_NA}]
Recall that
\begin{align*}
       &\mathbb{\Lambda}^{n, *}(x\vert s) =\int_{0}^{x}
    \frac{1}{1-H_{n,\floor{ns}}(y-)}  \widetilde{\amsmathbb{H}}^{1}_{n}(\mathrm{d} y \vert s)\\
     & \qquad -
     \int_{0}^{x} \frac{ (nh_{n})^{-1}\sum_{i=1}^{n}K((s-i/n)/h_{n})(1\{Z_{\floor{ns}}^{(n,i)}\geq y\} - (1 - H_{n,\floor{ns}}(y -)))}{(1-H_{n,\floor{ns}}(y-))^{}} H_{n,\floor{ns}}^{1}(\mathrm{d} y). 
\end{align*}
    Then 
    \begin{align*}
     &\mathbb{\Lambda}^{n, *}(x\vert s) - \Lambda_{\floor{ns}}(x) =  \int_{0}^{x}
    \frac{1}{1-H_{n,\floor{ns}}(y-)}  \big(\widetilde{\amsmathbb{H}}^{1}_{n}(\mathrm{d} y \vert s) - H^{1}_{n,\floor{ns}}(\mathrm{d} y) \big)
      \\
     & \qquad -
     \int_{0}^{x} \frac{ (nh_{n})^{-1}\sum_{i=1}^{n}K((s-i/n)/h_{n})(1\{Z_{\floor{ns}}^{(n,i)}\geq y\} - (1 - H_{n,\floor{ns}}(y -)))}{(1-H_{n,\floor{ns}}(y-))^{2}} H_{n,\floor{ns}}^{1}(\mathrm{d} y) \\
     &=
        \frac{1}{nh_{n}}\sum_{i=1}^{n}K\Big(\frac{s - i/n}{h_{n}}\Big) \Big\{
        \int_{0}^{x} \frac{1\{ Z_{\floor{ns}}^{(n,i)} < y\} - H_{n,\floor{ns}}(y-)}{(1-H_{n,\floor{ns}}(y-))^{2}}  H_{n,\floor{ns}}^{1}(\mathrm{d} y) \\
        & \quad +
         \frac{1\{ Z_{\floor{ns}}^{(n,i)} \leq x\}\delta_{\floor{ns}}^{(n,i)} - H_{n,\floor{ns}}^{1}(x) }{1-H_{n,\floor{ns}}(x-)} 
       - 
       \int_{0}^{x}
    \frac{1\{Z_{\floor{ns}}^{(n,i)} \leq y\}\delta_{\floor{ns}}^{(n,i)} - H_{n,\floor{ns}}^{1}(y)}{(1-H_{n,\floor{ns}}(y-))^{2}}H_{n,\floor{ns}}(\mathrm{d} y-)  \Big\} \\
        &+
        \Big(\frac{1}{nh_{n}}\sum_{i=1}^{n}K\Big(\frac{s - i/n}{h_{n}}\Big) -1\Big)
        \Big(\frac{H_{n,\floor{ns}}^{1}(x) }{1-H_{n,\floor{ns}}(x-)}
        -
        \int_{0}^{x}
    \frac{H_{n,\floor{ns}}^{1}(y)}{(1-H_{n,\floor{ns}}(y-))^{2}} 
   H_{n,\floor{ns}}(\mathrm{d} y-) 
        \Big) \\
        &=
         \frac{1}{nh_{n}}\sum_{i=1}^{n}K\Big(\frac{s - i/n}{h_{n}}\Big)l_{n,i}(x\vert s) \\
        &+
        \Big(\frac{1}{nh_{n}}\sum_{i=1}^{n}K\Big(\frac{s - i/n}{h_{n}}\Big) -1\Big)
        \Big(\frac{H_{n,\floor{ns}}^{1}(x) }{1-H_{n,\floor{ns}}(x-)}
        -
        \int_{0}^{x}
    \frac{H_{n,\floor{ns}}^{1}(y)}{(1-H_{n,\floor{ns}}(y-))^{2}} 
   H_{n,\floor{ns}}(\mathrm{d} y-) 
        \Big),
    \end{align*}
    where the latter term is bounded by $ \mathcal{O}(1/(k_{n}h_{n}))$ since,
    \begin{align*}
         \Big(\frac{1}{nh_{n}}\sum_{i=1}^{n}K\Big(\frac{s - i/n}{h_{n}}\Big) -1\Big)
         =
         \mathcal{O}\Big(\frac{1}{nh_{n}}\Big) 
    \end{align*}
    and on $[0, T_{s}^{(n)}]$, we have
    \begin{align*}
        &\Big\vert\frac{H_{n,\floor{ns}}^{1}(x) }{1-H_{n,\floor{ns}}(x-)}
        -
        \int_{0}^{x}
    \frac{H_{n,\floor{ns}}^{1}(y)}{(1-H_{n,\floor{ns}}(y-))^{2}} 
   H_{n,\floor{ns}}(\mathrm{d} y-) 
        \Big\vert \\
        & \leq
        \frac{1}{1-H_{n,\floor{ns}}(T_{s}^{(n)}-)}
        +
        \int_{0}^{T_{s}^{(n)}}
    \frac{1}{(1-H_{n,\floor{ns}}(y-))^{2}} 
   H_{n,\floor{ns}}(\mathrm{d} y-) \leq
     \mathcal{O}\Big(\frac{n}{k_{n}}\Big), 
    \end{align*}
    by definition of $T_{s}^{(n)}$ and~\eqref{eq_integral_bound}.
    This yields that 
    \begin{align*}
        \sup_{x \leq T_{s}^{(n)}} \vert R_{n}^{*}(x \vert s) \vert = \mathcal{O}\Big(\frac{1}{k_{n}h_{n}}\Big).
    \end{align*} 
    Lemma~\ref{lemma_Z_n_k_n_eventually_contained} yields that $[0,Z_{n:n-k_{n}}]$ is contained in $[0, T_{s}^{(n)}]$ with probability going to one, which concludes the proof.
\end{proof}

\begin{proof}[Proof of Theorem~\ref{thm_Nelson_Aalen_decomposition}]
    The proof follows directly from Theorem~\ref{thm_NA_weak_repr} and Proposition~\ref{prop_array_repr_NA}.
\end{proof}

\subsection{Proof of Theorem~\ref{thm_Lambda_consistency_speed} (Consistency rate; Local Nelson--Aalen estimator)}\label{appendix_proof_of_nelson_aalen_consistency_rate}
Let $n\in \amsmathbb{N}$ and $s \in (0,1)$ be fixed and define the $\sigma$-algebra
\begin{align*}
    \mathcal{F}_{x}^{(n)} = 
    \sigma(\{ (1\{Z_{\floor{ns}}^{(n,i)}\leq x\}, 1\{Z_{\floor{ns}}^{(n,i)}\leq x\}\delta_{\floor{ns}}^{(n,i)},
    1\{Z_{\floor{ns}}^{(n,i)}\leq x\}Z_{\floor{ns}}^{(n,i)}), i=1,\ldots,n\}).
\end{align*}
    If we define 
\begin{align*}
    M_{n}(x \vert s) :=
   \frac{1}{nh_{n}} \sum_{i=1}^{n}  K\Big(\frac{s - i/n}{h_{n}}\Big) \Big[ 1\{Z_{\floor{ns}}^{(n,i)}\leq x\}\delta_{\floor{ns}}^{(n,i)} - \int_{0}^{x}1\{Z_{\floor{ns}}^{(n,i)}\geq y\} \Lambda_{\floor{ns}}(\mathrm{d} y) \Big],
\end{align*}
then $\{(M_{n}(x\vert s), \mathcal{F}_{x}^{(n)}): x \in [0, x^{*})\}$ is a square-integrable martingale with predictable variation process
\begin{align*}
    \langle M_{x} \rangle (x\vert s) = 
    \frac{1}{(nh_{n})^{2}} \sum_{i=1}^{n} K^{2}\Big(\frac{s - i/n}{h_{n}}\Big) \int_{0}^{x}1\{Z_{\floor{ns}}^{(n,i)}\geq y\} \Lambda_{\floor{ns}}(\mathrm{d} y).
\end{align*}
With the above notation, note that
\begin{align*}
     &\mathbb{\Lambda}^{n, *}(x\vert s) - \Lambda_{\floor{ns}}(x) =  \int_{0}^{x}\frac{1}{1-H_{n,\floor{ns}}(y-)} M_{n}(\mathrm{d} y\vert s) \\
    &+      
    \Big(\frac{1}{nh_{n}}\sum_{i=1}^{n}K\Big(\frac{s - i/n}{h_{n}}\Big) -1\Big)
        \int_{0}^{x}
    \frac{H_{n,\floor{ns}}^{1}(y)}{(1-H_{n,\floor{ns}}(y-))^{2}} 
   H_{n,\floor{ns}}(\mathrm{d} y-) \\
    & =
     \int_{0}^{x}\frac{1}{1-H_{n,\floor{ns}}(y-)} M_{n}(\mathrm{d} y\vert s)  
     +
     \mathcal{O}\Big(\frac{1}{k_{n}h_{n}}\Big)
     =: L_{n}(x\vert s) + \mathcal{O}\Big(\frac{1}{k_{n}h_{n}}\Big).
\end{align*}
where $L_{n}(x\vert s)$ is also a square-integrable martingale with respect to $\mathcal{F}_{x}^{(n)}$ for every fixed $n \in \amsmathbb{N}$ and $s \in (0,1)$. Its predictable variation process is given by
\begin{align*}
    \langle L_{n} \rangle (x\vert s)
    =
    \int_{0}^{x}\frac{1}{(1-H_{n,\floor{ns}}(y-))^{2}} \langle M_{n} \rangle (\mathrm{d} y\vert s).
\end{align*}
From Theorem~\ref{thm_NA_weak_repr}, we have that $ \mathbb{\Lambda}^{n}(\cdot\vert s) = \mathbb{\Lambda}^{n, *}(\cdot\vert s) + R_{n}^{0}(\cdot \vert s)$, where $R_{n}^{0}(\cdot \vert s)$ satisfies~\eqref{eq_error_term_rate_NA_repr}.
Theorem~\ref{thm_NA_weak_repr} yields that 
\begin{align*}
     &\sup_{x \leq Z_{n:n-k_{n}}} \vert \mathbb{\Lambda}^{n}(x\vert s) - \Lambda_{\floor{ns}}(x) \vert \\
     &\leq 
     \sup_{x \leq Z_{n:n-k_{n}}} \vert \mathbb{\Lambda}^{n}(x\vert s) - \mathbb{\Lambda}^{n, *}(x\vert s) \vert  
     +
     \sup_{x \leq Z_{n:n-k_{n}}} \vert  \mathbb{\Lambda}^{n, *}(x\vert s) - \Lambda_{\floor{ns}}(x) \vert  
     \\
     &=
      \sup_{x \leq Z_{n:n-k_{n}}} \vert   L_{n}(x \vert s) \vert  + 
      \mathcal{O}_{\amsmathbb{P}}\big(\sqrt{\log(n)} (h_{n}^{2} \vee 1/\sqrt{n} \vee 1/(nh_{n}))\big)\\
     &
      =
      \sup_{x \leq Z_{n:n-k_{n}}} \vert   L_{n}(x \vert s) \vert  + 
      o_{\amsmathbb{P}}\Big(\frac{1}{\sqrt{k_{n}h_{n}}}\Big),
\end{align*}
by~\eqref{eq_n_log_n_h2_over_kn_little_o}.
For fixed $n \in \amsmathbb{N}$ and $s \in (0,1)$, an application of Lenglart's inequality yields that for any $\varepsilon>0$ and $y > 0$,
\begin{align*}
    \amsmathbb{P}\big(\sup_{x \leq T_{s}^{(n)}} \vert  L_{n}(x \vert s) \vert \geq \varepsilon/\sqrt{k_{n}h_{n}}\big)
    \leq
    \frac{yk_{n}h_{n}}{\varepsilon^{2}} +  \amsmathbb{P}\big( \langle L_{n} \rangle (T_{s}^{(n)} \vert s) \geq y\big).
\end{align*}
We immediately note that 
\begin{align*}
   &\langle L_{n} \rangle (T_{s}^{(n)} \vert s)
    =
    \int_{0}^{T_{s}^{(n)}}\frac{1}{(1-H_{n,\floor{ns}}(y-))^{2}} \langle M_{n} \rangle (\mathrm{d} y\vert s) \\
    &=
     \frac{1}{(nh_{n})^{2}} \sum_{i=1}^{n} K^{2}\Big(\frac{s - i/n}{h_{n}}\Big) 
     \int_{0}^{T_{s}^{(n)}}\frac{1\{Z_{\floor{ns}}^{(n,i)}\geq y\}}{(1-H_{n,\floor{ns}}(y-))^{2}} \Lambda_{\floor{ns}}(\mathrm{d} y) \\
     & \leq
      \frac{M_{K}}{nh_{n}} 
     \int_{0}^{T_{s}^{(n)}}\frac{1-\widetilde{\amsmathbb{H}}_{n}(y- \vert s)}{(1-H_{n,\floor{ns}}(y-))^{2}} \Lambda_{\floor{ns}}(\mathrm{d} y)
     =
      \frac{M_{K}}{nh_{n}} 
     \int_{0}^{T_{s}^{(n)}}\frac{1-\widetilde{\amsmathbb{H}}_{n}(y- \vert s)}{(1-H_{n,\floor{ns}}(y-))^{3}} H_{n,\floor{ns}}^{1}(\mathrm{d} y) \\
     & \leq
     \frac{M_{K}}{nh_{n}} 
     \sup_{x \leq T_{s}^{(n)}} \Big\vert  \frac{1-\widetilde{\amsmathbb{H}}_{n}(x- \vert s)}{1-H_{n,\floor{ns}}(x-)} \Big\vert
     \int_{0}^{T_{s}^{(n)}}\frac{1}{(1-H_{n,\floor{ns}}(y-))^{2}} H_{n,\floor{ns}}(\mathrm{d} y) \\
     & \leq
     \frac{M_{K}}{p_{s}k_{n}h_{n}} 
     \sup_{x \leq T_{s}^{(n)}}\Big\vert \frac{1-\widetilde{\amsmathbb{H}}_{n}(x- \vert s)}{1-H_{n,\floor{ns}}(x-)} \Big\vert =: \frac{M_{K}}{p_{s}k_{n}h_{n}} J_{s}^{(n)},
\end{align*}
using \eqref{eq_integral_bound}. Decomposing the probability statement of interest yields
\begin{align*}
    &\amsmathbb{P}\big(\sup_{x \leq Z_{n:n-k_{n}}} \vert  L_{n}(x \vert s) \vert \geq \varepsilon/\sqrt{k_{n}h_{n}}\big)\\
    &\leq
    \amsmathbb{P}\big(\sup_{x \leq T_{s}^{(n)}} \vert  L_{n}(x \vert s) \vert \geq \varepsilon/\sqrt{k_{n}h_{n}}\big) 
    +
    \amsmathbb{P}\big(T_{s}^{(n)} \leq Z_{n:n-k_{n}} \big).
\end{align*}
Since $p_{s}$ was chosen such that~\eqref{eq_Z_n_k_n_eventually_contained} holds, the latter probability is asymptotically negligible. In addition, Lemma~\ref{lemma_H_tilde_tail_fraction} yields that $J_{s}^{(n)}\overset{\amsmathbb{P}}{\to} 1$, and hence, for $y = CM_{K}/(p_{s}k_{n}h_{n})$, we have that
\begin{align*}
    &\amsmathbb{P}\big(\sup_{x \leq Z_{n:n-k_{n}}} \vert  L_{n}(x \vert s) \vert \geq \varepsilon/\sqrt{k_{n}h_{n}}\big) \leq
   \frac{CM_{K}}{p_{s}\varepsilon^{2}} +
     \amsmathbb{P}\big(J_{s}^{(n)} \geq C \big) + o(1),
\end{align*}
for every $\varepsilon > 0$ and $C > 0$. Then
\begin{align*}
    \lim_{\varepsilon \to \infty}\limsup_{n \to \infty} \Big( \frac{CM_{K}}{p_{s}\varepsilon^{2}} +
     \amsmathbb{P}\big(J_{s}^{(n)} \geq C \big) \Big) 
     =
     \limsup_{n\to \infty} \amsmathbb{P}\big(J_{s}^{(n)} \geq C \big),
\end{align*}
where the limit $\varepsilon \to \infty$ can be chosen arbitrarily slowly. Letting also $C \to \infty$ makes the latter limit go to zero, which concludes the proof.\\

\subsection{Proof of Proposition~\ref{prop_Beran_consistency_speed} (Consistency rate; Beran estimator)}\label{appendix_proof_of_beran_consistency_rate}
The proof is based on the proof of the corresponding unconditional counterpart in~\cite{Csorgo_96} Theorem 2. 
Setting $c_{n}(x \vert s) = -\log(1-\amsmathbb{F}^{(n)}(x \vert s)) - \mathbb{\Lambda}^{n}(x \vert s)$, for $x < Z_{n:n}$, the continuity of $F_{n,\floor{ns}}$, with the Breslow-Crowley expansion (see~\cite{Breslow}) yields
\begin{align*}
    \Big\vert \frac{\amsmathbb{F}^{(n)}(x \vert s) - F_{n,\floor{ns}}(x)}{1-F_{n,\floor{ns}}(x)} - h(x) \Big\vert
    \leq
    \vert (\mathbb{\Lambda}^{n}(x\vert s) - \Lambda_{\floor{ns}}(x)) - h(x) \vert
    + \vert R_{n4}(x \vert s) \vert,
\end{align*}
for any function $h: \amsmathbb{R} \mapsto \amsmathbb{R}$, where
\begin{align*}
    R_{n4}(x \vert s) = \exp(\vert \mathbb{\Lambda}^{n}(x\vert s) - \Lambda_{\floor{ns}}(x) \vert)\big[\vert\mathbb{\Lambda}^{n}(x\vert s) - \Lambda_{\floor{ns}}(x) \vert^{2}/2 + \vert c_{n}(x \vert s) \exp(\vert c_{n}(x \vert s)) \vert\big]. 
\end{align*}
Since we have no ties, we note that $\Delta \mathbb{\Lambda}^{n}(x\vert s) \leq 1/2$ whenever $x \leq \amsmathbb{H}^{\leftarrow}_{n}(1- 2/(nh_{n}) \vert s)$.
On the latter event, it holds that
\begin{align*}
   \vert c_{n}(x \vert s) \vert &= \vert -\log(1-\amsmathbb{F}^{(n)}(x \vert s)) - \mathbb{\Lambda}^{n}(x \vert s)\vert \\
   &=
   \Big\vert \sum_{0 \leq y \leq x}\big[-\log(1-\Delta \mathbb{\Lambda}^{n}(y\vert s)) -\Delta \mathbb{\Lambda}^{n}(y\vert s) \big] \Big\vert
   \leq
   2\sum_{0 \leq y \leq x} (\Delta \mathbb{\Lambda}^{n}(y\vert s))^{2},
\end{align*}
where we used that $-\log(1-u) -u \leq 2u^{2}$ for $u \leq 1/2$. Since for every $y$
\begin{align*}
    (\Delta \mathbb{\Lambda}^{n}(y\vert s))^{2}
    =
    \frac{(\Delta \amsmathbb{H}_{n}^{1}(y \vert s))^{2}}{(1-\amsmathbb{H}_{n}(y- \vert s))^{2}}
    \leq
    \frac{1}{nh_{n}} \frac{\Delta \amsmathbb{H}_{n}^{1}(y \vert s)}{(1-\amsmathbb{H}_{n}(y- \vert s))^{2}},
\end{align*}
it holds that
\begin{align*}
      2\sum_{0 \leq y \leq x} (\Delta \mathbb{\Lambda}^{n}(y\vert s))^{2}
      \leq
      \frac{2}{nh_{n}}\sum_{0 \leq y \leq x} \frac{\Delta \amsmathbb{H}_{n}^{1}(y \vert s)}{(1-\amsmathbb{H}_{n}(y- \vert s))^{2}}
      =
       \frac{2}{nh_{n}} \int_{0}^{x}\frac{1}{(1-\amsmathbb{H}_{n}(y- \vert s))^{2}} \amsmathbb{H}_{n}^{1}(\mathrm{d} y \vert s).
\end{align*}
Define the event
\begin{align*}
    \Omega_{s}^{(n)}(p_{s}):= \{H_{n,\floor{ns}}^{\leftarrow}(1-p_{s}k_{n}/n) \leq \amsmathbb{H}^{\leftarrow}_{n}(1- 2/(nh_{n}) \vert s)\} = \{T_{s}^{(n)} \leq \amsmathbb{H}^{\leftarrow}_{n}(1- 2/(nh_{n}) \vert s)\}.
\end{align*}
Since $p_{s} \in (0,1)$ it holds that $p_{s}k_{n}h_{n}/2 > 1$ for $n$ sufficiently large and consequently $\amsmathbb{P}(\Omega_{s}^{(n)}(p_{s})) \to 1$, due to Lemma~\ref{lemma_H_n_and_H_n_1_unif_cons} and Vervaat's Lemma. On $\Omega_{s}^{(n)}(p_{s})$, note that for $x \leq T_{s}^{(n)}$, we have
\begin{align*}
    &\vert c_{n}(x \vert s) \vert \leq 
    \frac{2}{nh_{n}} \int_{0}^{T_{s}^{(n)}}\frac{1}{(1-\amsmathbb{H}_{n}(y- \vert s))^{2}} \amsmathbb{H}_{n}^{1}(\mathrm{d} y \vert s) \\
    &\leq
    \sup_{y \leq T_{s}^{(n)}} 
    \frac{(1 - H_{n,\floor{ns}}(y-))^{2}}{(1-\amsmathbb{H}_{n}(y- \vert s))^{2}} 
    \frac{2}{nh_{n}}
    \int_{0}^{T_{s}^{(n)}} \frac{1}{(1 - H_{n,\floor{ns}}(y-))^{2}} \amsmathbb{H}_{n}^{1}(\mathrm{d} y \vert s)\\
    & \leq
     \sup_{y \leq T_{s}^{(n)}} \Big(
    \frac{1 - H_{n,\floor{ns}}(y-)}{1-\amsmathbb{H}_{n}(y- \vert s)} \Big)^{2}
    \frac{2}{nh_{n}}
    \mathcal{O}_{\amsmathbb{P}}\Big(\frac{n}{k_{n}}\Big)
    =
    \mathcal{O}_{\amsmathbb{P}}\Big(\frac{1}{k_{n}h_{n}}\Big),
\end{align*}
using Markov's inequality as in bound of $R_{n2}(x \vert s)$, as well as ~\eqref{eq_integral_bound} and~\eqref{eq_inverse_tail_fraction}.
Since 
\begin{align*}
    \amsmathbb{P}\big(Z_{n:n-k_{n}} \leq T_{s}^{(n)} \big)
    \to 1,
    \qquad 
    \amsmathbb{P}(\Omega_{s}^{(n)}(p_{s})) \to 1,
\end{align*}
    we conclude that
    \begin{align*}
        \sup_{x \leq Z_{n:n-k_{n}}} \vert c_{n}(x \vert s) \vert \leq  \mathcal{O}_{\amsmathbb{P}}\Big(\frac{1}{k_{n}h_{n}}\Big).
    \end{align*}
    In particular
    \begin{align}\label{eq_NA_KM_relation_plus_remainder}
         &\sup_{x \leq Z_{n:n-k_{n}}} \Big\vert \frac{\amsmathbb{F}^{(n)}(x \vert s) - F_{n,\floor{ns}}(x)}{1-F_{n,\floor{ns}}(x)} - h(x) \Big\vert \nonumber \\
         &\leq
         \sup_{x \leq Z_{n:n-k_{n}}} \vert (\mathbb{\Lambda}^{n}(x\vert s) - \Lambda_{\floor{ns}}(x)) - h(x) \vert  +   \sup_{x \leq Z_{n:n-k_{n}}} \vert R_{n4}(x \vert s) \vert,
    \end{align}
    where 
    \begin{align*}
          \sup_{x \leq Z_{n:n-k_{n}}} \vert R_{n4}(x \vert s) \vert = \mathcal{O}_{\amsmathbb{P}}\Big(\frac{1}{k_{n}h_{n}}\Big).
    \end{align*}
    Choosing $h(\cdot) = 0$ yields the statement of the proposition.

\subsection{Proof of Lemma~\ref{lemma_kn_hn_sequences} (Admissible tuning sequences)}\label{appendix_proof_of_admissible_tuning_sequences}
\begin{proof}[Proof of Lemma~\ref{lemma_kn_hn_sequences}]
The choices $k_{n} \asymp n^{\delta}$ where $\delta \in (1/2,1)$ and $h_{n} \asymp n^{-e}$ for $e\in \big((1-\delta)\vee (2-\delta)/5\big), \delta)$, yields that $nh_{n} \asymp n^{1-e} \to \infty$ and $k_{n}h_{n} \asymp n^{\delta -e} \to \infty$, which is the first statement of the Theorem. For the second statement we have $h_{n} \asymp n^{-e} \to 0$, secondly $k_{n}/n \asymp n^{-(1-\delta)} \to 0$ and finally $nh_{n}^{5} \asymp n^{1-5e} \to 0$ holds, since $e>(2-\delta)/5$. The statement
\begin{align*}
     \frac{n\log(n)}{k_{n}}h_{n}^{2}
        = o\Big(\frac{1}{\sqrt{k_{n}h_{n}}}\Big),
\end{align*}
holds, if
\begin{align*}
    \log(n)\frac{n\sqrt{h_{n}^{5}}}{\sqrt{k_{n}}}
    \asymp
    \log(n)n^{1-\delta/2 - 5e/2}
    =
    o(1),
\end{align*}
which in turn is true, whenever $2-\delta-5e < 0$. This corresponds to the requirement that $e>(2-\delta)/5$.
Likewise the statement
\begin{align*}
     \frac{n\log(n)}{k_{n}}\frac{1}{\sqrt{n}}
        = o\Big(\frac{1}{\sqrt{k_{n}h_{n}}}\Big),
\end{align*}
holds, if
\begin{align*}
    \log(n)^{2}\frac{nh_{n}}{k_{n}}
    \asymp
    \log(n)^{2}n^{1-\delta -e}
    =
    o(1). 
\end{align*}
The latter statement is true since $e>(1-\delta)$.
Finally we note that the choice of $\delta \in (1/2,1)$ ensures that $\big((1-\delta)\vee (2-\delta)/5\big), \delta)$ is non-empty. This concludes the proof.
\end{proof}

\subsection{Proof of Theorem~\ref{thm_c_F_consistency} (Consistency of scedasis estimator)}\label{appendix_proof_of_scedasis_consistency}
\begin{proof}
    Let $s \in (0,1)$ and $s_{0}\in (0,1)$ be given and consider the following factorization of $ \hat{c}_{F}(s \vert s_{0})$
   \begin{align*}
        \hat{c}_{F}(s \vert s_{0})
        =
        \frac{1-\amsmathbb{F}^{(n)}(Z_{n:n-k_{n}} \vert s)}{1-F(Z_{n:n-k_{n}})}
        \Big( \frac{1-\amsmathbb{F}^{(n)}(Z_{n:n-k_{n}} \vert s_{0})}{1-F(Z_{n:n-k_{n}})}\Big)^{-1}.
   \end{align*}
   Note that 
   \begin{align*}
       \frac{1-\amsmathbb{F}^{(n)}(Z_{n:n-k_{n}} \vert s)}{1-F(Z_{n:n-k_{n}})}
       &=
       \frac{1-F_{n,\floor{ns}}(Z_{n:n-k_{n}})}{1-F(Z_{n:n-k_{n}})}
       \frac{1-\amsmathbb{F}^{(n)}(Z_{n:n-k_{n}} \vert s)}{1-F_{n,\floor{ns}}(Z_{n:n-k_{n}})} \\
       &=
       \frac{1-F_{n,\floor{ns}}(Z_{n:n-k_{n}})}{1-F(Z_{n:n-k_{n}})}
       \Big(\frac{F_{n,\floor{ns}}(Z_{n:n-k_{n}}) - \amsmathbb{F}^{(n)}(Z_{n:n-k_{n}} \vert s)}{1-F_{n,\floor{ns}}(Z_{n:n-k_{n}})} 
       +
       1\Big) \\
       &=
        \frac{1-F_{n,\floor{ns}}(Z_{n:n-k_{n}})}{1-F(Z_{n:n-k_{n}})}
       (o_{\amsmathbb{P}}(1) +1),
   \end{align*}
   due to \eqref{eq_local_beran_tail_negl_in_Z}. Using~\eqref{eq_Scedasis_F_and_G}, we have that the latter fraction is
   \begin{align*}
       \frac{1-F_{n,\floor{ns}}(Z_{n:n-k_{n}})}{1-F(Z_{n:n-k_{n}})}
       &=
       c_{F}\Big(\frac{\floor{ns}}{n}\Big)(1+o_{\amsmathbb{P}}(1)) \\
       &=
       c_{F}(s)(1+o_{\amsmathbb{P}}(1))
       +
       \Big(c_{F}\Big(\frac{\floor{ns}}{n}\Big)-c_{F}(s)\Big)(1+o_{\amsmathbb{P}}(1)),
   \end{align*}
   where the first of these terms converges to $c_{F}(s)$, while the second term is asymptotically negligible, since 
   \begin{align*}
        \Big\vert c_{F}\Big(\frac{\floor{ns}}{n}\Big)-c_{F}(s)\Big\vert
        \leq
        \text{cst}\frac{\vert \floor{ns}-ns \vert}{n}
        \leq
        \text{cst}\frac{1}{n},
   \end{align*}
   using that $s \mapsto c_{F}(s)$ is continuous. 
   Consequently,
   \begin{align*}
       \frac{1-\amsmathbb{F}^{(n)}(Z_{n:n-k_{n}} \vert s)}{1-F(Z_{n:n-k_{n}})}
        \overset{\amsmathbb{P}}{\to}
        c_{F}(s),
   \end{align*}
   and hence
   \begin{align*}
        \hat{c}_{F}(s \vert s_{0})
        =
        \frac{1-\amsmathbb{F}^{(n)}(Z_{n:n-k_{n}} \vert s)}{1-F(Z_{n:n-k_{n}})}
        \Big( \frac{1-\amsmathbb{F}^{(n)}(Z_{n:n-k_{n}} \vert s_{0})}{1-F(Z_{n:n-k_{n}})}\Big)^{-1}
        \overset{\amsmathbb{P}}{\to}
        c_{F}(s)(c_{F}(s_{0}))^{-1},
   \end{align*}
   due to the Continuous Mapping Theorem. This concludes the proof.
\end{proof}

\subsection{Uniform bound on local degenerate \texorpdfstring{$U$-statistic}{U-statistic} process}\label{appendix_uniform_bound_on_degenerate_u_statistics_process}
In Appendix~\ref{appendix_subs_local_nelson_aalen_decomposition} we decompose the local Nelson--Aalen estimator $\mathbb{\Lambda}^{n}(\cdot\vert s)$ into two leading terms and several remainder terms. One of these remainders is the degenerate $U$-statistic process $(U_{n}(\cdot \vert s) - \widehat{U}_{n}(\cdot \vert s))/(nh_{n})^{2}$ and this remainder is by far the most challenging: it necessitates extending Theorem 1.5 in~\cite{U_statistics_Stute_94} to local estimators based on locally transported observations. This appendix proves this extension following closely the proof techniques of their paper.
Recall the notation $K_{n,s}(i,j) := K((s -j/n)/h_{n})K((s -i/n)/h_{n})$ and consider a $U$-statistic process of the form
\begin{align*}
    U_{n}(u,v \vert s)
    &=
    \sum_{i \neq j}^{n} h_{n,\floor{ns}}(F_{n,\floor{ns}}^{\leftarrow}(U_{i}), F_{n,\floor{ns}}^{\leftarrow}(U_{j})) \\
    & \qquad \qquad \times 
    1\{F_{n,\floor{ns}}^{\leftarrow}(U_{i}) \leq u, F_{n,\floor{ns}}^{\leftarrow}(U_{j}) \leq v\} 
    K_{n,s}(i,j) \\
    &=
    \sum_{1\leq i < j \leq n} h_{n,\floor{ns}}(F_{n,\floor{ns}}^{\leftarrow}(U_{i}), F_{n,\floor{ns}}^{\leftarrow}(U_{j})) \\
    & \qquad \qquad \times 
    1\{F_{n,\floor{ns}}^{\leftarrow}(U_{i}) \leq u, F_{n,\floor{ns}}^{\leftarrow}(U_{j}) \leq v\} 
    K_{n,s}(i,j) \\
     & \quad+
   \sum_{1\leq j < i \leq n}h_{n,\floor{ns}}(F_{n,\floor{ns}}^{\leftarrow}(U_{i}), F_{n,\floor{ns}}^{\leftarrow}(U_{j})) \\
    & \qquad \qquad \times 
    1\{F_{n,\floor{ns}}^{\leftarrow}(U_{i}) \leq u, F_{n,\floor{ns}}^{\leftarrow}(U_{j}) \leq v\} 
    K_{n,s}(i,j)\\
    & =: I_{n}(u,v \vert s) +  II_{n}(u,v \vert s),
\end{align*}
where $h_{n,\floor{ns}}$ is left unspecified for now. Define $\widetilde{h}_{s}^{(n)}(u,v) := h_{n,\floor{ns}}(F_{n,\floor{ns}}^{\leftarrow}(u), F_{n,\floor{ns}}^{\leftarrow}(v))$ and note that then 
 \begin{align*}
    I_{n}(u,v \vert s)
    &=\sum_{1\leq i < j \leq n} 
     \widetilde{h}^{(n)}_{s}(U_{i}, U_{j})1\{U_{i} \leq F_{n,\floor{ns}}(u),U_{j} \leq F_{n,\floor{ns}}(v)\} K_{n,s}(i,j),
\end{align*}
for $ u,v \in [0, x^{\ast})$. This process is the analogue of $I_{n}(u,v)$ from the proof section in~\cite{U_statistics_Stute_94}; the only differences are that $\widetilde{h}^{(n)}_{s}$ and its time change $F_{n,\floor{ns}}$ depend on the sample size $n$ and the fixed design point $s$, and that we multiply by the kernel product $K_{n,s}(i,j)$. The process $I_{n}(\cdot, \cdot \vert s)$ can be decomposed into its Hájek projection plus four remainder terms.
\begin{align*}
      &I_{n}(u,v \vert s)
      =
       \sum_{1\leq i < j \leq n}K_{n,s}(i,j)\Big[
       \int_{0}^{F_{n,\floor{ns}}(u)}\widetilde{h}^{(n)}_{s}(w, U_{j}) 1\{U_{j} \leq F_{n,\floor{ns}}(v)\} \mathrm{d} w\\
       & \  +
       \int_{0}^{F_{n,\floor{ns}}(v)}\widetilde{h}^{(n)}_{s}(U_{i}, t) 1\{U_{i} \leq F_{n,\floor{ns}}(u)\} \mathrm{d} t
       -
        \int_{0}^{F_{n,\floor{ns}}(u)}\int_{0}^{F_{n,\floor{ns}}(v)}
        \widetilde{h}^{(n)}_{s}(w,t)\mathrm{d} w \mathrm{d} t
       \Big] \\
       & \quad + 
       \beta_{s}^{(n)}(u,v)
       -
       \gamma_{s}^{(n)}(u,v)
       -
       \delta_{s}^{(n)}(u,v)
       +
       \varepsilon_{s}^{(n)}(u,v), 
\end{align*}
where
\begin{align*}
     &\beta_{s}^{(n)}(u,v)
     =
      \sum_{1\leq i < j \leq n} \beta_{s, (i,j)}^{(n)}(u,v)K_{n,s}(i,j), 
      \qquad
    \gamma_{s}^{(n)}(u,v)
     =
      \sum_{1\leq i < j \leq n} \gamma_{s, (i,j)}^{(n)}(u,v)K_{n,s}(i,j), \\
      &
    \delta_{s}^{(n)}(u,v)
     =
      \sum_{1\leq i < j \leq n} \delta_{s, (i,j)}^{(n)}(u,v)K_{n,s}(i,j),
      \qquad
      \ \varepsilon_{s}^{(n)}(u,v)
     =
      \sum_{1\leq i < j \leq n} \varepsilon_{s, (i,j)}^{(n)}(u,v)K_{n,s}(i,j),
\end{align*}
with
\begin{align*}
    \beta_{s, (i,j)}^{(n)}(u,v)
    &=
    \widetilde{h}^{(n)}_{s}(U_{i}, U_{j})1\{U_{i} \leq F_{n,\floor{ns}}(u),U_{j} \leq F_{n,\floor{ns}}(v)\} \\
    & \quad -
    \int_{0}^{F_{n,\floor{ns}}(u)}
    \frac{1\{U_{i} \geq w\}}{1-w} 
    \widetilde{h}^{(n)}_{s}(w, U_{j}) 1\{U_{j} \leq F_{n,\floor{ns}}(v)\} \mathrm{d} w\\
    & \quad -
    \int_{0}^{F_{n,\floor{ns}}(v)}
    \frac{1\{U_{j} \geq t\}}{1-t} 
    \widetilde{h}^{(n)}_{s}(U_{i}, t) 1\{U_{i} \leq F_{n,\floor{ns}}(u)\} \mathrm{d} t\\
    & \quad +
     \int_{0}^{F_{n,\floor{ns}}(u)}\int_{0}^{F_{n,\floor{ns}}(v)}
     \frac{1\{U_{i} \geq w\}}{1-w} 
      \frac{1\{U_{j} \geq t\}}{1-t} 
    \widetilde{h}^{(n)}_{s}(w,t)\mathrm{d} t \mathrm{d} w, 
    \\
    \gamma_{s, (i,j)}^{(n)}(u,v) 
    &= 
     \int_{0}^{F_{n,\floor{ns}}(u)}
    \frac{1\{U_{i} \leq w\}-w}{1-w}\Big[
        \widetilde{h}^{(n)}_{s}(w, U_{j}) 1\{U_{j} \leq F_{n,\floor{ns}}(v)\} \\
        & \qquad \qquad \qquad \qquad \qquad \qquad \qquad \qquad -
        \int_{0}^{F_{n,\floor{ns}}(v)}  \frac{1\{U_{j} \geq t\}}{1-t}    \widetilde{h}^{(n)}_{s}(w,t) \mathrm{d} t
    \Big] \mathrm{d} w,
    \\
    \delta_{s, (i,j)}^{(n)}(u,v) 
    &= 
     \int_{0}^{F_{n,\floor{ns}}(v)}
    \frac{1\{U_{j} \leq t\}-t}{1-t}\Big[
        \widetilde{h}^{(n)}_{s}(U_{i}, t) 1\{U_{i} \leq F_{n,\floor{ns}}(u)\} \\
        & \qquad \qquad \qquad \qquad \qquad \qquad \qquad \qquad -
        \int_{0}^{F_{n,\floor{ns}}(u)}  \frac{1\{U_{i} \geq w\}}{1-w}    \widetilde{h}^{(n)}_{s}(w,t) \mathrm{d} w
    \Big] \mathrm{d} t, 
    \\
    \varepsilon_{s, (i,j)}^{(n)}(u,v) 
    &= 
    \int_{0}^{F_{n,\floor{ns}}(u)}\int_{0}^{F_{n,\floor{ns}}(v)}
         \frac{1\{U_{i} \leq w\}-w}{1-w}\frac{1\{U_{j} \leq t\}-t}{1-t}
         \widetilde{h}^{(n)}_{s}(w,t)\mathrm{d} w \mathrm{d} t.
\end{align*}
For fixed $w\in [0,F_{n,\floor{ns}}(u)]$, the bracketed term in $ \gamma_{s, (i,j)}^{(n)}$ is the martingale part of the Doob-Meyer decomposition of $\widetilde{h}^{(n)}_{s}(w, U_{j}) 1\{U_{j} \leq F_{n,\floor{ns}}(v)\}$, with respect to the filtration $\mathcal{F}_{s}^{(n)}(v) := \sigma(1\{U_{j} \leq F_{n,\floor{ns}}(t)\}, t \leq v)$; likewise for $ \delta_{s, (i,j)}^{(n)}$. Consequently we consider the process $\beta_{s, (i,j)}^{(n)}$ as the crucial martingale part of the decomposition of the two-parameter process $ \widetilde{h}^{(n)}_{s}(U_{i}, U_{j})1\{U_{i} \leq F_{n,\floor{ns}}(u),U_{j} \leq F_{n,\floor{ns}}(v)\}$. Since dependence occurs between some of the summands of $\beta_{s}^{(n)}$, we introduce the associated sliced processes
\begin{align*}
    \beta_{s}^{(n), k}(u,v) = 
    \sum_{1 \leq i < k \leq j \leq n} \beta_{s, (i,j)}^{(n)}(u,v)K_{n,s}(i,j), \qquad 2 \leq k \leq n,
\end{align*}
which ensures that the random variables appearing in the first coordinate are independent from the ones in the second. This ensures that the conditional independence property (F4) of~\cite{Cairoli_Walsh} is satisfied, and hence we get the proper filtration for every fixed $s\in (0,1)$, $1\leq k\leq n$ and $n \in \amsmathbb{N}$,
\begin{align*}
    \mathcal{F}^{k,s}(u,v) = \sigma(1\{U_i \leq w\}, 1\{U_j \leq t\}, 1 \leq i < k, k \leq j \leq n, w \leq F_{n,\floor{ns}}(u), t \leq F_{n,\floor{ns}}(v)).
\end{align*}
Then the sliced process $\beta_{s}^{(n), k}$ is adapted to $\mathcal{F}^{k,s}$ and we have the following Lemma.
\begin{lemma}\label{Lemma_beta_n_k_martingale}
    For each $2 \leq k \leq n$ and fixed $s \in (0,1)$ we have that
    \begin{enumerate}[label=(\alph*), ref=\thelemma \ (\alph*)] 
    \item $\beta_{s}^{(n), k}$  is a centered strong martingale with respect to the filtration $\mathcal{F}^{k,s}$ \label{Lemma_beta_n_k_martingale_MG_prop}
    \item For each $(u,v)$, it holds that $\beta_{s}^{(n), k}$ is a degenerate, centered $U$-statistic. \label{Lemma_beta_n_k_martingale_degenerate_prop}
    \end{enumerate}
\end{lemma}
\begin{proof}
    For the first statement we check that the conditional expectation of an increment over a given rectangle given the past equals zero.
    Using the independence we need only consider a fixed pair $i < k \leq j$ and the increment can be checked along one axis at a time, since the increments in both directions are conditionally independent. For $F_{n,\floor{ns}}(u) < F_{n,\floor{ns}}(u')$ and $F_{n,\floor{ns}}(v) < F_{n,\floor{ns}}(v')$, let 
    \begin{align*}
        R= (F_{n,\floor{ns}}(u),F_{n,\floor{ns}}(u')]\times (F_{n,\floor{ns}}(v),F_{n,\floor{ns}}(v')]
    \end{align*}
    Independence and the fact that $K_{n,s}(i,j)$ is deterministic yield that we need only focus on the random part of each term.
    \begin{align*}
        &\amsmathbb{E}[ \widetilde{h}^{(n)}_{s}(U_{i}, U_{j})1\{(U_{i},U_{j}) \in R\} \vert U_{i} \leq F_{n,\floor{ns}}(u), U_{j}]\\
        &=
        \amsmathbb{E}\Big[ 
        \int_{ F_{n,\floor{ns}}(u)}^{F_{n,\floor{ns}}(u')}
        \frac{1\{U_{i} \geq w\}}{1-w} 
        \widetilde{h}^{(n)}_{s}(w, U_{j}) \\
        & \qquad \qquad \qquad \times 1\{F_{n,\floor{ns}}(v) < U_{j} \leq F_{n,\floor{ns}}(v')\} \mathrm{d} w \Big\vert U_{i} \leq F_{n,\floor{ns}}(u), U_{j}\Big]\\
        &=
        1\{F_{n,\floor{ns}}(v) < U_{j} \leq F_{n,\floor{ns}}(v')\}\\
        & \qquad \qquad \qquad \times  
        \int_{ F_{n,\floor{ns}}(u)}^{F_{n,\floor{ns}}(u')}
     \amsmathbb{E}\Big[ \frac{1\{U_{i} \geq w\}}{1-w} \Big\vert U_{i} \leq F_{n,\floor{ns}}(u)\Big]
    \widetilde{h}^{(n)}_{s}(w, U_{j}) \mathrm{d} w \\
        &=
        1\{F_{n,\floor{ns}}(v) < U_{j} \leq F_{n,\floor{ns}}(v')\}
         \frac{1\{U_{i} > F_{n,\floor{ns}}(u)\}}{1-F_{n,\floor{ns}}(u)} 
        \int_{ F_{n,\floor{ns}}(u)}^{F_{n,\floor{ns}}(u')}
    \widetilde{h}^{(n)}_{s}(w, U_{j}) \mathrm{d} w,
    \end{align*}
    and similarly
    \begin{align*}
      & \amsmathbb{E}\Big[ 
        \int_{F_{n,\floor{ns}}(v)}^{F_{n,\floor{ns}}(v')}
        \frac{1\{U_{j} \geq t\}}{1-t} 
        \widetilde{h}^{(n)}_{s}(U_{i}, t)1\{F_{n,\floor{ns}}(u) < U_{i} \leq F_{n,\floor{ns}}(u')\} \mathrm{d} t \Big\vert U_{i} \leq F_{n,\floor{ns}}(u), U_{j}\Big]\\  
        &=
        \amsmathbb{E}\Big[ 
          \int_{R}
         \frac{1\{U_{i} \geq w\}}{1-w} 
          \frac{1\{U_{j} \geq t\}}{1-t} 
        \widetilde{h}^{(n)}_{s}(w,t)\mathrm{d} (t \times w)  \Big\vert U_{i} \leq F_{n,\floor{ns}}(u), U_{j}\Big]\\  
        &=
          \int_{R}
         \amsmathbb{E}\Big[ \frac{1\{U_{i} \geq w\}}{1-w}  \Big\vert U_{i} \leq F_{n,\floor{ns}}(u)\Big]
          \frac{1\{U_{j} \geq t\}}{1-t} 
        \widetilde{h}^{(n)}_{s}(w,t)\mathrm{d} (t \times w)  \\  
        &=
           \frac{1\{U_{i} \geq F_{n,\floor{ns}}(u)\}}{1-F_{n,\floor{ns}}(u)} 
          \int_{R}
          \frac{1\{U_{j} \geq t\}}{1-t} 
        \widetilde{h}^{(n)}_{s}(w,t)\mathrm{d} (t \times w) . 
    \end{align*}
    Conditional expectations with respect to $\sigma(U_{j} \leq v, U_{i})$ are given similarly. Consequently, the conditional expectation of an increment over a given rectangle given the past equals zero. The statement of $\beta_{s}^{(n), k}$ having mean zero follows directly from the structure of each $\beta_{s, (i,j)}^{(n)}$; it is a process minus its predictable compensator and another process minus its predictable compensator. This also yields that $\beta_{s}^{(n),k}$ is centered for each $(u,v)$, and the degeneracy follows from the fact that 
    \begin{align*}
    &\amsmathbb{E}[\beta_{s, (i,j)}^{(n)}(u,v) \vert U_{j}]
    =
      \amsmathbb{E}[\widetilde{h}^{(n)}_{s}(U_{i}, U_{j})1\{U_{i} \leq F_{n,\floor{ns}}(u)\} \vert U_{j}]
      1\{U_{j} \leq F_{n,\floor{ns}}(v)\} \\
    & \quad -
    \int_{0}^{F_{n,\floor{ns}}(u)}
    \frac{ \amsmathbb{E}[1\{U_{i} \geq w\}]}{1-w} 
    \widetilde{h}^{(n)}_{s}(w, U_{j}) 1\{U_{j} \leq F_{n,\floor{ns}}(v)\} \mathrm{d} w\\
    & \quad -
    \int_{0}^{F_{n,\floor{ns}}(v)}
    \frac{1\{U_{j} \geq t\}}{1-t} 
     \amsmathbb{E}[\widetilde{h}^{(n)}_{s}(U_{i}, t) 1\{U_{i} \leq F_{n,\floor{ns}}(u)\}] \mathrm{d} t\\
    & \quad +
     \int_{0}^{F_{n,\floor{ns}}(u)}\int_{0}^{F_{n,\floor{ns}}(v)}
     \frac{ \amsmathbb{E}[1\{U_{i} \geq w\}]}{1-w} 
      \frac{1\{U_{j} \geq t\}}{1-t} 
    \widetilde{h}^{(n)}_{s}(w,t)\mathrm{d} t \mathrm{d} w \\
    &=
    1\{U_{j} \leq F_{n,\floor{ns}}(v)\} 
     \int_{0}^{F_{n,\floor{ns}}(u)}
    \widetilde{h}^{(n)}_{s}(w, U_{j})\mathrm{d} w \\
    & \quad -
    1\{U_{j} \leq F_{n,\floor{ns}}(v)\} 
     \int_{0}^{F_{n,\floor{ns}}(u)}
    \widetilde{h}^{(n)}_{s}(w, U_{j})\mathrm{d} w \\
    & \quad -
    \int_{0}^{F_{n,\floor{ns}}(v)}
    \frac{1\{U_{j} \geq t\}}{1-t} 
       \int_{0}^{F_{n,\floor{ns}}(u)} \widetilde{h}^{(n)}_{s}(w,t)\mathrm{d} w
     \mathrm{d} t \\
     & \quad+ 
     \int_{0}^{F_{n,\floor{ns}}(u)}\int_{0}^{F_{n,\floor{ns}}(v)}
      \frac{1\{U_{j} \geq t\}}{1-t} 
    \widetilde{h}^{(n)}_{s}(w,t)\mathrm{d} t \mathrm{d} w = 0,
    \end{align*}
    and likewise $\amsmathbb{E}[\beta_{s, (i,j)}^{(n)}(u,v) \vert U_{i}]=0$. Since $\beta_{s}^{(n)}$ is a sum of degenerate terms, it is itself degenerate. This concludes the proof.
\end{proof}
\begin{lemma}\label{lemma_sliced_beta_pth_moment_bound}
    For any $0 \leq u_{0}, v_{0} < x^{\ast}$ and $p>1$, using $R_0 := [0, u_0] \times [0,v_0]$,
    \begin{align*}
        \amsmathbb{E}\Big[ \sup_{(u,v) \in R_0} \big\vert \beta_{s}^{(n), k} \big\vert^{p}\Big]
        \leq
        \Big(\frac{p}{p-1}\Big)^{2p} \sup_{(u,v) \in R_0}  \amsmathbb{E}\Big[ \big\vert \beta_{s}^{(n), k} \big\vert^{p}\Big].
    \end{align*}
\end{lemma}
\begin{proof}
    Follows directly from Lemma~\ref{Lemma_beta_n_k_martingale_MG_prop} and Theorem 1.2(b) of~\cite{Cairoli_Walsh}.
\end{proof}
\begin{remark}
    In Lemma~\ref{lemma_sliced_beta_pth_moment_bound} and all the subsequent results in this Appendix we require that $u_{0}, v_{0} < x^{\ast}$. However, in Remark~\ref{remark_extending_U_stat_result_to_extended_variables} we argue why this can be relaxed to the setting where $u_{0}< x^{\ast}$ and $v_{0}\leq x^{\ast}$ or vice versa.
\end{remark}
We are now concerned with bounding the right-hand side of Lemma~\ref{lemma_sliced_beta_pth_moment_bound}. To this end, fix $(u,v)$ and let, for $2\leq k \leq r \leq n$,
\begin{align*}
    S_{r}^{s} = \sum_{1 \leq i < k \leq j \leq r} \beta_{s, (i,j)}^{(n)}(u,v)K_{n,s}(i,j),
    \qquad
    \mathcal{F}_{r} = \sigma(U_{1},\ldots, U_{k-1},\ldots, U_{r}),
\end{align*}
where we note that $S_{r}^{s}$ is adapted to $\mathcal{F}_{r}$ for every $s$. Moreover
\begin{align*}
    \amsmathbb{E}[S_{r}^{s} \vert \mathcal{F}_{r-1}]
    =
    S_{r-1} + \sum_{1 \leq i < k}  \amsmathbb{E}[\beta_{s, (i,r)}^{(n)}(u,v) \vert \mathcal{F}_{r-1}]  K_{n,s}(i,r)
    =
     S_{r-1}, 
\end{align*}
using the degeneracy of  Lemma~\ref{Lemma_beta_n_k_martingale_degenerate_prop}. Consequently $(S_{r}^{s},\mathcal{F}_{r})_{k \leq r \leq n}$ is a martingale with $S_{n}^{s} =\beta_{s}^{(n), k}(u,v)$. Define 
\begin{align*}
    D_{r}^{s} = \sum_{1 \leq i < k} \beta_{s, (i,r)}^{(n)}(u,v)K_{n,s}(i,r), \qquad k \leq r \leq n.
\end{align*}
Applying Burkholder's inequality with Jensen's inequality yields that
\begin{align*}
    \amsmathbb{E}\big[ \vert \beta_{s}^{(n), k}(u,v) \vert^{p}\big]
    =
    \amsmathbb{E}\big[ \vert S_{n}^{s} \vert^{p}\big]
    \leq
    B^{p}_{p} \amsmathbb{E}\Big[ \big\vert \sum_{r=k}^{n} (D_{r}^{s})^{2} \big\vert^{p/2}\Big]
    \leq
     B^{p}_{p}(n-k+1)^{p/2 - 1} \sum_{r=k}^{n} \amsmathbb{E}[ \vert  D_{r}^{s}\vert^{p}],
\end{align*}
for $p \geq 2$, where the constant
\begin{align*}
    B_{p} = 18p^{3/2}/(p-1)^{1/2},
\end{align*}
can be found in~\cite{ChowTeicher1997} Theorem 1, page 414. To provide an upper bound for $\amsmathbb{E}[ \vert D_{r}^{s}\vert^{p}]$, set for $1\leq i <k$ and $k\leq r \leq n$,
\begin{align*}
    T_{i}^{s} = 
    \sum_{j=1}^{i}\beta_{s, (j,r)}^{(n)}(u,v) K_{n,s}(j,r),
    \qquad 
    \mathcal{G}_{i} = \sigma(U_{1},\ldots, U_{i}, U_{r}).
\end{align*}
We note that $T_{i}^{s}$ is adapted to $\mathcal{G}_{i}$ and Lemma~\ref{Lemma_beta_n_k_martingale_degenerate_prop} yields that
\begin{align*}
     \amsmathbb{E}[\beta_{s, (i,r)}^{(n)}(u,v) \vert \mathcal{G}_{i-1}]  K_{n,s}(i,r) = 0.
\end{align*}
Thus using similar arguments as above, $(T_{i}^{(s)}, \mathcal{G}_{i})_{1 \leq i < k}$ is a martingale. Let $E_{i}^{s} = \beta_{s, (i,r)}^{(n)}(u,v) K_{n,s}(i,r)$ and note that since $T_{k-1}^{s} = D_{r}^{s}$, an application of Burkholder's and Jensen's inequality yields
\begin{align*}
    \amsmathbb{E}[ \vert D_{r}^{s}\vert^{p}]
    \leq
     B^{p}_{p}(k-1)^{p/2 - 1} \sum_{i=1}^{k-1} \amsmathbb{E}[\vert  E_{i}^{s}\vert^{p}],
\end{align*}
for $p \geq 2$. Summarizing, we have that
\begin{align*}
       &\amsmathbb{E}\big[ \vert \beta_{s}^{(n), k}(u,v) \vert^{p}\big] \leq
          B^{2p}_{p}(n-k+1)^{p/2 - 1} (k-1)^{p/2 - 1} \sum_{r=k}^{n} \sum_{i=1}^{k-1} \amsmathbb{E}[\vert E_{i}^{s}\vert^{p}] \\
          &=
          B^{2p}_{p}(n-k+1)^{p/2 - 1} (k-1)^{p/2 - 1} \sum_{r=k}^{n} \sum_{i=1}^{k-1} \amsmathbb{E} [\vert \beta_{s, (i,r)}^{(n)}(u,v) \vert^{p}] K_{n,s}(i,r)^{p} \\
          &=
          B^{2p}_{p}(n-k+1)^{p/2 - 1} (k-1)^{p/2 - 1} \amsmathbb{E}[ \vert \beta_{s, (1,2)}^{(n)}(u,v) \vert^{p} ]\sum_{r=k}^{n} \sum_{i=1}^{k-1}  K_{n,s}(i,r)^{p} \\
          & =
           B^{2p}_{p}(n-k+1)^{p/2 - 1} (k-1)^{p/2 - 1} \amsmathbb{E} [\vert \beta_{s, (1,2)}^{(n)}(u,v) \vert^{p} ]\sum_{r=k}^{n} K\Big(\frac{s-r/n}{h_{n}}\Big)^{p} \sum_{i=1}^{k-1}  K\Big(\frac{s-i/n}{h_{n}}\Big)^{p} \\
           &= 
            B^{2p}_{p}(n-k+1)^{p/2 - 1} (k-1)^{p/2 - 1} \amsmathbb{E}[\vert \beta_{s, (1,2)}^{(n)}(u,v) \vert^{p}]
            \mathcal{O}((n-k+1)h_{n}(k-1)h_{n}),
\end{align*}
using that the kernel is compactly supported.
Since the $(U_{i})$'s are iid, we have now proved the following Lemma regarding the sliced process of interest. 
\begin{lemma} \label{Lemma_sliced_beta_unif_bound}
    For $p \geq 2$, and $0 \leq u_{0}, v_{0} < x^{\ast}$, using $R_0 := [0, u_0] \times [0,v_0]$,
    \begin{align*}
          \amsmathbb{E}\Big[ \sup_{(u,v) \in R_0} \vert \beta_{s}^{(n), k} \vert^{p}\Big]
          \leq
            (a_{1, s}^{(n)})^{p}  (n-k+1)^{p/2} (k-1)^{p/2}\mathcal{O}(h_{n}^{2}), 
    \end{align*}
    with
    \begin{align*}
         (a_{1, s}^{(n)})^{p} = 
          \Big(\frac{p}{p-1}\Big)^{2p}  B^{2p}_{p}  \sup_{(u,v) \in R_0} \amsmathbb{E}\big[ \vert  \beta_{s, (1,2)}^{(n)}(u,v) \vert^{p}\big]. 
    \end{align*}
\end{lemma}
This result will be used to provide a uniform bound of the original process of interest $\beta_{s}^{(n)}(u,v)$. Since the function $u \mapsto u(1-u)$ on $[0,1]$ attains its maximum at $u = 1/2$, it follows that
\begin{align*}
    (n-k+1)^{p/2} (k-1)^{p/2} \leq (n/2)^{p}, \qquad 2 \leq k \leq n, 
\end{align*}
and consequently Lemma~\ref{Lemma_sliced_beta_unif_bound} yields
\begin{align*}
    \big\lVert \sup_{(u,v) \in R_0} \vert \beta_{s}^{(n), k} \vert \big\rVert_{p}
    \leq
    a_{1, s}^{(n)} \mathcal{O}(nh_{n}^{2/p}),
\end{align*}
where $\Vert\cdot \Vert_{p}$ denotes the $\mathcal{L}^{p}$-norm. To uniformly bound $\beta_{s}^{(n)}$, we first consider the case where $n= 2^{m}$ for integer $m$. Then write
\begin{align*}
    \beta_{s}^{(n)}(u,v) &= \sum_{l=1}^{m} B_{l,s}(u,v), \qquad
    B_{l,s}(u,v) = \sum_{k=1}^{2^{m-l}} B_{kl,s}(u,v), \\
    B_{kl,s}(u,v) &= \sum_{(i,j) \in o_{kl}} \beta_{s, (i,j)}^{(n)}(u,v)K_{n,s}(i,j),
\end{align*}
where $o_{kl} = \{(i,j): (k-1)2^{l} < i \leq (2k-1)2^{l-1} < j \leq  k2^{l}\}$, and note that since each $B_{kl,s}$ is a sum over pairwise disjoint index sets, every $B_{l,s}$ is also a strong martingale (similarly to Lemma~\ref{Lemma_beta_n_k_martingale_MG_prop}). Theorem 1.2(b) of~\cite{Cairoli_Walsh} can thus be applied again to infer that
\begin{align*}
        \amsmathbb{E}\Big[ \sup_{(u,v) \in R_0} \vert B_{l,s}(u,v) \vert^{p}\Big]
        \leq
        \Big(\frac{p}{p-1}\Big)^{2p} \sup_{(u,v) \in R_0}  \amsmathbb{E}\left[ \vert B_{l,s}(u,v) \vert^{p}\right].
\end{align*}
To bound the above right-hand side, we use the fact that for independent, zero-mean $\xi_{s}$ and $\eta_{s}$, there exists some $\varepsilon = \varepsilon(p,s) > 0$ such that
\begin{align*}
    \amsmathbb{E}\left[\vert \xi_{s} + \eta_{s} \vert^{p}\right]
    \leq
    (1-\varepsilon) 2^{p} \max\Big( \amsmathbb{E}\left[\vert \xi_{s} \vert^{p}\right], 
     \amsmathbb{E}[\vert \eta_{s} \vert^{p}]\Big).
\end{align*}
Using again Lemma~\ref{Lemma_beta_n_k_martingale}, we have that
\begin{align*}
    &\sup_{(u,v) \in R_0}  \amsmathbb{E}[ \vert B_{l,s}(u,v) \vert^{p}]
    =
    \sup_{(u,v) \in R_0}  \amsmathbb{E}\Big[ \big\vert \sum_{k=1}^{2^{m-l}} B_{kl,s}(u,v) \big\vert^{p}\Big] \\
    &=
    \sup_{(u,v) \in R_0}  \amsmathbb{E}\Big[ \big\vert  B_{(2^{m-l}) l,s}(u,v) + \sum_{k=1}^{2^{m-l-1}} B_{kl,s}(u,v) \big\vert^{p}\Big] \\
    &\leq
     (1-\varepsilon) 2^{p} \sup_{(u,v) \in R_0} \max\Big( \amsmathbb{E}[\vert B_{(2^{m-l}) l,s}(u,v) \vert^{p}], 
     \amsmathbb{E}\Big[ \big\vert  \sum_{k=1}^{2^{m-l-1}} B_{kl,s}(u,v) \big\vert^{p}\Big]\Big)\\
    & =
     (1-\varepsilon) 2^{p} \sup_{(u,v) \in R_0} 
     \amsmathbb{E}\Big[ \big\vert  \sum_{k=1}^{2^{m-l-1}} B_{kl,s}(u,v) \big\vert^{p}\Big] \leq
      (1-\varepsilon)^{m-l} 2^{p(m-l)} \sup_{(u,v) \in R_0} 
     \amsmathbb{E}[\vert  B_{1l,s}(u,v) \vert^{p}]. 
\end{align*}
We note that $B_{1l,s}$ is a sliced process and consequently Lemma~\ref{Lemma_sliced_beta_unif_bound} yields
\begin{align*}
    & \sup_{(u,v) \in R_0} 
     \amsmathbb{E}[\vert  B_{1l,s}(u,v) \vert^{p}]
     \leq
     \amsmathbb{E}\Big[ \sup_{(u,v) \in R_0} \vert  B_{1l,s}(u,v) \vert^{p}\Big] \\
     & =
     \amsmathbb{E}\Big[ \sup_{(u,v) \in R_0} \big\vert  
      \sum_{(i,j) \in o_{1l}} \beta_{s, (i,j)}^{(n)}(u,v)K_{n,s}(i,j)
      \big\vert^{p}\Big]
        \leq 
      (a_{1, s}^{(n)})^{p}  \mathcal{O}(2^{pl}h_{n}^{2}) .
\end{align*}
Gathering what we have above
\begin{align*}
    \sup_{(u,v) \in R_0}  \amsmathbb{E}[ \vert B_{l,s}(u,v) \vert^{p}]
    \leq
    (1-\varepsilon)^{m-l} 2^{p(m-l)} \mathcal{O}(2^{pl}h_{n}^{2})
    =
     (1-\varepsilon)^{m-l} (a_{1, s}^{(n)})^{p} \mathcal{O}(n^{p}h_{n}^{2}),
\end{align*}
and consequently
\begin{align*}
    \big\lVert \sup_{(u,v) \in R_0} \vert \beta_{s}^{(n)}(u,v) \vert\big\rVert_{p}
    \leq
    a_{1, s}^{(n)}\mathcal{O}(nh_{n}^{p/2}) \sum_{l=1}^{m} (1-\varepsilon)^{(m-l)/p} \leq A  a_{1, s}^{(n)}\mathcal{O}(nh_{n}^{p/2}),
\end{align*}
where $1 \leq A = A(p,s) < \infty$ is a constant depending only on $s$ and $p$. In the following Lemma we extend the above result to a general $n$. 
\begin{lemma}\label{Lemma_beta_unif_bound}
    For $p \geq 2$, and $0 \leq u_{0}, v_{0} < x^{\ast}$, using $R_0 := [0, u_0] \times [0,v_0]$,
    \begin{align}\label{eq_beta_uniform_L_p_bound}
          \big\lVert \sup_{(u,v) \in R_0} \vert \beta_{s}^{(n)}(u,v) \vert\big\rVert_{p}
          \leq
            2A a_{1, s}^{(n)}\mathcal{O}(nh_{n}^{p/2}).
    \end{align}
\end{lemma}
\begin{proof}
    We use induction to prove the result. The left-hand side of~\eqref{eq_beta_uniform_L_p_bound} is bounded by $A  a_{1, s}^{(n)}\mathcal{O}(nh_{n}^{p/2})$ if $n=2^{m}$, which yields the base case when $n=2$.
    We then assume the induction hypothesis that~\eqref{eq_beta_uniform_L_p_bound} holds for all $m < n$. For general $n$ let $\underline{n}$ be the largest power of $2$ contained in $n$. The process $\beta_{s}^{(n)}$ can then be decomposed as 
    \begin{align*}
         \beta_{s}^{(n)} = \sum_{1 \leq i < \underline{n} \leq j \leq n} \beta_{s, (i,j)}^{(n)}K_{n,s}(i,j)
         +
          \sum_{1 \leq i < j \leq \underline{n}} \beta_{s, (i,j)}^{(n)}K_{n,s}(i,j)
          +
          \sum_{\underline{n} \leq i < j \leq n} \beta_{s, (i,j)}^{(n)}K_{n,s}(i,j).
    \end{align*}
    Hence 
    \begin{align*}
        \big\Vert \sup_{(u,v) \in R_0} \vert \beta_{s}^{(n)}(u,v) \vert\big\Vert_{p}
          &\leq
          \Big\Vert \sup_{(u,v) \in R_0} \big\vert \sum_{1 \leq i < \underline{n} \leq j \leq n} \beta_{s, (i,j)}^{(n)}(u,v)K_{n,s}(i,j) \big\vert\Big\Vert_{p}\\
          & \quad +
          \Big\Vert \sup_{(u,v) \in R_0} \big\vert  \sum_{1 \leq i < j \leq \underline{n}} \beta_{s, (i,j)}^{(n)}(u,v)K_{n,s}(i,j) \big\vert\Big\Vert_{p}\\
          & \quad +
          \Big\Vert \sup_{(u,v) \in R_0} \big\vert \sum_{\underline{n} \leq i < j \leq n} \beta_{s, (i,j)}^{(n)}(u,v)K_{n,s}(i,j) \big\vert\Big\Vert_{p}.
    \end{align*}
    The first sum is a sliced process and consequently Lemma~\ref{Lemma_sliced_beta_unif_bound} yields
    \begin{align*}
         \Big\Vert \sup_{(u,v) \in R_0} \big\vert \sum_{1 \leq i < \underline{n} \leq j \leq n} \beta_{s, (i,j)}^{(n)}(u,v)K_{n,s}(i,j) \big\vert\Big\Vert_{p} \leq   a_{1, s}^{(n)} \mathcal{O}(nh_{n}^{p/2}).
    \end{align*}
    The second sum is indexed up to $\underline{n}$ which is a power of $2$ and hence our above bound yields
    \begin{align*}
        \Big\Vert \sup_{(u,v) \in R_0} \big\vert  \sum_{1 \leq i < j \leq \underline{n}} \beta_{s, (i,j)}^{(n)}(u,v)K_{n,s}(i,j) \big\vert\Big\Vert_{p}
        \leq
        A  a_{1, s}^{(n)}\mathcal{O}(\underline{n}h_{n}^{p/2}),
    \end{align*}
    and the induction hypothesis implies that the latter term is bounded by $ 2A a_{1, s}^{(n)} \mathcal{O}((n-\underline{n})h_{n}^{p/2})$. Summarizing,
    \begin{align*}
        \big\Vert \sup_{(u,v) \in R_0} \vert \beta_{s}^{(n)}(u,v) \vert\big\Vert_{p}
          &\leq
          a_{1, s}^{(n)} \mathcal{O}(nh_{n}^{p/2}) + 
          A  a_{1, s}^{(n)}\mathcal{O}(\underline{n}h_{n}^{p/2}) +
           2A a_{1, s}^{(n)} \mathcal{O}((n-\underline{n})h_{n}^{p/2}) \\  
           &\leq 2A a_{1, s}^{(n)} \mathcal{O}(nh_{n}^{p/2}),
    \end{align*}
    since $n \leq 2\underline{n}$. This concludes the proof.
\end{proof}
\begin{remark}
    The constant $a_{1, s}^{(n)}$ can be bounded as follows: The triangle inequality yields that
\begin{align*}
    a_{1, s}^{(n)} 
    &\leq
    \Big(\frac{p}{p-1}\Big)^{2}  B^{2p}_{p}  \sup_{(u,v) \in R_0} \amsmathbb{E}\big[ \vert  \beta_{s, (1,2)}^{(n)}(u,v) \vert^{p}\big]^{\frac{1}{p}} \\
    & \leq
    \Big(\frac{p}{p-1}\Big)^{2}  B^{2p}_{p}  \sup_{(u,v) \in R_0} \big\Vert\widetilde{h}^{(n)}_{s}(U_{i}, U_{j})1\{U_{i} \leq F_{n,\floor{ns}}(u),U_{j} \leq F_{n,\floor{ns}}(v)\}\big\Vert_{p} \\
    & \qquad \times \Big(1 + \frac{1}{1-F_{n,\floor{ns}}(u_{0})} + \frac{1}{1-F_{n,\floor{ns}}(v_{0})} + \frac{1}{(1-F_{n,\floor{ns}}(u_{0}))(1-F_{n,\floor{ns}}(v_{0}))} \Big) \\
    &=
    \Big(\frac{p}{p-1}\Big)^{2}  B^{2p}_{p}  
    \Big( \int_{0}^{F_{n,\floor{ns}}(u_{0})}\int_{0}^{F_{n,\floor{ns}}(v_{0})}
        \Big\vert \widetilde{h}^{(n)}_{s}(w,t)  \Big\vert^{p}
           \mathrm{d} w \mathrm{d} t \Big)^{1/p}  \\
    & \qquad \times \Big(1 + \frac{1}{1-F_{n,\floor{ns}}(u_{0})} + \frac{1}{1-F_{n,\floor{ns}}(v_{0})} + \frac{1}{(1-F_{n,\floor{ns}}(u_{0}))(1-F_{n,\floor{ns}}(v_{0}))} \Big).
\end{align*}
\end{remark}

This concludes the process $\beta_{s}^{(n)}$, and we are now ready to provide a uniform bound for the process $\gamma_{s}^{(n)}$. As above we introduce a sliced version of $\gamma_{s}^{(n)}$ for every $2\leq k \leq n$,
\begin{align*}
     \gamma_{s}^{(n),k}(u,v)
     =
      \sum_{1\leq i < k \leq j \leq n} \gamma_{s, (i,j)}^{(n)}(u,v) K_{n,s}(i,j), 
\end{align*}
and recall that 
\begin{align*}
    \gamma_{s, (i,j)}^{(n)}(u,v) 
    &= 
     \int_{0}^{F_{n,\floor{ns}}(u)}
    \frac{1\{U_{i} \leq w\}-w}{1-w}\Big[
        \widetilde{h}^{(n)}_{s}(w, U_{j}) 1\{U_{j} \leq F_{n,\floor{ns}}(v)\} \\
        & \qquad \qquad \qquad \qquad \qquad \qquad \qquad \qquad -
        \int_{0}^{F_{n,\floor{ns}}(v)}  \frac{1\{U_{j} \geq t\}}{1-t}    \widetilde{h}^{(n)}_{s}(w,t) \mathrm{d} t
    \Big] \mathrm{d} w.
\end{align*}
Consider the filtration
\begin{align*}
    \mathcal{F}^{k,s}(v) = 
    \sigma(U_{1}, \ldots, U_{k-1}, 1\{U_{j} \leq t\}, \ k \leq j \leq n, \ t \leq F_{n,\floor{ns}}(v)),
\end{align*}
and note that for each $u< x^{\ast}$, $\gamma_{s}^{(n),k}(u,\cdot) $ is a martingale with respect to the filtration $\mathcal{F}^{k,s}$. 
The kernel product $K_{n,s}(i,j)$ is non-random and bounded and consequently,
\begin{align*}
    \sup_{0\leq u < u_{0}} \vert \gamma_{s}^{(n),k}(u,\cdot) \vert, \qquad u_{0}< x^{\ast},
\end{align*}
is a non-negative submartingale. Doob's inequality then yields that
 \begin{align*}
        \amsmathbb{E}\Big[ \sup_{(u,v) \in R_0} \vert \gamma_{s}^{(n), k}(u, v) \vert^{p}\Big]
        \leq
        \Big(\frac{p}{p-1}\Big)^{p}  \amsmathbb{E}\Big[ \sup_{0 \leq u \leq u_{0}} \vert \gamma_{s}^{(n), k}(u, v_{0}) \vert^{p}\Big],
\end{align*}
and we are now interested in bounding the latter supremum. The absolute value of the process $\gamma_{s}^{(n), k}(u, v_{0})$ can be bounded for $0\leq u \leq u_{0} < x^{\ast}$ as follows
\begin{align*}
    &\vert \gamma_{s}^{(n), k}(u, v_{0}) \vert
     = 
     \Big\vert \sum_{1\leq i < k \leq j \leq n}K_{n,s}(i,j)\int_{0}^{F_{n,\floor{ns}}(u)}\frac{1\{U_{i} \leq w\}-w}{1-w} \\
     &\qquad \qquad \times 
    \Big[
        \widetilde{h}^{(n)}_{s}(w, U_{j}) 1\{U_{j} \leq F_{n,\floor{ns}}(v_{0})\}  -
        \int_{0}^{F_{n,\floor{ns}}(v_{0})}  \frac{1\{U_{j} \geq t\}}{1-t}    \widetilde{h}^{(n)}_{s}(w,t) \mathrm{d} t
    \Big] \mathrm{d} w \Big\vert \\
    &\leq
     \frac{1}{1- F_{n,\floor{ns}}(u)} 
     \sup_{0 \leq w < F_{n,\floor{ns}}(u)}  \Big\vert \sum_{i=1}^{k-1} (1\{U_{i} \leq w\}-w)  K\Big(\frac{s-i/n}{h_{n}}\Big) \Big\vert \\
     & \quad  \times 
     \int_{0}^{F_{n,\floor{ns}}(u)}\Big\vert
     \sum_{j=k}^{n}K\Big(\frac{s-j/n}{h_{n}}\Big)
   \Big[
        \widetilde{h}^{(n)}_{s}(w, U_{j}) 1\{U_{j} \leq F_{n,\floor{ns}}(v_{0})\} \\
        & \qquad \qquad \qquad \qquad \qquad \qquad \qquad \qquad  \qquad \qquad  -
        \int_{0}^{F_{n,\floor{ns}}(v_{0})}  \frac{1\{U_{j} \geq t\}}{1-t}    \widetilde{h}^{(n)}_{s}(w,t)  \mathrm{d} t
    \Big] \Big\vert \mathrm{d} w.
\end{align*} 
Computing the $p$th moment of the latter expression factorizes into computing the $p$th moment of the latter supremum and the latter integral since the $(U_{i})'s$ are independent. We begin with the supremum: Note that $w \mapsto (1\{U_{i} \leq w\} - w)$ is a centered process. Consequently, we may use the symmetrization inequality of Lemma 2.3.6 in~\cite{vanderVaart2023}, for $(\varepsilon_{i})$ iid Rademacher variables
\begin{align*}
     &\amsmathbb{E}\Big[ \sup_{0 \leq w < F_{n,\floor{ns}}(u)}  \Big\vert \sum_{i=1}^{k-1} (1\{U_{i} \leq w\}-w)  K\Big(\frac{s-i/n}{h_{n}}\Big) \Big\vert^{p}\Big] \\
     &\leq
     2^{p}\amsmathbb{E}\Big[ \sup_{0 \leq w < F_{n,\floor{ns}}(u)}  \Big\vert \sum_{i=1}^{k-1} \varepsilon_{i} 1\{U_{i} \leq w\}  K\Big(\frac{s-i/n}{h_{n}}\Big) \Big\vert^{p}\Big] =: 2^{p}
     \amsmathbb{E}\Big[\sup_{0 \leq w < F_{n,\floor{ns}}(u)}\vert S(t) \vert^{p}\Big].
\end{align*}
Let $i_{1},\ldots, i_{k-1}$ be the indices of the order statistics $U_{k-1:1},\ldots, U_{k-1:k-1}$. \\ 
When $t \in (U_{k-1:m}, U_{k-1:m+1}]$ we have that
    \begin{align*}
        S(t) = S_{m}:= \sum_{j=1}^{m}\varepsilon_{i_{j}}K\Big(\frac{s-i_{j}/n}{h_{n}}\Big), 
        \quad m=1,\ldots, k-1.
    \end{align*}
    Consequently we may bound the supremum $\amsmathbb{P}$-a.s. as follows
    \begin{align*}
        \sup_{t \leq F_{n,\floor{ns}}(u)}(S(t))^{2}
        \leq
        \max_{1\leq m \leq k-1}(S_{m})^{2}.
    \end{align*}
    Define $\overline{U}_{k-1}:=(U_{k-1:1},\ldots, U_{k-1:k-1})$, and note that $(S_{m})$ is a martingale with respect to $\mathcal{G}_{m} := \sigma(\overline{U}_{k-1}, \varepsilon_{i_{1}},\ldots, \varepsilon_{i_{m}})$,  since it is adapted, 
    \begin{align*}
        \amsmathbb{E}[S_{m+1} \vert \mathcal{G}_{m}]
        =
        \amsmathbb{E}[S_{m} + \varepsilon_{i_{m+1}}K((s-i_{m+1}/n)/h_{n}) \vert \mathcal{G}_{m}]
        =
        S_{m},
    \end{align*} 
    and clearly, for any fixed $k-1$, we have $\amsmathbb{E}[\vert S_{k-1}^{p}\vert]< \infty$. 
    In particular, Doob's inequality yields that
     \begin{align*}
        &\amsmathbb{E}\Big[\sup_{0 \leq w < F_{n,\floor{ns}}(u)}\vert S(t) \vert^{p}\Big]
         \leq
        \amsmathbb{E}\Big[ \max_{1\leq m \leq k-1}\vert S_{m} \vert^{p} \Big] \leq
        \Big(\frac{p}{p-1}\Big)^{p} \amsmathbb{E}[\vert S_{k-1} \vert^{p}]\\
        &
        =
         \Big(\frac{p}{p-1}\Big)^{p} \amsmathbb{E}\Big[\Big\vert \sum_{j=1}^{k-1}\varepsilon_{i_{j}}K\Big(\frac{s-i_{j}/n}{h_{n}}\Big) \Big\vert^{p}\Big]
        \leq
             \Big(\frac{pB_{p}}{p-1}\Big)^{p} \amsmathbb{E}\Big[\Big\vert \sum_{j=1}^{k-1}\varepsilon_{i_{j}}^{2}K\Big(\frac{s-i_{j}/n}{h_{n}}\Big)^{2}\Big\vert^{p/2}\Big] \\
        &
             \leq
               \Big(\frac{pB_{p}}{p-1}\Big)^{p}(k-1)^{p/2 -1} \sum_{j=1}^{k-1} \amsmathbb{E}[\vert\varepsilon_{i_{j}}\vert^{p}] K\Big(\frac{s-i_{j}/n}{h_{n}}\Big)^{p}
            \\
           &  \leq \Big(\frac{pB_{p}}{p-1}\Big)^{p}(k-1)^{p/2 -1}\sum_{i=1}^{k-1}K^{p}\Big(\frac{s-i/n}{h_{n}}\Big)
           =  \mathcal{O}((k-1)^{p/2}h_{n}).
\end{align*}
Consider now the integral above. Since for each $w$, the bracket in the integral is a martingale, Burkholder's and Jensen's inequality yield 
\begin{align*}
    &\amsmathbb{E}\Big[ \int_{0}^{F_{n,\floor{ns}}(u_{0})}\Big\vert
     \sum_{j=k}^{n}K\Big(\frac{s-j/n}{h_{n}}\Big)
   \Big[
        \widetilde{h}^{(n)}_{s}(w, U_{j}) 1\{U_{j} \leq F_{n,\floor{ns}}(v_{0})\} \\
        & \qquad \qquad \qquad \qquad \qquad \qquad \qquad \qquad  \qquad \qquad  -
        \int_{0}^{F_{n,\floor{ns}}(v_{0})}  \frac{1\{U_{j} \geq t\}}{1-t}    \widetilde{h}^{(n)}_{s}(w,t)  \mathrm{d} t
    \Big] \Big\vert^{p} \mathrm{d} w \Big]  \\
    & \leq 
    \int_{0}^{F_{n,\floor{ns}}(u_{0})}
    B_{p}^{p}(n-k+1)^{p/2 - 1}
     \sum_{j=k}^{n} K\Big(\frac{s-j/n}{h_{n}}\Big)^{p} \\
        & \qquad \qquad \times
        \amsmathbb{E}\Big[ \Big\vert
   \Big[\widetilde{h}^{(n)}_{s}(w, U_{j}) 1\{U_{j} \leq F_{n,\floor{ns}}(v_{0})\}     -
        \int_{0}^{F_{n,\floor{ns}}(v_{0})}  \frac{1\{U_{j} \geq t\}}{1-t}    \widetilde{h}^{(n)}_{s}(w,t)  \mathrm{d} t
    \Big] \Big\vert^{p} \Big] \mathrm{d} w \\
        &= 
       B_{p}^{p}(n-k+1)^{p/2 - 1} 
       \sum_{j=k}^{n} K\Big(\frac{s-j/n}{h_{n}}\Big)^{p}
    \int_{0}^{F_{n,\floor{ns}}(u_{0})} 
     \amsmathbb{E}\Big[ \Big\vert \widetilde{h}^{(n)}_{s}(w, U_{1}) 1\{U_{1} \leq F_{n,\floor{ns}}(v_{0})\}  \\
        & \qquad \qquad \qquad \qquad \qquad \qquad \qquad \qquad \qquad \qquad 
     -
        \int_{0}^{F_{n,\floor{ns}}(v_{0})}  \frac{1\{U_{1} \geq t\}}{1-t}    \widetilde{h}^{(n)}_{s}(w,t)  \mathrm{d} t \Big\vert^{p}\Big]
        \mathrm{d} w  \\
        & \leq
          \mathcal{O}(  (n-k+1)^{p/2} h_{n})B_{p}^{p} \Big(1 +  \frac{1}{1-F_{n,\floor{ns}}(v_{0})}\Big)  \int_{0}^{F_{n,\floor{ns}}(u_{0})}\int_{0}^{F_{n,\floor{ns}}(v_{0})}
        \Big\vert \widetilde{h}^{(n)}_{s}(w,t)  \Big\vert^{p}
           \mathrm{d} w \mathrm{d} t. 
\end{align*}
Consequently
\begin{align*}
    \amsmathbb{E}\Big[ \sup_{(u,v) \in R_0} \vert \gamma_{s}^{(n)}(u,v) \vert^{p}\Big]
    &\leq
     \frac{1}{1- F_{n,\floor{ns}}(u)} \mathcal{O}((n-k+1)^{p/2}(k-1)^{p/2}h_{n}^{2})B_{p}^{2p}\\
     & \times \Big(1 +  \frac{1}{1-F_{n,\floor{ns}}(v_{0})}\Big) 
     \int_{0}^{F_{n,\floor{ns}}(u_{0})}\int_{0}^{F_{n,\floor{ns}}(v_{0})}
        \Big\vert \widetilde{h}^{(n)}_{s}(w,t)  \Big\vert^{p}
           \mathrm{d} w \mathrm{d} t.
\end{align*}
This bound is similar to that of Lemma~\ref{Lemma_sliced_beta_unif_bound}. The arguments leading to Lemma~\ref{Lemma_beta_unif_bound} may be repeated to transfer the bound from the sliced process $ \gamma_{s}^{(n),k}$ to the original process $\gamma_{s}^{(n)}$. This is the content of the following Lemma.
\begin{lemma}\label{Lemma_gamma_unif_bound}
    For $p \geq 2$, and $0 \leq u_{0}, v_{0} < x^{\ast}$, using $R_0 := [0, u_0] \times [0,v_0]$,
    \begin{align*}
          \big\Vert \sup_{(u,v) \in R_0} \vert \gamma_{s}^{(n)}(u,v) \vert \big\Vert_{p}
          \leq
             2A a_{2, s}^{(n)}\mathcal{O}(nh_{n}^{2/p}),
    \end{align*}
    where 
    \begin{align*}
         a_{2, s}^{(n)} &\leq
         \frac{pB_{p}^{2}}{(p-1)(1-F_{n,\floor{ns}}(u_{0}))} \Big(1 +  \frac{1}{1-F_{n,\floor{ns}}(v_{0})}\Big)\\
        & \qquad \qquad \times 
         \Big( \int_{0}^{F_{n,\floor{ns}}(u_{0})}\int_{0}^{F_{n,\floor{ns}}(v_{0})}
        \Big\vert \widetilde{h}^{(n)}_{s}(w,t)  \Big\vert^{p}
           \mathrm{d} w \mathrm{d} t \Big)^{1/p}.
    \end{align*}
\end{lemma}
A similar bound is obtained for the process $\delta_{s}^{(n)}$ with identical arguments as those in Lemma~\ref{Lemma_gamma_unif_bound}.
\begin{lemma}\label{Lemma_delta_unif_bound}
    For $p \geq 2$, and $0 \leq u_{0}, v_{0} < x^{\ast}$, using $R_0 := [0, u_0] \times [0,v_0]$,
    \begin{align*}
          \big\Vert \sup_{(u,v) \in R_0} \vert \delta_{s}^{(n)}(u,v) \vert\big\Vert_{p}
          \leq
            2A a_{3, s}^{(n)}\mathcal{O}(nh_{n}^{2/p}),
    \end{align*}
    where 
    \begin{align*}
         a_{3, s}^{(n)} &\leq
         \frac{pB_{p}^{2}}{(p-1)(1-F_{n,\floor{ns}}(v_{0}))}\Big(1 +  \frac{1}{1-F_{n,\floor{ns}}(u_{0})}\Big)\\
        & \qquad \qquad \times 
         \Big( \int_{0}^{F_{n,\floor{ns}}(u_{0})}\int_{0}^{F_{n,\floor{ns}}(v_{0})}
        \Big\vert \widetilde{h}^{(n)}_{s}(w,t)  \Big\vert^{p}
           \mathrm{d} w \mathrm{d} t \Big)^{1/p}.
    \end{align*}
\end{lemma}
Finally, we proceed to the term
\begin{align*}
    \varepsilon_{s}^{(n)}(u,v)
     =
      \sum_{1\leq i < j \leq n} \varepsilon_{s, (i,j)}^{(n)}(u,v)K_{n,s}(i,j),
\end{align*}
where 
\begin{align*}
    \varepsilon_{s, (i,j)}^{(n)}(u,v) 
    = 
    \int_{0}^{F_{n,\floor{ns}}(u)}\int_{0}^{F_{n,\floor{ns}}(v)}
         \frac{1\{U_{i} \leq w\}-w}{1-w}\frac{1\{U_{j} \leq t\}-t}{1-t}
         \widetilde{h}^{(n)}_{s}(w,t)\mathrm{d} w \mathrm{d} t.
\end{align*}
As above we introduce a sliced version of $\varepsilon_{s}^{(n)}$,
\begin{align*}
     \varepsilon_{s}^{(n),k}(u,v)
     =
      \sum_{1\leq i < k \leq j \leq n} \varepsilon_{s, (i,j)}^{(n)}(u,v) K_{n,s}(i,j), 
\end{align*}
and note that 
\begin{align*}
    &\amsmathbb{E}\Big[ \sup_{(u,v) \in R_0} \vert \varepsilon_{s}^{(n), k} \vert^{p}\Big]=
    \amsmathbb{E}\Big[ \sup_{(u,v) \in R_0} \Big\vert 
     \sum_{1\leq i < k \leq j \leq n} K_{n,s}(i,j)  \\
     & \qquad \qquad \qquad \times
     \int_{0}^{F_{n,\floor{ns}}(u)}\int_{0}^{F_{n,\floor{ns}}(v)}
         \frac{1\{U_{i} \leq w\}-w}{1-w}\frac{1\{U_{j} \leq t\}-t}{1-t}
         \widetilde{h}^{(n)}_{s}(w,t)\mathrm{d} w \mathrm{d} t
    \Big\vert^{p}\Big] \\
         & \leq
         \frac{1}{(1- F_{n,\floor{ns}}(u_{0}))^{p}} \frac{1}{(1- F_{n,\floor{ns}}(v_{0}))^{p}}\int_{0}^{F_{n,\floor{ns}}(u_{0})}\int_{0}^{F_{n,\floor{ns}}(v_{0})}
          \Big\vert \widetilde{h}^{(n)}_{s}(w,t)  \Big\vert^{p} \mathrm{d} w \mathrm{d} t  \\
         & \qquad \times
         \amsmathbb{E}\Big[ \sup_{0 \leq w < F_{n,\floor{ns}}(u)}  \Big\vert \sum_{i=1}^{k-1} (1\{U_{i} \leq w\}-w)  K\Big(\frac{s-i/n}{h_{n}}\Big) \Big\vert^{p}\Big]
         \\
        & \qquad \times \amsmathbb{E}\Big[ \sup_{0 \leq w < F_{n,\floor{ns}}(u)}  \Big\vert \sum_{j=k}^{n} (1\{U_{j} \leq w\}-w)  K\Big(\frac{s-j/n}{h_{n}}\Big) \Big\vert^{p}\Big] 
        \\
        & \leq
        \frac{\mathcal{O}((n-k+1)^{p/2}(k-1)^{p/2}h_{n}^{2}) }{(1-F_{n,\floor{ns}}(u_{0}))^{p}(1-F_{n,\floor{ns}}(v_{0}))^{p}}
        \int_{0}^{F_{n,\floor{ns}}(u_{0})}\int_{0}^{F_{n,\floor{ns}}(v_{0})}
         \Big\vert \widetilde{h}^{(n)}_{s}(w,t)  \Big\vert^{p}
           \mathrm{d} w \mathrm{d} t,
\end{align*}
by similar arguments as those leading to Lemma~\ref{Lemma_gamma_unif_bound}. Consequently we get the following Lemma.
\begin{lemma}\label{Lemma_epsilon_unif_bound}
    For $p \geq 2$, and $0 \leq u_{0}, v_{0} < x^{\ast}$, using $R_0 := [0, u_0] \times [0,v_0]$,
    \begin{align*}
          \big\Vert \sup_{(u,v) \in R_0} \vert \varepsilon_{s}^{(n)}(u,v) \vert\big\Vert_{p}
          \leq
             2A a_{4, s}^{(n)}\mathcal{O}(nh_{n}^{2/p}),
    \end{align*}
    where 
    \begin{align*}
         a_{4, s}^{(n)} \leq
         \frac{1}{(1-F_{n,\floor{ns}}(u_{0}))(1-F_{n,\floor{ns}}(v_{0}))}
         \Big( \int_{0}^{F_{n,\floor{ns}}(u_{0})}\int_{0}^{F_{n,\floor{ns}}(v_{0})}
        \Big\vert \widetilde{h}^{(n)}_{s}(w,t)  \Big\vert^{p}
           \mathrm{d} w \mathrm{d} t \Big)^{1/p}.
    \end{align*}
\end{lemma} 
\begin{remark}
    The four constants $(a_{i, s}^{(n)})_{i=1}^{4}$ all involve the $\mathcal{L}^{p}$-norm of the process 
    \begin{align*}
         [0,u_{0}] \times [0,v_{0}] \ni (u,v) \mapsto  \widetilde{h}^{(n)}_{s}(U_{i}, U_{j})1\{U_{i} \leq F_{n,\floor{ns}}(u),U_{j} \leq F_{n,\floor{ns}}(v)\}.
    \end{align*}
    Writing out the definition of $\widetilde{h}_{s}^{(n)}(\cdot,\cdot)$ and applying a time change, we have that
    \begin{align*}
        & \int_{0}^{F_{n,\floor{ns}}(u_{0})}\int_{0}^{F_{n,\floor{ns}}(v_{0})}
        \Big\vert \widetilde{h}^{(n)}_{s}(w,t)  \Big\vert^{p}
           \mathrm{d} w \mathrm{d} t \\
           & \qquad =
            \int_{0}^{F_{n,\floor{ns}}(u_{0})}\int_{0}^{F_{n,\floor{ns}}(v_{0})}
        \Big\vert h_{n,\floor{ns}}(F_{n,\floor{ns}}^{\leftarrow}(w), F_{n,\floor{ns}}^{\leftarrow}(t))  \Big\vert^{p}
           \mathrm{d} w \mathrm{d} t \\
           &  \qquad  =
           \int_{0}^{u_{0}}\int_{0}^{v_{0}} \Big\vert h_{n,\floor{ns}}(w,t) \Big\vert^{p} \mathrm{d} F_{n,\floor{ns}}(w)\mathrm{d} F_{n,\floor{ns}}(t),
    \end{align*}
    which leads to the following theorem.
\end{remark}
\begin{theorem}\label{thm_U_stat_repr_Xi_Xj_remainder_unif_bound}
    Assume $h \in \mathcal{L}_{p}(F_{n,\floor{ns}} \otimes F_{n,\floor{ns}})$, with $p \geq 2$. It then holds that
    \begin{align*}
        &I_{n}(u,v \vert s)
      =
      R_{s}^{(n)}(u,v) + \sum_{1\leq i < j \leq n}K_{n,s}(i,j)\Big[
       \int_{0}^{u}h_{n,\floor{ns}}(x,X_{\floor{ns}}^{(n,j)}) 1\{X_{\floor{ns}}^{(n,j)} \leq v\} F_{n,\floor{ns}}(\mathrm{d} x)\\
       & \quad +
       \int_{0}^{v}h_{n,\floor{ns}}(X_{\floor{ns}}^{(n,i)}, y) 1\{X_{\floor{ns}}^{(n,i)} \leq u\} F_{n,\floor{ns}}(\mathrm{d} y)\\
       & \quad
       -
        \int_{0}^{u}\int_{0}^{v}
        h_{n,\floor{ns}}(x,y)F_{n,\floor{ns}}(\mathrm{d} x) F_{n,\floor{ns}}(\mathrm{d} y)
       \Big], 
    \end{align*}
    where $(X_{\floor{ns}}^{(n,i)})_{i=1}^{n}$ are independent copies of $X_{\floor{ns}}^{(n)}$ and for each $u_{0}, v_{0} < x^{\ast}$, 
    \begin{align}\label{eq_U_stat_remainder_unif_bound}
       \amsmathbb{E}\Big[ \sup_{u \leq u_{0}, v \leq v_{0}} \vert R_{s}^{(n)}(u,v) \vert^{p} \Big]
       \leq (a_{s}^{(n)})^{p}\mathcal{O}(n^{p}h_{n}^{2}).
    \end{align}
    The constant $a_{s}^{(n)}$ satisfies
    \begin{align*}
        a_{s}^{(n)}
        \leq
        \widetilde{A} \Big[\int_{0}^{u_{0}}\int_{0}^{v_{0}} \big\vert h_{n,\floor{ns}}(x,y) \big\vert^{p}F_{n,\floor{ns}}(\mathrm{d} x) F_{n,\floor{ns}}(\mathrm{d} y)\Big]^{1/p},
    \end{align*}
    with $\widetilde{A}$ depending only on $p$ and $s$.
\end{theorem}
\begin{proof}
The above remark yields that 
    \begin{align*}
       \amsmathbb{E}\Big[ \sup_{u \leq u_{0}, v \leq v_{0}} \vert R_{s}^{(n)}(u,v) \vert^{p} \Big]
       \leq (a_{s}^{(n)})^{p}\mathcal{O}(n^{p}h_{n}^{2}),
    \end{align*}
    where 
    \begin{align*}
        a_{s}^{(n)}
        \leq
        \widetilde{A} \Big[\int_{0}^{u_{0}}\int_{0}^{v_{0}} \big\vert h_{n,\floor{ns}}(x,y) \big\vert^{p}F_{n,\floor{ns}}(\mathrm{d} x) F_{n,\floor{ns}}(\mathrm{d} y)\Big]^{1/p},
    \end{align*}
    and the constant $\widetilde{A}$ is proportional to
    \begin{align*}
         \frac{pB_{p}^{2}}{(p-1)}\frac{1}{(1-F_{n,\floor{ns}}(u_{0}))(1-F_{n,\floor{ns}}(v_{0}))}.
    \end{align*}
    Hence it remains to argue that $\widetilde{A}$ can be bounded by a constant only depending on $p$ and $s$. Recall the decomposition of interest
    \begin{align*}
      &\sum_{1\leq i < j \leq n} 
     \widetilde{h}^{(n)}_{s}(U_{i}, U_{j})1\{U_{i} \leq F_{n,\floor{ns}}(u),U_{j} \leq F_{n,\floor{ns}}(v)\} K_{n,s}(i,j)\\
     &=
     \sum_{1\leq i < j \leq n}K_{n,s}(i,j)\Big[
       \int_{0}^{F_{n,\floor{ns}}(u)}\widetilde{h}^{(n)}_{s}(w, U_{j}) 1\{U_{j} \leq F_{n,\floor{ns}}(v)\} \mathrm{d} w\\
       & \  +
       \int_{0}^{F_{n,\floor{ns}}(v)}\widetilde{h}^{(n)}_{s}(U_{i}, t) 1\{U_{i} \leq F_{n,\floor{ns}}(u)\} \mathrm{d} t
       -
        \int_{0}^{F_{n,\floor{ns}}(u)}\int_{0}^{F_{n,\floor{ns}}(v)}
        \widetilde{h}^{(n)}_{s}(w,t)\mathrm{d} w \mathrm{d} t
       \Big] \\
       & \quad + 
       \beta_{s}^{(n)}(u,v)
       -
       \gamma_{s}^{(n)}(u,v)
       -
       \delta_{s}^{(n)}(u,v)
       +
       \varepsilon_{s}^{(n)}(u,v),
    \end{align*}
    and note that when both $u_{0}$ and $v_{0}$ are such that $F_{n,\floor{ns}}(u_{0})$ and $F_{n,\floor{ns}}(v_{0})$ are less than $1/2$, the assertion of the theorem is an immediate consequence of Lemmas~\ref{Lemma_beta_unif_bound}-\ref{Lemma_epsilon_unif_bound}. If at least one is larger than $1/2$, decompose the above indicator $1\{U_{i} \leq F_{n,\floor{ns}}(u),U_{j} \leq F_{n,\floor{ns}}(v)\}$ on the left-hand side above (when both $F_{n,\floor{ns}}(u_{0}), F_{n,\floor{ns}}(v_{0}) \geq 1/2$) into
    \begin{align*}
        &1\{U_{i} \leq 1/2,U_{j} \leq 1/2\} \  +
         1\{1/2 < U_{i} \leq F_{n,\floor{ns}}(u), 1/2 < U_{j} \leq F_{n,\floor{ns}}(v)\} \\
         & \qquad +
         1\{U_{i} \leq 1/2,  1/2 < U_{j} \leq F_{n,\floor{ns}}(v)\}
         +
          1\{1/2 < U_{i} \leq F_{n,\floor{ns}}(u), U_{j} \leq 1/2 \}.
    \end{align*}
    The assertion of the theorem follows immediately for the first term. We show that it also follows for the second, and note that similar arguments can be used for the third and fourth. It holds that
    \begin{align*}
         &1\{1/2 < U_{i} \leq F_{n,\floor{ns}}(u), 1/2 < U_{j} \leq F_{n,\floor{ns}}(v)\}\\
         &=
          1\{1-F_{n,\floor{ns}}(u) \leq 1- U_{i} < 1/2,1-F_{n,\floor{ns}}(v) \leq 1- U_{j} <1/2\} \\
          & =:
           1\{1-F_{n,\floor{ns}}(u) \leq \widetilde{U}_{i} \leq 1/2, 1-F_{n,\floor{ns}}(v) \leq \widetilde{U}_{j} \leq 1/2\},
    \end{align*}
    where $\widetilde{U}_{1}, \ldots, \widetilde{U}_{n}$ are iid uniform random variables. In particular
    \begin{align*}
         &\widetilde{h}^{(n)}_{s}(U_{i}, U_{j})1\{1/2 < U_{i} \leq F_{n,\floor{ns}}(u), 1/2 < U_{j} \leq F_{n,\floor{ns}}(v)\}\\
         &=
         \widetilde{h}^{(n)}_{s}(U_{i}, U_{j}) 1\{1-F_{n,\floor{ns}}(u) \leq \widetilde{U}_{i} \leq 1/2, 1-F_{n,\floor{ns}}(v) \leq \widetilde{U}_{j} \leq 1/2\}\\
         & =
         \widehat{h}^{(n)}_{s}(\widetilde{U}_{i}, \widetilde{U}_{j}) 1\{1-F_{n,\floor{ns}}(u) \leq \widetilde{U}_{i} \leq 1/2, 1-F_{n,\floor{ns}}(v) \leq \widetilde{U}_{j} \leq 1/2\}, 
    \end{align*}
    for $\widehat{h}^{(n)}_{s}(x,y) := \widetilde{h}^{(n)}_{s}(1-x, 1-y)$. The corresponding $p$th moment of the supremum of the remainder term is then bounded by $(a_{s}^{(n)})^{p}\mathcal{O}(n^{p}h_{n}^{2})$, where 
    \begin{align*}
        a_{s}^{(n)}
        \leq
        \widetilde{A} \Big[ \int_{1-F_{n,\floor{ns}}(u_{0})}^{1/2}\int_{1-F_{n,\floor{ns}}(v_{0})}^{1/2}
        \Big\vert \widehat{h}^{(n)}_{s}(w,t)  \Big\vert^{p}
           \mathrm{d} w \mathrm{d} t \Big]^{1/p},
    \end{align*}
    with $\widetilde{A}$ depending only on $p$ and $s$, and 
    \begin{align*}
        &\int_{1-F_{n,\floor{ns}}(u_{0})}^{1/2}\int_{1-F_{n,\floor{ns}}(v_{0})}^{1/2}
        \Big\vert \widehat{h}^{(n)}_{s}(w,t)  \Big\vert^{p}
         \mathrm{d} w \mathrm{d} t \\
           &=
           \int_{1-F_{n,\floor{ns}}(u_{0})}^{1/2}\int_{1-F_{n,\floor{ns}}(v_{0})}^{1/2}
        \Big\vert \widetilde{h}^{(n)}_{s}(1-w,1-t)  \Big\vert^{p}
           \mathrm{d} w \mathrm{d} t \\
           &=
           \int_{1/2}^{F_{n,\floor{ns}}(u_{0})}\int_{1/2}^{F_{n,\floor{ns}}(v_{0})}
        \Big\vert \widetilde{h}^{(n)}_{s}(w,t)  \Big\vert^{p}
           \mathrm{d} w \mathrm{d} t 
           \leq
           \int_{0}^{F_{n,\floor{ns}}(u_{0})}\int_{0}^{F_{n,\floor{ns}}(v_{0})}
        \Big\vert \widetilde{h}^{(n)}_{s}(w,t)  \Big\vert^{p}
           \mathrm{d} w \mathrm{d} t .
    \end{align*}
    This concludes the proof.    
\end{proof}
Finally we present the more general result where we have pairs $(X_{\floor{ns}}^{(n,i)},Y_{\floor{ns}}^{(n,i)})_{i=1}^{n}$ with distribution functions $(F_{n,\floor{ns}}, G_{n,\floor{ns}})$ , in which case we write
\begin{align*}
    I_{n}(u,v \vert s)
    =
    \sum_{1 \leq i < j \leq n} h_{n,\floor{ns}}(X_{\floor{ns}}^{(n,i)},Y_{\floor{ns}}^{(n,i)}) 1\{X_{\floor{ns}}^{(n,i)} \leq u, Y_{\floor{ns}}^{(n,i)} \leq v\}K_{n,s}(i,j)
\end{align*}
\begin{theorem}\label{thm_U_stat_repr_Xi_Yj_remainder_unif_bound}
    Assume $h \in \mathcal{L}_{p}(F_{n,\floor{ns}} \otimes G_{n,\floor{ns}})$, with $p \geq 2$. Then we have 
    \begin{align*}
        &I_{n}(u,v \vert s)
      =
      R_{s}^{(n)}(u,v) + \sum_{1\leq i < j \leq n}K_{n,s}(i,j)\Big[
       \int_{0}^{u}h_{n,\floor{ns}}(x,Y_{\floor{ns}}^{(n,j)}) 1\{Y_{\floor{ns}}^{(n,j)} \leq v\} F_{n,\floor{ns}}(\mathrm{d} x)\\
       & \quad +
       \int_{0}^{v}h_{n,\floor{ns}}(X_{\floor{ns}}^{(n,i)}, y) 1\{X_{\floor{ns}}^{(n,i)} \leq u\} G_{n,\floor{ns}}(\mathrm{d} y)\\
       & \quad
       -
        \int_{0}^{u}\int_{0}^{v}
        h_{n,\floor{ns}}(x,y)F_{n,\floor{ns}}(\mathrm{d} x) G_{n,\floor{ns}}(\mathrm{d} y)
       \Big], 
    \end{align*}
   where $R_{s}^{(n)}$ satisfies~\eqref{eq_U_stat_remainder_unif_bound} where the integral in the bound is taken with respect to $(F_{n,\floor{ns}} \otimes G_{n,\floor{ns}})$.
\end{theorem} 
\begin{proof}
    The proof follows almost directly by repeating arguments similar to those of Theorem~\ref{thm_U_stat_repr_Xi_Xj_remainder_unif_bound}. However, the reduction to uniform random variables now involves a sequence $(U_{i}, V_{i})$ of iid random vectors in the unit square with uniform marginals such that
    \begin{align*}
        X_{\floor{ns}}^{(n,i)} = F_{n,\floor{ns}}^{\leftarrow}(U_{i}), 
        \qquad 
        Y_{\floor{ns}}^{(n,i)} = G_{n,\floor{ns}}^{\leftarrow}(V_{i}).
    \end{align*}
    The processes $\beta_{s, (i,j)}^{(n)}$, $\gamma_{s, (i,j)}^{(n)}$, $\delta_{s, (i,j)}^{(n)}$ and $\varepsilon_{s, (i,j)}^{(n)}$ are defined similarly but with $U_j$ replaced by $V_j$. Due to possible dependence between $U_i$ and $V_i$, we repeat the above arguments again using the sliced processes. This concludes the proof.
\end{proof}
\begin{remark}\label{remark_extending_U_stat_result_to_extended_variables}
    Theorem~\ref{thm_U_stat_repr_Xi_Yj_remainder_unif_bound} is stated for sequences of integrable real-valued random variables. As mentioned in~\cite{U_statistics_Stute_94}, the result may easily be generalized to the setup where one of the sequences consists of extended random variables. This is precisely the setting needed here, because we apply the theorem for the extended random variables
    \begin{align*}
    W_{\floor{ns}}^{(n,i)} = \begin{cases}
       Z_{\floor{ns}}^{(n,i)}, \quad \delta_{\floor{ns}}^{(n,i)} =1, \\
        \infty, \quad \ \ \  \ \delta_{\floor{ns}}^{(n,i)}=0.
    \end{cases}
\end{align*}
\end{remark}
\newpage
\section{Additional empirical process results}\label{appendix_additional_empirical_processs_results}
In this section we prove the additional empirical process results used in Appendix~\ref{appendix_proof_of_main_results}. The first result is to establish that the survival baseline function $1-H$ evaluated in the order statistic $Z_{n:n-k_{n}}$ is asymptotically of order $k_{n}/n$. The second result is a technical result bounding the variance of the difference between the empirical building blocks $\amsmathbb{H}_{n}(\cdot \vert s)$, $\amsmathbb{H}_{n}^{1}(\cdot \vert s)$ and their corresponding counterparts based on the locally transported observations. This is used to establish the third result which provides the asymptotic uniform error rate of these differences. The key insight is that they are based on the same observations and consequently correlated. In particular, the variance of the differences is asymptotically smaller than that of Lemma~\ref{lemma_H_n_and_H_n_1_unif_cons}, making the bias governed by the local transportation the predominant asymptotic contribution for many standard choices of bandwidth sequences. The proof of the third result is based on a concentration inequality argument which adds an extra $\sqrt{\log(n)}$-factor to the rate. The fourth result is a Rényi-type limit theorem extending the result on $A_{n}^{(\gamma)}(p;k_{n})$ in Lemma 1 of~\cite{Csorgo_96}. The proof technique used there which is based on Theorem 4.5.1 in~\cite{csorgo1986weighted} is unavailable here due to the heterogeneity of the data. Instead we combine a symmetrization and martingale argument to prove the result. The fifth and last result is the locally transported version of Lemma~\ref{lemma_H_tail_fraction} which follows from the third result.

\begin{lemma}\label{lemma_order_stat_baseline_equiv}
    It holds that $(n/k_{n})(1-H(Z_{n:n-k_{n}})) \overset{\amsmathbb{P}}{\to} 1/C_{H}(1)$.
\end{lemma}
\begin{proof}[Proof of Lemma~\ref{lemma_order_stat_baseline_equiv}]
    For any $x \leq x^{*}$, let
    \begin{align*}
        S_{n}(x) = \sum_{i=1}^{n} 1\{Z_{i}^{(n)} > x\},
    \end{align*}
    and note that
    \begin{align*}
        \{Z_{n:n-k_{n}} \leq x \} = \{S_{n}(x) \leq k_{n}\}, \quad    \{Z_{n:n-k_{n}} \geq x \} = \{S_{n}(x) \geq k_{n} +1\}.
    \end{align*}
    Fix any sequence $\delta_{n} \in (0,1)$ such that $\delta_n \to 0$, and $\delta_{n}\sqrt{k_{n}} \to \infty$, for instance $\delta_{n} := k_{n}^{-1/4}$. Let
    \begin{align*}
        x_{n}^{-} := H^{\leftarrow}(1-(1+\delta_{n})k_{n}/(nC_{H}(1)), \qquad 
        x_{n}^{+} := H^{\leftarrow}(1-(1-\delta_{n})k_{n}/(nC_{H}(1))),
    \end{align*}
    and note that both of these sequences are approaching $x^{*}$ as $n \to \infty$. The tail equivalence~\eqref{eq_H_scedasis_dens} then yields 
    \begin{align*}
        \amsmathbb{E}[S_{n}(x_{n}^{+})] &= \sum_{i=1}^{n} (1-H_{n,i}(x_{n}^{+}))
        = (1 + o(1))(1-H(x_{n}^{+}))\sum_{i=1}^{n} c_{H}(i/n) \\
        &= 
        (1 + o(1)) (1-\delta_{n})\frac{k_{n}}{C_{H}(1)}\frac{1}{n} \sum_{i=1}^{n} c_{H}(i/n)
        =
         (1-\delta_{n})k_{n} (1 + o(1)),
    \end{align*}
    where we used that the tail equivalence is uniform in $i \leq n$. Similarly we have that $\amsmathbb{E}[S_{n}(x_{n}^{-})] = (1+\delta_{n})k_{n} (1 + o(1))$. It then suffices to bound the variance crudely using the fact that the sum of interest has indicator summands
    \begin{align*}
        \operatorname{Var}(S_{n}(x_{n}^{+})) = \sum_{i=1}^{n} (1-H_{n,i}(x_{n}^{+})) H_{n,i}(x_{n}^{+}) 
        \leq  \amsmathbb{E}[S_{n}(x_{n}^{+})] = (1-\delta_{n})k_{n} (1 + o(1)),
    \end{align*}
    and likewise $\operatorname{Var}(S_{n}(x_{n}^{-})) \leq (1+\delta_{n})k_{n} (1 + o(1))$. Chebyshev's inequality yields
    \begin{align*}
       &\amsmathbb{P}(S_{n}(x_{n}^{-}) \leq k_{n}) 
       \leq
       \amsmathbb{P}(\vert S_{n}(x_{n}^{-}) - \amsmathbb{E}[S_{n}(x_{n}^{-})]\vert \geq \amsmathbb{E}[S_{n}(x_{n}^{-})] - k_{n}) \\
       & \leq
       \frac{\operatorname{Var}(S_{n}(x_{n}^{-}))}{(\amsmathbb{E}[S_{n}(x_{n}^{-})] - k_{n})^{2}}
       \leq
       \frac{(1+\delta_{n})k_{n} (1 + o(1))}{k_{n}^{2}\delta_{n}^{2}}
       \leq \frac{(1+o(1))}{k_{n}\delta_{n}^{2}} \to 0.
    \end{align*}
    Hence $\amsmathbb{P}(Z_{n:n-k_{n}} \leq x_{n}^{-}) \to 0$. Similarly
    \begin{align*}
        \amsmathbb{P}(S_{n}(x_{n}^{+}) \geq k_{n} + 1) \leq 
         \frac{\operatorname{Var}(S_{n}(x_{n}^{+}))}{(\amsmathbb{E}[S_{n}(x_{n}^{+})] - k_{n} -1)^{2}}
          \leq \frac{(1+o(1))}{k_{n}\delta_{n}^{2}} \to 0,
    \end{align*}
    so that $\amsmathbb{P}(Z_{n:n-k_{n}} \geq x_{n}^{+}) \to 0$. We conclude that
    \begin{align*}
        \amsmathbb{P}(x_{n}^{-} < Z_{n:n-k_{n}} < x_{n}^{+}) \to 1. 
    \end{align*}
    Since $x \mapsto 1-H(x)$ is monotone, we get that
    \begin{align*}
        1-H(x_{n}^{+}) \leq 1-H(Z_{n:n-k_{n}}) \leq 1-H(x_{n}^{-}),
    \end{align*}
    with probability tending to one. Using the tail equivalence, the latter statement can be written as
    \begin{align*}
         (1-\delta_{n})\frac{k_{n}}{n} \frac{1}{C_{H}(1)} (1 + o(1)) \leq
         1-H(Z_{n:n-k_{n}}) \leq 
          (1+\delta_{n})\frac{k_{n}}{n} \frac{1}{C_{H}(1)} (1 + o(1)),
    \end{align*}
    with probability tending to one as $n\to \infty$. Since $\delta_{n} \to 0$ we then have that 
    \begin{align*}
        \frac{n}{k_{n}}(1-H(Z_{n:n-k_{n}})) \overset{\amsmathbb{P}}{\to} 1/C_{H}(1).
    \end{align*}
\end{proof}
In order to relate the variance between the empirical building blocks $\amsmathbb{H}_{n}(\cdot \vert s)$, $\amsmathbb{H}_{n}^{1}(\cdot \vert s)$ and their corresponding locally transported counterparts $\widetilde{\amsmathbb{H}}_{n}(\cdot \vert s)$, $\widetilde{\amsmathbb{H}}_{n}^{1}(\cdot \vert s)$ we use the following uniform Lipschitz condition in the design direction. In the remainder of this Appendix we choose $p_{s}$ such that~\eqref{eq_Z_n_k_n_eventually_contained} holds and define $T_{s}^{(n)}:= T_{s,p_{s}}^{(n)}$. 

\begin{lemma}\label{Lemma_var_H_tilde_minus_H}
     Assume Condition~\ref{cond_H_H1_C1_F_G_Lip} and fix $s \in (0,1)$. For every $x \leq T_{s}^{(n)}$ it then holds that 
      \begin{align*}
       &\operatorname{Var}(\widetilde{\amsmathbb{H}}_{n}(x \vert s)
         -
         \amsmathbb{H}_{n}(x \vert s))
          \leq
          \frac{5(M_{F} + M_{G})}{(nh_{n})^{2}} \sum_{i=1}^{n}  K\Big(\frac{s - i/n}{h_{n}}\Big)^{2}\vert i-\floor{ns}\vert /n.
          \\
        &\operatorname{Var}(\widetilde{\amsmathbb{H}}^{1}_{n}(x \vert s)
         -
         \amsmathbb{H}^{1}_{n}(x \vert s))
          \leq
          \frac{2M_{F} + 3M_{G}}{(nh_{n})^{2}} \sum_{i=1}^{n}  K\Big(\frac{s - i/n}{h_{n}}\Big)^{2}\vert i-\floor{ns}\vert /n.
    \end{align*}
\end{lemma}
\begin{proof}[Proof of Lemma~\ref{Lemma_var_H_tilde_minus_H}]
    We first prove the second of the above assertions. Due to independence we have
    \begin{align*}
         &\operatorname{Var}(\widetilde{\amsmathbb{H}}^{1}_{n}(x \vert s)
         -
         \amsmathbb{H}^{1}_{n}(x \vert s))
          \\ &=
          \frac{1}{(nh_{n})^{2}} \sum_{i=1}^{n}  K\Big(\frac{s - i/n}{h_{n}}\Big)^{2}
         \operatorname{Var} (1\{Z_{\floor{ns}}^{(n,i)}\leq x\}\delta_{\floor{ns}}^{(n,i)}
         -
         1\{Z_{i}^{(n)}\leq x\}\delta_{i}^{(n)}).
    \end{align*}
    The variance of each summand can be bounded as follows 
    \begin{align*}
         &\operatorname{Var} (1\{Z_{\floor{ns}}^{(n,i)}\leq x\}\delta_{\floor{ns}}^{(n,i)}
         -
         1\{Z_{i}^{(n)}\leq x\}\delta_{i}^{(n)})
         \leq
         \amsmathbb{E}[ (1\{Z_{\floor{ns}}^{(n,i)}\leq x\}\delta_{\floor{ns}}^{(n,i)}
         -
         1\{Z_{i}^{(n)}\leq x\}\delta_{i}^{(n)})^{2}] \\
         &=
          \amsmathbb{E}[ \vert 1\{Z_{\floor{ns}}^{(n,i)}\leq x\}\delta_{\floor{ns}}^{(n,i)}
         -
         1\{Z_{i}^{(n)}\leq x\}\delta_{i}^{(n)}\vert ] 
         =
        \amsmathbb{P}(1\{Z_{\floor{ns}}^{(n,i)}\leq x\}\delta_{\floor{ns}}^{(n,i)}
        \neq
         1\{Z_{i}^{(n)}\leq x\}\delta_{i}^{(n)}).
    \end{align*}
    Define $U_{i} := F_{n,i}(X_{i}^{(n)})$, and $V_{i} := G_{n,i}(C_{i}^{(n)})$ and note that these are independent standard uniform random variables. The latter probability can then be expressed as follows
    \begin{align*}
         &\amsmathbb{P}(1\{Z_{\floor{ns}}^{(n,i)}\leq x\}\delta_{\floor{ns}}^{(n,i)}
        \neq
         1\{Z_{i}^{(n)}\leq x\}\delta_{i}^{(n)})
         =
          \amsmathbb{P}(1\{X_{\floor{ns}}^{(n,i)}\leq x \wedge C_{\floor{ns}}^{(n,i)} \}
        \neq
         1\{X_{i}^{(n)}\leq x \wedge C_{i}^{(n)}\}) \\
         &=  \amsmathbb{P}(1\{U_{i}\leq F_{n,\floor{ns}}(x \wedge G_{n,\floor{ns}}^{\leftarrow}(V_{i}))\}
        \neq
        1\{U_{i}\leq F_{n,i}(x \wedge G_{n,i}^{\leftarrow}(V_{i}))\}) \\
        &=
        \int_{0}^{1} 
         \amsmathbb{P}(1\{U_{i}\leq F_{n,\floor{ns}}(x \wedge G_{n,\floor{ns}}^{\leftarrow}(v))\}
        \neq
        1\{U_{i}\leq F_{n,i}(x \wedge G_{n,i}^{\leftarrow}(v))\}) \mathrm{d} v \\
        &=
        \int_{0}^{1} 
        \vert  F_{n,\floor{ns}}(x \wedge G_{n,\floor{ns}}^{\leftarrow}(v))
        -
         F_{n,i}(x \wedge G_{n,i}^{\leftarrow}(v))
         \vert
         \mathrm{d} v.
    \end{align*}
    Assume without loss of generality that $G_{n,i}(x)\leq G_{n,\floor{ns}}(x)$. Then 
    \begin{align*}
        &\int_{0}^{1} 
        \vert  F_{n,\floor{ns}}(x \wedge G_{n,\floor{ns}}^{\leftarrow}(v))
        -
         F_{n,i}(x \wedge G_{n,i}^{\leftarrow}(v))
         \vert
         \mathrm{d} v \\
         &=
         \int_{0}^{G_{n,i}(x)} 
        \vert  F_{n,\floor{ns}}( G_{n,\floor{ns}}^{\leftarrow}(v))
        -
         F_{n,i}( G_{n,i}^{\leftarrow}(v))
         \vert
         \mathrm{d} v \\
         &\quad+
         \int_{G_{n,i}(x)}^{G_{n,\floor{ns}}(x)} 
        \vert  F_{n,\floor{ns}}( G_{n,\floor{ns}}^{\leftarrow}(v))
        -
         F_{n,i}(x)
         \vert
         \mathrm{d} v 
         +
         \vert  F_{n,\floor{ns}}(x )
        -
         F_{n,i}(x)
         \vert (1-G_{n,\floor{ns}}(x)).
    \end{align*}
    We use Condition~\ref{cond_H_H1_C1_F_G_Lip} to bound the above three terms. The second of the latter terms is bounded as follows
    \begin{align*}
         &\int_{G_{n,i}(x)}^{G_{n,\floor{ns}}(x)} 
        \vert  F_{n,\floor{ns}}( G_{n,\floor{ns}}^{\leftarrow}(v))
        -
         F_{n,i}(x)
         \vert
         \mathrm{d} v \\
         &\leq
         2\sup_{x\leq T_{s}^{(n)}} \vert G_{n,\floor{ns}}(x) - G_{n,i}(x) \vert 
         \leq  2M_{G}\frac{\vert i-\floor{ns} \vert}{n}.
    \end{align*}
    Likewise for the third term
    \begin{align*}
       \vert  F_{n,\floor{ns}}(x )
        -
         F_{n,i}(x)
         \vert (1-G_{n,\floor{ns}}(x))
         \leq
         \sup_{x\leq T_{s}^{(n)}} \vert F_{n,\floor{ns}}(x) - F_{n,i}(x) \vert \leq M_{F}\frac{\vert i-\floor{ns} \vert}{n}.
    \end{align*}
    We turn to the first term and note that
    \begin{align*}
        &\int_{0}^{G_{n,i}(x)} 
        \vert  F_{n,\floor{ns}}( G_{n,\floor{ns}}^{\leftarrow}(v))
        -
         F_{n,i}( G_{n,i}^{\leftarrow}(v))
         \vert
         \mathrm{d} v\\
         &\leq 
         \int_{0}^{G_{n,i}(x)} 
        \vert  F_{n,\floor{ns}}( G_{n,\floor{ns}}^{\leftarrow}(v))
        -
         F_{n,\floor{ns}}( G_{n,i}^{\leftarrow}(v))
         \vert
         \mathrm{d} v\\
         & \quad +
         \int_{0}^{G_{n,i}(x)} 
        \vert F_{n,\floor{ns}}( G_{n,i}^{\leftarrow}(v))
        -
         F_{n,i}( G_{n,i}^{\leftarrow}(v))
         \vert
         \mathrm{d} v
         \\
         &\leq
          \int_{0}^{G_{n,i}(x)} 
         \amsmathbb{P}(1\{U_{i}\leq F_{n,\floor{ns}}(G_{n,\floor{ns}}^{\leftarrow}(v))\}
        \neq
        1\{U_{i}\leq F_{n,i}( G_{n,i}^{\leftarrow}(v))\}) \mathrm{d} v \\
        & \quad  +\sup_{x\leq T_{s}^{(n)}} \vert F_{n,\floor{ns}}(x) - F_{n,i}(x) \vert \\
        &\leq
        \int_{0}^{G_{n,i}(x)} 
         \amsmathbb{P}(1\{G_{n,i}(X_{\floor{ns}}^{(n,i)})\leq v \}
        \neq
        1\{G_{n,\floor{ns}}(X_{\floor{ns}}^{(n,i)})\leq v \}) \mathrm{d} v
        +
          M_{F}\frac{\vert i-\floor{ns} \vert}{n},
    \end{align*}
    where we first used the triangle inequality followed by Condition~\ref{cond_H_H1_C1_F_G_Lip}. Applying Tonelli's Theorem on the first of the latter terms yields
    \begin{align*}
        &\int_{0}^{G_{n,i}(x)} 
         \amsmathbb{P}(1\{G_{n,i}(X_{\floor{ns}}^{(n,i)})\leq v \}
        \neq
        1\{G_{n,\floor{ns}}(X_{\floor{ns}}^{(n,i)})\leq v \}) \mathrm{d} v \\
        &=
        \int_{0}^{G_{n,i}(x)} 
         \amsmathbb{E}[ \vert1\{G_{n,i}(X_{\floor{ns}}^{(n,i)})\leq v \}
        -
        1\{G_{n,\floor{ns}}(X_{\floor{ns}}^{(n,i)})\leq v \}\vert ] \mathrm{d} v \\
        &=
        \amsmathbb{E}\Big[ \int_{0}^{G_{n,i}(x)} 
         \vert1\{G_{n,i}(X_{\floor{ns}}^{(n,i)})\leq v \}
        -
        1\{G_{n,\floor{ns}}(X_{\floor{ns}}^{(n,i)})\leq v \}\vert  \mathrm{d} v \Big] \\
        &=
        \amsmathbb{E}\big[ \vert (G_{n,i}(X_{\floor{ns}}^{(n,i)}) \wedge G_{n,i}(x))
        -
        (G_{n,\floor{ns}}(X_{\floor{ns}}^{(n,i)}) \wedge G_{n,i}(x))\vert \big].
    \end{align*}
    If $x \leq X_{\floor{ns}}^{(n,i)}$, then we have
    \begin{align*}
        & \amsmathbb{E}\big[\vert (G_{n,i}(X_{\floor{ns}}^{(n,i)}) \wedge G_{n,i}(x))
        -
        (G_{n,\floor{ns}}(X_{\floor{ns}}^{(n,i)}) \wedge G_{n,i}(x))\vert \big]\\
        &=
          \amsmathbb{E}\big[ G_{n,i}(x)
        -
        (G_{n,\floor{ns}}(X_{\floor{ns}}^{(n,i)}) \wedge G_{n,i}(x)) \big]  \leq
        \amsmathbb{E}\big[ G_{n,i}(x)
        -
        (G_{n,\floor{ns}}(x) \wedge G_{n,i}(x)) \big] = 0,
    \end{align*}
    since we assumed that $G_{n,i}(x)\leq G_{n,\floor{ns}}(x)$. If $X_{\floor{ns}}^{(n,i)} < x$ then
    \begin{align*}
        & \amsmathbb{E}\big[\vert (G_{n,i}(X_{\floor{ns}}^{(n,i)}) \wedge G_{n,i}(x))
        -
        (G_{n,\floor{ns}}(X_{\floor{ns}}^{(n,i)}) \wedge G_{n,i}(x))\vert \big]\\
        &\leq
        \amsmathbb{E}\big[\vert G_{n,i}(X_{\floor{ns}}^{(n,i)})
        -
        G_{n,\floor{ns}}(X_{\floor{ns}}^{(n,i)}) \vert \big]
        \leq
        \sup_{x\leq T_{s}^{(n)}} \vert G_{n,\floor{ns}}(x) - G_{n,i}(x) \vert 
         \leq  M_{G}\frac{\vert i-\floor{ns} \vert}{n},
    \end{align*}
    since $X_{\floor{ns}}^{(n,i)} < x \leq T_{s}^{(n)}$. To summarize we have now shown that for every $x\leq  T_{s}^{(n)}$ it holds that 
    \begin{align*}
        \operatorname{Var} (1\{Z_{\floor{ns}}^{(n,i)}\leq x\}\delta_{\floor{ns}}^{(n,i)}
         -
         1\{Z_{i}^{(n)}\leq x\}\delta_{i}^{(n)})
         \leq
         (2M_{F} +3M_{G})\frac{\vert i-\floor{ns} \vert}{n}.
    \end{align*}
    which in turn implies that
    \begin{align*}
        &\operatorname{Var}(\widetilde{\amsmathbb{H}}^{1}_{n}(x \vert s)
         -
         \amsmathbb{H}^{1}_{n}(x \vert s))
          \leq
          \frac{2M_{F} + 3M_{G}}{(nh_{n})^{2}} \sum_{i=1}^{n}  K\Big(\frac{s - i/n}{h_{n}}\Big)^{2}\vert i-\floor{ns}\vert /n.
    \end{align*}
    To establish the first assertion note that as above
    \begin{align*}
        \operatorname{Var} (1\{Z_{\floor{ns}}^{(n,i)}\leq x\}
         -
         1\{Z_{i}^{(n)}\leq x\})
         \leq
         \amsmathbb{E}[\vert 1\{Z_{\floor{ns}}^{(n,i)}\leq x\}
         -
         1\{Z_{i}^{(n)}\leq x\}\vert],
    \end{align*}
    where
    \begin{align*}
        &\vert 1\{Z_{\floor{ns}}^{(n,i)}\leq x\}
         -
         1\{Z_{i}^{(n)}\leq x\}\vert \leq
         \vert 1\{Z_{\floor{ns}}^{(n,i)}\leq x\}\delta_{\floor{ns}}^{(n,i)}
         -
         1\{Z_{i}^{(n)}\leq x\}\delta_{i}^{(n)}\vert \\
         &\quad +
         \vert 1\{Z_{\floor{ns}}^{(n,i)}\leq x\}(1-\delta_{\floor{ns}}^{(n,i)})
         -
         1\{Z_{i}^{(n)}\leq x\}(1-\delta_{i}^{(n)})\vert\\
         &=
         \vert 1\{X_{\floor{ns}}^{(n,i)}\leq x \wedge C_{\floor{ns}}^{(n,i)}\}
         -
         1\{X_{i}^{(n)}\leq x \wedge C_{i}^{(n)}\}\vert \\
         &\quad +
         \vert 1\{C_{\floor{ns}}^{(n,i)}\leq x \wedge X_{\floor{ns}}^{(n,i)}\}
         -
         1\{C_{i}^{(n)}\leq x \wedge X_{i}^{(n)}\}\vert. 
    \end{align*}
    Consequently we have that 
    \begin{align*}
         &\operatorname{Var}(\widetilde{\amsmathbb{H}}_{n}(x \vert s)
         -
         \amsmathbb{H}_{n}(x \vert s)) \\
         &\leq
         \frac{1}{(nh_{n})^{2}} \sum_{i=1}^{n}  K\Big(\frac{s - i/n}{h_{n}}\Big)^{2} \Big[
         \amsmathbb{P}(1\{X_{\floor{ns}}^{(n,i)}\leq x \wedge C_{\floor{ns}}^{(n,i)} \}
        \neq
         1\{X_{i}^{(n)}\leq x \wedge C_{i}^{(n)}\}) \\
         &\qquad \qquad \qquad \qquad+
          \amsmathbb{P}(1\{C_{\floor{ns}}^{(n,i)}\leq x \wedge X_{\floor{ns}}^{(n,i)} \}
        \neq
         1\{C_{i}^{(n)}\leq x \wedge X_{i}^{(n)}\})
         \Big],
    \end{align*}
    where the first sum is bounded exactly as $\operatorname{Var}(\widetilde{\amsmathbb{H}}^{1}_{n}(x \vert s) -\amsmathbb{H}^{1}_{n}(x \vert s))$ and the second sum is bounded similarly but with $F_{n,i}$ (respectively $F_{n,\floor{ns}}$) and $G_{n,i}$ (respectively $G_{n,\floor{ns}}$) playing opposite roles. This concludes the proof.
\end{proof}

\begin{lemma}\label{lemma_H_minus_H_tilde}
      Assume Condition~\ref{cond_H_H1_C1_F_G_Lip}. Then 
      \begin{align*}
      &\sup_{x \leq Z_{n:n-k_{n}}}\big\vert \widetilde{\amsmathbb{H}}_{n}(x \vert s)
         -
         \amsmathbb{H}_{n}(x \vert s) \big\vert 
        =
        \mathcal{O}_{\amsmathbb{P}}\big(\log(n) (h_{n}^{2} \vee 1/\sqrt{n} \vee 1/(nh_{n}))\big),\\
        &\sup_{x \leq Z_{n:n-k_{n}}}\big\vert \widetilde{\amsmathbb{H}}^{1}_{n}(x \vert s)
         -
         \amsmathbb{H}^{1}_{n}(x \vert s) \big\vert 
        =
        \mathcal{O}_{\amsmathbb{P}}\big(\log(n) (h_{n}^{2} \vee 1/\sqrt{n} \vee 1/(nh_{n}))\big),
    \end{align*}
    for every $s \in (0,1)$.
\end{lemma}
\begin{proof}[Proof of Lemma~\ref{lemma_H_minus_H_tilde}]
We prove the second statement and note that the first statement follows by analogous arguments. As above let $T_{s}^{(n)} := H_{n,\floor{ns}}^{\leftarrow}(1-p_{s}k_{n}/n)$, where $p_{s}$ is chosen in accordance with Lemma~\ref{lemma_Z_n_k_n_eventually_contained}. 
Lemma~\ref{Lemma_var_H_tilde_minus_H} yields that
    \begin{align*}
        &\operatorname{Var}(\widetilde{\amsmathbb{H}}^{1}_{n}(x \vert s)
         -
         \amsmathbb{H}^{1}_{n}(x \vert s))
          \leq
          \frac{2M_{F} + 3M_{G}}{(nh_{n})^{2}} \sum_{i=1}^{n}  K\Big(\frac{s - i/n}{h_{n}}\Big)^{2}\vert i-\floor{ns}\vert /n =
          \frac{M_{V}}{n}(1+o(1)),
    \end{align*}
    reusing the calculations from the proof of Lemma~\ref{lemma_H_n_and_H_n_1_unif_cons} for some strictly positive constant $M_{V}$. Define
    $D_{n}(x) := \widetilde{\amsmathbb{H}}^{1}_{n}(x \vert s)-\amsmathbb{H}^{1}_{n}(x \vert s)$. Since $x \mapsto D_{n}(x)$ is a piecewise constant function with (possible) jumps at the random times 
    \begin{align*}
        \xi_{s}^{n} := (Z_{1}^{(n)}, \ldots, Z_{n}^{(n)}, Z_{\floor{ns}}^{(n,1)}, \ldots, Z_{\floor{ns}}^{(n,n)}),
    \end{align*}
    we know that it attains its maximum in one of these random points. In particular, we have that 
    \begin{align*}
        &\{\sup_{x \leq T_{s}^{(n)}}\vert D_{n}(x) \vert > u\}
        =
        \big\{\max_{x \leq [0, T_{s}^{(n)}] \cap \xi_{s}^{n}}\vert D_{n}(x) \vert > u\big\}
        =
        \bigcup_{x\leq [0, T_{s}^{(n)}] \cap \xi_{s}^{n}}\Big(\{ \vert D_{n}(x) \vert > u \}\Big) \\
        & \subseteq \bigcup_{i=1}^{n}\Big(\{ \vert D_{n}(Z_{i}^{(n)} \wedge T_{s}^{(n)}) \vert > u \}\Big)
        \cup 
        \bigcup_{i=1}^{n}\Big(\{ \vert D_{n}(Z_{\floor{ns}}^{(n,i)} \wedge T_{s}^{(n)}) \vert > u \}\Big).
    \end{align*}
    A union bound then yields
    \begin{align*}
        \amsmathbb{P}\big(\sup_{x \leq T_{s}^{(n)}}\vert D_{n}(x) \vert > u\big)
        \leq
        \sum_{i=1}^{n} \amsmathbb{P}\big(\vert D_{n}(Z_{i}^{(n)} \wedge T_{s}^{(n)}) \vert > u\big)
        +
        \sum_{i=1}^{n} \amsmathbb{P}\big(\vert D_{n}(Z_{\floor{ns}}^{(n,i)} \wedge T_{s}^{(n)}) \vert > u\big).
    \end{align*}
    We consider now the task of bounding the first of the latter two terms. Each summand can be bounded as
    \begin{align*}
        &\amsmathbb{P}\big(\vert D_{n}(Z_{i}^{(n)} \wedge T_{s}^{(n)}) \vert > u\big)
        =
        \int_{0}^{x^{*}} \amsmathbb{P}\big(\vert D_{n}(z \wedge T_{s}^{(n)}) \vert > u \big\vert Z_{i}^{(n)} = z\big) H_{n,i}(\mathrm{d} z) \\
        &= 
        \int_{0}^{T_{s}^{(n)}} \amsmathbb{P}\big(\vert D_{n}(z ) \vert > u \big\vert Z_{i}^{(n)} = z\big) H_{n,i}(\mathrm{d} z)
        +
        \int_{T_{s}^{(n)}}^{x^{*}} \amsmathbb{P}\big(\vert D_{n}(T_{s}^{(n)}) \vert > u \big\vert Z_{i}^{(n)} = z\big) H_{n,i}(\mathrm{d} z).
    \end{align*}
    We focus on the first of the latter terms. Define 
    \begin{align*}
        D_{n}^{-i}(x) := D_{n}(x) -  \frac{1}{nh_{n}} K\Big(\frac{s - i/n}{h_{n}}\Big) 
         (1\{Z_{\floor{ns}}^{(n,i)}\leq x\}\delta_{\floor{ns}}^{(n,i)}
         -
         1\{Z_{i}^{(n)}\leq x\}\delta_{i}^{(n)}),
    \end{align*} 
    i.e. $D_{n}(x)$ with the $i$th summand removed. Remark that then $D_{n}(x) - D_{n}^{-i}(x)\leq 2M_{K}/(nh_{n})$ uniformly in $x$, and using the independence across $i$'s we get that
    \begin{align*}
        &\amsmathbb{P}\big(\vert D_{n}(z) \vert > u \big\vert Z_{i}^{(n)} = z\big) 
        =
        \amsmathbb{P}\big(\vert D_{n}^{-i}(z) + (D_{n}(z) - D_{n}^{-i}(z)) \vert > u \big\vert Z_{i}^{(n)} = z\big)\\
        &\leq  \amsmathbb{P}\big(\vert D_{n}^{-i}(z) \vert > u/2 \big\vert Z_{i}^{(n)} = z\big)
        +
        \amsmathbb{P}\big(\vert D_{n}(z) - D_{n}^{-i}(z) \vert > u/2 \big\vert Z_{i}^{(n)} = z\big) \\
        & \leq
        \amsmathbb{P}\big(\vert D_{n}^{-i}(z) \vert > u/2\big)
        +
        1\{2M_{K}/(nh_{n}) > u/2 \}.
    \end{align*}
    Suppose we pick $u > 4M_{K}/(nh_{n})$ so that the latter indicator is identically zero.
    If we then add and subtract the $i$th term to $D_{n}^{-i}(z)$, we have that 
    \begin{align*}
        \amsmathbb{P}\big(\vert D_{n}^{-i}(z) \vert > u/2\big)
        \leq
        \amsmathbb{P}\big(\vert D_{n}(z) \vert > u/4\big)
        +
        1\{2M_{K}/(nh_{n}) > u/4\},
    \end{align*}
    where the latter indicator is identically zero for $u > 8M_{K}/(nh_{n})$. For such $u$ it holds that 
    \begin{align*}
       &\int_{0}^{T_{s}^{(n)}} \amsmathbb{P}\big(\vert D_{n}(z ) \vert > u \big\vert Z_{i}^{(n)} = z\big) H_{n,i}(\mathrm{d} z)
        \leq
        \int_{0}^{T_{s}^{(n)}} \amsmathbb{P}\big(\vert D_{n}(z) \vert > u/4\big) H_{n,i}(\mathrm{d} z) \\
        & \leq
        \int_{0}^{T_{s}^{(n)}} 
        \amsmathbb{P}\big(\vert D_{n}(z) - \amsmathbb{E}[D_{n}(z)] \vert > u/8 \big) 
        + 1\{\vert\amsmathbb{E}[D_{n}(z)] \vert > u/8 \}  
        H_{n,i}(\mathrm{d} z).
    \end{align*}
    As above we have that $\amsmathbb{E}[D_{n}(z)] \leq \mathcal{O}(h_{n}^{2})$ uniformly in $z$, and consequently the choice of $u > Mh_{n}^{2} \vee 8M_{K}/(nh_{n})$ for some potentially large but finite $n$ and positive $M>0$, yields that the latter indicator is identically zero from that $n$ and onward. We focus then on the leading term, employing the Bernstein-inequality
    \begin{align*}
        &\int_{0}^{T_{s}^{(n)}} 
        \amsmathbb{P}\big(\vert D_{n}(z) - \amsmathbb{E}[D_{n}(z)] \vert > u/8 \big) 
        H_{n,i}(\mathrm{d} z) \leq
        \sup_{x \leq T_{s}^{(n)}} \amsmathbb{P}\big(\vert D_{n}(x) - \amsmathbb{E}[D_{n}(x)] \vert > u/8 \big) \\
        &\leq \sup_{x \leq T_{s}^{(n)}}\exp{\Big(- \frac{(u/8)^{2}/2}{\operatorname{Var}(D_{n}(x)) + 2uM_{k}/(3nh_{n})}\Big) }.
    \end{align*}
    Since $\operatorname{Var}(D_{n}(x)) = \operatorname{Var}(\widetilde{\amsmathbb{H}}^{1}_{n}(x \vert s)-\amsmathbb{H}^{1}_{n}(x \vert s)) = (M_{V}/n)(1+o(1))$ uniformly in $x\leq T_{s}^{(n)}$, we consider $n\in \amsmathbb{N}$ such that $\operatorname{Var}(D_{n}(x)) \leq 2M_{V}/n$ and so
    \begin{align*}
        \sup_{x \leq T_{s}^{(n)}}\exp{\Big(- \frac{(u/8)^{2}/2}{\operatorname{Var}(D_{n}(x)) + 2uM_{k}/(3nh_{n})}\Big) }
        \leq \exp{\Big(- \frac{(u/8)^{2}/2}{2M_{V}/n + 2uM_{k}/(3nh_{n})}\Big) }.
    \end{align*}
     We now specify 
    \begin{align*}
        u = \sqrt{C}\log(n) (h_{n}^{2} \vee 1/\sqrt{n} \vee 1/(nh_{n})),
    \end{align*}
    for some sufficiently large positive constant $C$ in accordance with the above requirements on $u$. We check that the above bound can be made asymptotically negligible for suitable choices of $C$. If the above maximum is asymptotically attained by $h_{n}^{2}$, i.e.  $nh_{n}^{4} \to \infty$, then for sufficiently large $n\in \amsmathbb{N}$ it holds
    \begin{align*}
        &\exp{\Big(- \frac{(u/8)^{2}/2}{2M_{V}/n + 2uM_{k}/(3nh_{n})}\Big)} 
        =
        \exp{\Big(- \frac{C\log(n)^{2}h_{n}^{4}/128}{2M_{v}/n+ 2\sqrt{C}\log(n)h_{n}M_{k}/(3n)}\Big) } \\
        &\leq
        \exp{\Big(- \frac{\log(n)^{2}h_{n}^{4}C/128}{4M_{v}/n}\Big) }
        =
        n^{-\log(n)nh_{n}^{4}C/512} \to 0,
    \end{align*}
    as $n\to \infty$, using that $nh_{n}^{4} \to \infty$. Likewise if $(1/\sqrt{n})$ is the sequence of slowest decay i.e. $1/(nh_{n}^{4}) +  \sqrt{n}h_{n}\to \infty$, then for $n$ sufficiently large 
    \begin{align*}
        &\exp{\Big(- \frac{(u/8)^{2}/2}{2M_{V}/n + 2uM_{k}/(3nh_{n})}\Big)}  =
        \exp{\Big(- \frac{C\log(n)^{2}/128n}{2M_{v}/n+ 2\sqrt{C}\log(n)M_{k}/(3n^{3/2}h_{n})}\Big) }\\
        &\leq
        \exp{\Big(- \frac{C\log(n)^{2}/128n}{2M_{v}/n+ 2\sqrt{C}\log(n)M_{k}/(3n)}\Big) }
        \leq
        \exp{\Big(- \frac{C\log(n)^{2}/128n}{4\sqrt{C}\log(n)M_{k}/(3n)}\Big) }\\
        &
        =
        n^{- \sqrt{C}3/(M_{k}512)} \to 0,
    \end{align*}
    for $C>(M_{K}512/3)^{2}$.  Finally if $(1/(nh_{n}))$ is the sequence of slowest decay i.e. $1/(nh_{n}^{4}) +  1/(\sqrt{n}h_{n})\to \infty$, then for $n$ sufficiently large 
    \begin{align*}
         &\exp{\Big(- \frac{(u/8)^{2}/2}{2M_{V}/n + 2uM_{k}/(3nh_{n})}\Big)} 
         =
         \exp{\Big(- \frac{C\log(n)^{2}/128(nh_{n})^{2}}{2M_{v}/n+ 2\sqrt{C}\log(n)M_{k}/(3(nh_{n})^{2})}\Big) } \\
         &\leq
          \exp{\Big(- \frac{C\log(n)^{2}/128(nh_{n})^{2}}{4\sqrt{C}\log(n)M_{k}/(3(nh_{n})^{2})}\Big) }
          =
           n^{- \sqrt{C}3/(M_{k}512)} \to 0,
    \end{align*}
    as above.

    We may repeat identical arguments on the second integral; subtracting the $i$th term, applying a union bound, then adding it again and applying the Bernstein inequality where the variance term instead is $\operatorname{Var}(D_{n}(T_{s}^{(n)}))$, which can be bounded as above. As above let $C>(M_{K}512/3)^{2}$ and hence, for $n$ sufficiently large, it holds that
    \begin{align*}
         \sum_{i=1}^{n} \amsmathbb{P}\big(\vert D_{n}(Z_{i}^{(n)} \wedge T_{s}^{(n)}) \vert > u\big)
        &\leq
        n  \exp{\Big(- \frac{(u/8)^{2}/2}{2M_{V}/n + 2uM_{k}/(3nh_{n})}\Big) }
        \\
        &\leq
         n^{1- \sqrt{C}3/(M_{k}512)}(1+o(1)) \to 0,
    \end{align*}
    as $n\to \infty$.
    Identical arguments yield that
    \begin{align*}
         \sum_{i=1}^{n} \amsmathbb{P}\big(\vert D_{n}(Z_{\floor{ns}}^{(n,i)} \wedge T_{s}^{(n)}) \vert > u\big)
         \to 0.
    \end{align*}
    In particular, for any given $\varepsilon > 0$ we can choose $C > 0$ such that
    \begin{align*} 
         \amsmathbb{P}\Big(\sup_{x \leq T_{s}^{(n)}}\vert D_{n}(x) \vert >  \sqrt{C}\log(n) (h_{n}^{2} \vee 1/\sqrt{n} \vee 1/(nh_{n}))\varepsilon \Big) \to 0,
    \end{align*}
    and we may conclude that
    \begin{align*}
         \sup_{x \leq T_{s}^{(n)}}\big\vert \widetilde{\amsmathbb{H}}^{1}_{n}(x \vert s)
         -
         \amsmathbb{H}^{1}_{n}(x \vert s) \big\vert 
         = \mathcal{O}_{\amsmathbb{P}}\big(\log(n) (h_{n}^{2} \vee 1/\sqrt{n} \vee 1/(nh_{n}))\big).
    \end{align*}
    Lemma~\ref{lemma_Z_n_k_n_eventually_contained} then yields that the interval $[0, Z_{n:n-k_{n}}]$ can be covered by $[0, T_{s}^{(n)}]$ with probability going to one.
\end{proof}

\begin{lemma}\label{lemma_H_minus_H_gamma_fraction_rate}
    Assume Condition~\ref{cond_H_H1_C1_F_G_Lip} and that $nh_{n}^{2}/k_{n}$ and $\sqrt{nh_{n}}/k_{n}$ are decreasing. For every $s\in (0,1)$ and $\gamma \in [0,1/2)$ it holds that 
    \begin{align*}
        &\sup_{x \leq T_{s}^{(n)}}\frac{\vert\amsmathbb{H}_{n}(x- \vert s) - H_{n,\floor{ns}}(x-)\vert}{(1-H_{n,\floor{ns}}(x-))^{1-\gamma}} =
        \mathcal{O}\Big(
        \Big(\frac{n}{k_{n}}\Big)^{1-\gamma}(h_{n}^{2} 
        + 
        \frac{\sqrt{h_{n}}}{\sqrt{n}})
        +
        \frac{1}{\sqrt{nh_{n}}}\Big(\frac{n}{k_{n}}\Big)^{(1-2\gamma)/2}
        \Big).
    \end{align*}
\end{lemma}
\begin{proof} [Proof of Lemma~\ref{lemma_H_minus_H_gamma_fraction_rate}] 
    Using the triangle inequality we get that 
    \begin{align*}
        &\sup_{x \leq T_{s}^{(n)}}
        \frac{\vert\amsmathbb{H}_{n}(x- \vert s) - H_{n,\floor{ns}}(x-)\vert}{(1-H_{n,\floor{ns}}(x-))^{1-\gamma}}\\
        &\leq
        \sup_{x \leq T_{s}^{(n)}}
        \frac{\vert\amsmathbb{H}_{n}(x- \vert s) - \amsmathbb{E}[\amsmathbb{H}_{n}(x- \vert s)]\vert}{(1-H_{n,\floor{ns}}(x-))^{1-\gamma}}
        +
        \sup_{x \leq T_{s}^{(n)}}
        \frac{\vert H_{n,\floor{ns}}(x-) - \amsmathbb{E}[\amsmathbb{H}_{n}(x- \vert s)]\vert}{(1-H_{n,\floor{ns}}(x-))^{1-\gamma}}.
    \end{align*}
    Crudely bounding the latter of the two terms by computing the supremum in the numerator and infimum in the denominator, yields that the latter term is of order $ h_{n}^{2}(n/k_{n})^{1-\gamma}$.\\
    This concludes the deterministic bias term. For the random centered term we use a symmetrization argument, with $\mu_{i}(x) = K((s-i/n)/h_{n})(H_{n,i}(x-)-1)/(1-H_{n,\floor{ns}}(x-))^{1-\gamma}$ when analyzing the mean squared error 
    \begin{align*}
        &\amsmathbb{E}\Big[\sup_{x \leq T_{s}^{(n)}} \Big( \frac{1}{nh_{n}} \sum_{i=1}^{n}  K\Big(\frac{s - i/n}{h_{n}}\Big) 
        \frac{(H_{n,i}(x-) - 1\{Z_{i}^{(n)}\leq x-\} )}{(1-H_{n,\floor{ns}}(x-))^{1-\gamma}} \Big)^{2}\Big] \\
        &\leq
        4  \amsmathbb{E}\Big[\sup_{x \leq T_{s}^{(n)}}\Big( \frac{1}{nh_{n}}\sum_{i=1}^{n}\varepsilon_{i} \Big(  \frac{(H_{n,i}(x-) - 1\{Z_{i}^{(n)}\leq x-\} )}{(1-H_{n,\floor{ns}}(x-))^{1-\gamma}}  - \mu_{i}(u)\Big)
        K\Big(\frac{s-i/n}{h_{n}}\Big)\Big)^{2}\Big] \\
        &=
        4  \amsmathbb{E}\Big[\sup_{x \leq T_{s}^{(n)}}\Big( \frac{1}{nh_{n}}\sum_{i=1}^{n}\varepsilon_{i} 
        \frac{1\{Z_{i}^{(n)} \geq x\}}{(1-H_{n,\floor{ns}}(x-))^{1-\gamma}}
        K\Big(\frac{s-i/n}{h_{n}}\Big)\Big)^{2}\Big]. 
    \end{align*}
    Let $i_{1},\ldots, i_{n}$ be the indices corresponding to the order statistics $Z_{n:1},\ldots , Z_{n:n}$. Define
     \begin{align*}
        D(x):= \frac{1}{(1-H_{n,\floor{ns}}(x-))^{1-\gamma}},
    \end{align*}
    and let $D_{m}:= D(Z_{n:m} \wedge T_{s}^{(n)})$ for $1\leq m \leq n$ which we note is non-decreasing in $m$. Similarly we define the partial sums
    \begin{align*}
        S_{m} := \sum_{j=m}^{n}\varepsilon_{i_{j}} K\Big(\frac{s-i_{j}/n}{h_{n}}\Big), \quad m=1,\ldots, n.
    \end{align*}
    For any given $x \leq T_{s}^{(n)}$, let $m$ be the first index for which $x \leq T_{s}^{(n)} \wedge Z_{i_{m}}^{(n)}$. For that $x$ we have
    \begin{align*}
        \sum_{i=1}^{n}\varepsilon_{i} 
        1\{Z_{i}^{(n)} \geq x\}
        K\Big(\frac{s-i/n}{h_{n}}\Big)
        =
         \sum_{j=m}^{n}\varepsilon_{i_{j}} 
        K\Big(\frac{s-i_{j}/n}{h_{n}}\Big),
    \end{align*}
    and by monotonicity $D(x) \leq D(Z_{i_{m}}^{(n)} \wedge T_{s}^{(n)})$. Consequently
    \begin{align*}
        \Big(D(x)\sum_{i=1}^{n}\varepsilon_{i} 
        1\{Z_{i}^{(n)} \geq x\}
        K\Big(\frac{s-i/n}{h_{n}}\Big)\Big)^{2}
        &\leq
        \Big(D(Z_{i_{m}}^{(n)} \wedge T_{s}^{(n)})\sum_{j=m}^{n}\varepsilon_{i_{j}} 
        K\Big(\frac{s-i_{j}/n}{h_{n}}\Big)\Big)^{2}\\
        &=
        D_{m}^{2}S_{m}^{2},
    \end{align*}
    and in particular 
    \begin{align*}
        &\sup_{x \leq Z_{n:n-k_{n}}}\Big(  \sum_{i=1}^{n}\varepsilon_{i} 
        \frac{1\{Z_{i}^{(n)} \geq x\}}{(1-H_{n,\floor{ns}}(x-))^{1-\gamma}}
        K\Big(\frac{s-i/n}{h_{n}}\Big) \Big )^{2} \leq
        \max_{1\leq m \leq n} D_{m}^{2}S_{m}^{2}. 
    \end{align*}
    Using that $m \mapsto D_{m}$ is non-decreasing and defining $D_{0} \equiv 0$, we bound the latter argument as follows 
    \begin{align*}
        D_{m}^{2}S_{m}^{2}
        =
         \sum_{l=1}^{m} (D_{l}^{2} - D_{l-1}^{2})S_{m}^{2}
         \leq
        \sum_{l=1}^{n} (D_{l}^{2} - D_{l-1}^{2})\max_{l\leq r \leq n}S_{r}^{2},
    \end{align*}
    where we note that the right-hand side does not depend on $m$. Consequently
    \begin{align*}
        \amsmathbb{E}\big[\max_{1\leq m \leq n} D_{m}^{2}S_{m}^{2}\big]
        \leq
          \amsmathbb{E}\Big[ \sum_{l=1}^{n} (D_{l}^{2} - D_{l-1}^{2})\max_{l\leq r \leq n}S_{r}^{2}\Big]
          =
           \amsmathbb{E}\Big[\sum_{l=1}^{n}  (D_{l}^{2} - D_{l-1}^{2}) \amsmathbb{E}\big[ \max_{l\leq r \leq n}S_{r}^{2}\vert \overline{Z}_{n} \big]\Big],
    \end{align*}
    where $\overline{Z}_{n}:=(Z_{1}^{(n)},\ldots, Z_{n}^{(n)})$. We note that on $\sigma(\overline{Z}_{n})$, the process $(S_{m})$ is a reverse-time martingale with respect to $\mathcal{G}_{m} := \sigma(\overline{Z}_{n}, \varepsilon_{i_{m}},\ldots, \varepsilon_{i_{n}})$; it is the reversed sum of the martingale used in the proof of Lemma~\ref{lemma_H_n_and_H_n_1_unif_cons}.
    In particular, Doob's inequality yields that 
    \begin{align*}
        \amsmathbb{E}\big[ \max_{l\leq r \leq n}S_{r}^{2}\vert \overline{Z}_{n} \big]
        \leq
        4\amsmathbb{E}\big[ S_{l}^{2} \vert \overline{Z}_{n} \big]
        =
          4\amsmathbb{E}\Big[ \sum_{j=l}^{n} K^{2}\Big(\frac{s-i_{j}/n}{h_{n}} \Big) \vert \overline{Z}_{n} \Big],
    \end{align*}
    $\amsmathbb{P}$-a.s., which in turn implies that
    \begin{align*}
        &\amsmathbb{E}\big[\max_{1\leq m \leq n} D_{m}^{2}S_{m}^{2}\big]
        \leq
        4\amsmathbb{E}\Big[\sum_{l=1}^{n}  (D_{l}^{2} - D_{l-1}^{2}) \amsmathbb{E}\Big[ \sum_{j=l}^{n} K^{2}\Big(\frac{s-i_{j}/n}{h_{n}} \Big) \vert \overline{Z}_{n} \Big]\Big]\\
        &=
         4\amsmathbb{E}\Big[\sum_{l=1}^{n}  (D_{l}^{2} - D_{l-1}^{2}) \sum_{j=l}^{n} K^{2}\Big(\frac{s-i_{j}/n}{h_{n}} \Big) \Big]
         =
         4\amsmathbb{E}\Big[\sum_{j=1}^{n} K^{2}\Big(\frac{s-i_{j}/n}{h_{n}} \Big) \sum_{l=1}^{j}  (D_{l}^{2} - D_{l-1}^{2}) \Big] \\
         &
         =
         4\amsmathbb{E}\Big[\sum_{j=1}^{n} K^{2}\Big(\frac{s-i_{j}/n}{h_{n}} \Big) D_{j}^{2}\Big]
         =
         4\amsmathbb{E}\Big[\sum_{j=1}^{n} K^{2}\Big(\frac{s-i_{j}/n}{h_{n}} \Big) \sup_{x\leq T_{s}^{(n)}}1\{Z_{i_{j}} \geq x\}D(x)\Big] \\
         &=
         4\amsmathbb{E}\Big[\sum_{i=1}^{n} K^{2}\Big(\frac{s-i/n}{h_{n}} \Big) \sup_{x\leq T_{s}^{(n)}}1\{Z_{i} \geq x\}D(x)\Big] \\
         &=
         \sum_{i=1}^{n} K^{2}\Big(\frac{s-i/n}{h_{n}} \Big)\amsmathbb{E}\Big[ \sup_{x \leq  T_{s}^{(n)}}
        \frac{1\{Z_{i}^{(n)} \geq x\}}{(1-H_{n,\floor{ns}}(x-))^{2(1-\gamma)}}\Big]. 
    \end{align*}
    We have now shown that
    \begin{align*}
        &4 \amsmathbb{E}\Big[\sup_{x \leq T_{s}^{(n)}}\Big( \frac{1}{nh_{n}}\sum_{i=1}^{n}\varepsilon_{i} 
        \frac{1\{Z_{i}^{(n)} \geq x\}}{(1-H_{n,\floor{ns}}(x-))^{1-\gamma}}
        K\Big(\frac{s-i/n}{h_{n}}\Big)\Big)^{2}\Big]\\
        &\leq
        \frac{16}{(nh_{n})^{2}} \sum_{i=1}^{n}\amsmathbb{E}\Big[ \sup_{x \leq  T_{s}^{(n)}}
        \frac{1\{Z_{i}^{(n)} \geq x\}}{(1-H_{n,\floor{ns}}(x-))^{2(1-\gamma)}}\Big]
        K^{2}\Big(\frac{s-i/n}{h_{n}}\Big).
    \end{align*}
    We proceed with analyzing the expectation in each summand. First note that 
    \begin{align*}
        &\amsmathbb{E}\Big[\sup_{x \leq T_{s}^{(n)}}
        \frac{1\{Z_{i}^{(n)} \geq x\}}{(1-H_{n,\floor{ns}}(x-))^{2(1-\gamma)}}
        \Big]\\
        &=
        \amsmathbb{E}\Big[\sup_{H_{n,\floor{ns}}(x-) \leq 1-p_{s}k_{n}/n}
        \frac{1\{H_{n,\floor{ns}}(Z_{i}^{(n)}) > H_{n,\floor{ns}}(x-)\}}{(1-H_{n,\floor{ns}}(x-))^{2(1-\gamma)}}
        \Big],
    \end{align*}
    using continuity of $x \mapsto H_{n,\floor{ns}}(x)$. On the event $\{H_{n,\floor{ns}}(Z_{i}^{(n)}) > 1-p_{s}k_{n}/n\}$ the numerator equals one and
    \begin{align*}
        &\sup_{H_{n,\floor{ns}}(x-) \leq 1-p_{s}k_{n}/n}\frac{1\{H_{n,\floor{ns}}(Z_{i}^{(n)}) \geq  H_{n,\floor{ns}}(x-)\}}{(1-H_{n,\floor{ns}}(x-))^{2(1-\gamma)}}\\
        &=
        \sup_{H_{n,\floor{ns}}(x-) \leq 1-p_{s}k_{n}/n}\frac{1}{(1-H_{n,\floor{ns}}(x-))^{2(1-\gamma)}}
        =
       \Big(\frac{n}{p_{s}k_{n}}\Big)^{2(1-\gamma)}.
    \end{align*}
    On the event $\{H_{n,\floor{ns}}(Z_{i}^{(n)}) \leq 1-p_{s}k_{n}/n\}$, the supremum is attained as $H_{n,\floor{ns}}(x-)$ approaches $H_{n,\floor{ns}}(Z_{i}^{(n)})$ from below and hence
    \begin{align*}
        \sup_{H_{n,\floor{ns}}(x-) \leq 1-p_{s}k_{n}/n}
        \frac{1\{H_{n,\floor{ns}}(Z_{i}^{(n)}) > H_{n,\floor{ns}}(x-)\}}{(1-H_{n,\floor{ns}}(x-))^{2(1-\gamma)}}
        \leq
         \frac{1}{(1-H_{n,\floor{ns}}(Z_{i}^{(n)}))^{2(1-\gamma)}}.
    \end{align*}
    In particular, it holds that 
    \begin{align*}
        &\amsmathbb{E}\Big[\sup_{H_{n,\floor{ns}}(x-) \leq 1-p_{s}k_{n}/n}
        \frac{1\{H_{n,\floor{ns}}(Z_{i}^{(n)}) > H_{n,\floor{ns}}(x-)\}}{(1-H_{n,\floor{ns}}(x-))^{2(1-\gamma)}}
        \Big]
        \\
        &\leq
        \amsmathbb{E}\Big[
        \Big(\frac{n}{p_{s}k_{n}}\Big)^{2(1-\gamma)}1\{H_{n,\floor{ns}}(Z_{i}^{(n)}) > 1-p_{s}k_{n}/n\} \\
         & \qquad +
          \frac{1}{(1-H_{n,\floor{ns}}(Z_{i}^{(n)}))^{2(1-\gamma)}}1\{H_{n,\floor{ns}}(Z_{i}^{(n)}) \leq 1-p_{s}k_{n}/n\}
        \Big] \\
        &=
       \Big(\frac{n}{p_{s}k_{n}}\Big)^{2(1-\gamma)}\amsmathbb{P}(H_{n,\floor{ns}}(Z_{i}^{(n)}) > 1-p_{s}k_{n}/n) 
        +
        \int_{0}^{T_{s}^{(n)}}(1-H_{n,\floor{ns}}(z))^{-2(1-\gamma)} H_{n,i}(\mathrm{d} z) \\
        &\leq
        \Big(\frac{n}{p_{s}k_{n}}\Big)^{2(1-\gamma)}(1-H_{n,i}(H_{n,\floor{ns}}^{\leftarrow}(1-p_{s}k_{n}/n))) +  \int_{0}^{1-p_{s}k_{n}/n}(1-u)^{-2(1-\gamma)} \mathrm{d} u \\
         &\qquad 
         +    \int_{0}^{T_{s}^{(n)}}(1-H_{n,\floor{ns}}(z))^{-2(1-\gamma)} (H_{n,i}(\mathrm{d} z)-H_{n,\floor{ns}}(\mathrm{d} z))\\
         &\leq
        \Big(\frac{n}{p_{s}k_{n}}\Big)^{2(1-\gamma)}(p_{s}k_{n}/n - (H_{n,i}(H_{n,\floor{ns}}^{\leftarrow}(1-p_{s}k_{n}/n)) - H_{n,\floor{ns}}(H_{n,\floor{ns}}^{\leftarrow}(1-p_{s}k_{n}/n)))) \\
         &\qquad 
         +  \Big(\frac{n}{p_{s}k_{n}}\Big)^{2(1-\gamma)}\sup_{x\leq T_{s}^{(n)}} \vert H_{n,i}(x) - H_{n,\floor{ns}}(x) \vert  + \int_{p_{s}k_{n}/n}^{1}v^{-2(1-\gamma)} \mathrm{d} v\\
         &\leq 
        \Big(\frac{n}{p_{s}k_{n}}\Big)^{1-2\gamma} 
         +
         \frac{(n/(p_{s}k_{n}))^{1-2\gamma} -1}{1-2\gamma} \\
         & \qquad +
         2\Big(\frac{n}{p_{s}k_{n}}\Big)^{2(1-\gamma)} 
         \sup_{x\leq T_{s}^{(n)}} \vert H_{n,i}(x) - H_{n,\floor{ns}}(x) \vert
         \\
         &=
         \Big(\frac{n}{p_{s}k_{n}}\Big)^{1-2\gamma}\Big(1 + \frac{1}{1-2\gamma} +  \Big(\frac{n}{p_{s}k_{n}}\Big)\sup_{x\leq T_{s}^{(n)}} \vert H_{n,i}(x) - H_{n,\floor{ns}}(x) \vert \Big) - \frac{1}{1-2\gamma}\\
         &=
         \mathcal{O}\Big(\frac{n}{p_{s}k_{n}}\Big)^{1-2\gamma}
         +
         \mathcal{O}\Big(\frac{n}{p_{s}k_{n}}\Big)^{2(1-\gamma)}\sup_{x\leq T_{s}^{(n)}} \vert H_{n,i}(x) - H_{n,\floor{ns}}(x) \vert,
    \end{align*}
    since the first of the latter terms is growing as $n$ increases due to the fact that $\gamma \in [0,1/2)$. Returning to the mean squared error of interest, we now have
    \begin{align*}
        &\amsmathbb{E}\Big[\sup_{x \leq T_{s}^{(n)}} \Big( \frac{1}{nh_{n}} \sum_{i=1}^{n}  K\Big(\frac{s - i/n}{h_{n}}\Big) 
        \frac{(H_{n,\floor{ns}}(x-) - 1\{Z_{i}^{(n)}\leq x-\} )}{(1-H_{n,\floor{ns}}(x-))^{1-\gamma}} \Big)^{2}\Big] \\
        &\leq
       \frac{16}{(nh_{n})^{2}}  \sum_{i=1}^{n}\amsmathbb{E}\Big[\sup_{x \leq T_{s}^{(n)}}
        \frac{1\{Z_{i}^{(n)} \geq x\}}{(1-H_{n,\floor{ns}}(x-))^{2(1-\gamma)}}
        \Big] K^{2}\Big(\frac{s-i/n}{h_{n}}\Big) \\
        &=
        \mathcal{O} \Big(\frac{1}{(nh_{n})^{2}}\Big(\frac{n}{k_{n}}\Big)^{1-2\gamma}\Big)
          \sum_{i=1}^{n} K^{2}\Big(\frac{s-i/n}{h_{n}}\Big)\Big( 1 + \Big(\frac{n}{k_{n}}\Big)
          \sup_{x\leq T_{s}^{(n)}} \vert H_{n,i}(x) - H_{n,\floor{ns}}(x) \vert \Big) \\
          &=
          \mathcal{O} \Big(\frac{1}{nh_{n}}\Big(\frac{n}{k_{n}}\Big)^{1-2\gamma}\Big)
          +
          \mathcal{O} \Big(\frac{h_{n}}{n}\Big(\frac{n}{k_{n}}\Big)^{2(1-\gamma)}\Big).
    \end{align*}
    Consequently
    \begin{align*}
         &\sup_{x \leq T_{s}^{(n)}}
        \frac{\vert\amsmathbb{H}_{n}(x- \vert s) - H_{n,\floor{ns}}(x-)\vert}{(1-H_{n,\floor{ns}}(x-))^{1-\gamma}}\\
        &= \mathcal{O}_{\amsmathbb{P}}\Big(
        \Big(\frac{n}{k_{n}}\Big)^{1-\gamma}(h_{n}^{2} 
        + 
        \frac{\sqrt{h_{n}}}{\sqrt{n}})
        +
        \frac{1}{\sqrt{nh_{n}}}\Big(\frac{n}{k_{n}}\Big)^{(1-2\gamma)/2}
        \Big).
    \end{align*}
    Lemma~\ref{lemma_Z_n_k_n_eventually_contained} yields that the interval $[0, Z_{n:n-k_{n}}]$ can be covered by $[0, T_{s}^{(n)}] $ with probability going to one. This finishes the proof.
\end{proof}

\begin{lemma}\label{lemma_H_tilde_tail_fraction}
   Assume Condition~\ref{cond_H_H1_C1_F_G_Lip}. For any $s\in (0,1)$ it then holds that 
   \begin{align*}
    \sup_{x\leq Z_{n:n-k_{n}}}
    \Big|
        \frac{1-\widetilde{\amsmathbb{H}}_{n}(x-\vert s)}
        {1-H_{n,\floor{ns}}(x-)}
        -1
    \Big|
    =
    \mathcal{O}_{\amsmathbb{P}}\Big(\frac{1}{\sqrt{k_{n}h_{n}}} \Big).
\end{align*}
\end{lemma}
\begin{proof}[Proof of Lemma~\ref{lemma_H_tilde_tail_fraction}]
   Note that
    \begin{align*}
        &\frac{1-\widetilde{\amsmathbb{H}}_{n}(x- \vert s)}{1- H_{n,\floor{ns}}(x-)} -1
        \frac{H_{n,\floor{ns}}(x-)-\amsmathbb{E}[\widetilde{\amsmathbb{H}}_{n}(x- \vert s)]}{1-H_{n,\floor{ns}}(x-)}
        +
        \frac{\amsmathbb{E}[\widetilde{\amsmathbb{H}}_{n}(x- \vert s)]- \widetilde{\amsmathbb{H}}_{n}(x- \vert s)}{1-H_{n,\floor{ns}}(x-)}.
    \end{align*}
    The first of the latter terms can be bounded for any $ x \leq T_{s}^{(n)}$ by
    \begin{align*}
        \sup_{x \leq T_{s}^{(n)}}\frac{H_{n,\floor{ns}}(x)}{1-H_{n,\floor{ns}}(x)} \Big\vert
        1 - \frac{1}{nh_{n}}\sum_{i=1}^{n}K\Big(\frac{s-i/n}{h_{n}}\Big)
        \Big\vert
        \leq
        \Big(\frac{n}{p_{s}k_{n}}\Big)\mathcal{O}\Big(\frac{1}{nh_{n}}\Big)
        =
        \mathcal{O}\Big(\frac{1}{k_{n}h_{n}}\Big).
    \end{align*}
    The latter of the above terms can be bounded by repeating arguments similar to those in the proof of Lemma~\ref{lemma_H_minus_H_gamma_fraction_rate}, but noting that every indicator function in $\widetilde{\amsmathbb{H}}_{n}(x- \vert s)$ has expected value identical to $H_{n,\floor{ns}}(x-)$, which is also present in the denominator. This is in contrast to the mean squared error calculations in Lemma~\ref{lemma_H_minus_H_gamma_fraction_rate}, where the fact that $\amsmathbb{H}_{n}(x- \vert s)$ is constructed from independent but non-identically distributed variables with mean $H_{n,i}(x-)$, induces a bias contribution of order $\sqrt{nh_{n}}/k_{n}$. This contribution is not present here and hence 
    \begin{align*}
         \sup_{x \leq T_{s}^{(n)}}  \Big\vert\frac{\amsmathbb{E}[\widetilde{\amsmathbb{H}}_{n}(x- \vert s)]- \widetilde{\amsmathbb{H}}_{n}(x- \vert s)}{1-H_{n,\floor{ns}}(x-)}  \Big\vert
         =
         \mathcal{O}_{\amsmathbb{P}}\Big(\frac{1}{\sqrt{k_{n}h_{n}}}\Big).
    \end{align*}
    Consequently
    \begin{align*}
    \sup_{x\leq Z_{n:n-k_{n}}}
    \Big|
        \frac{1-\widetilde{\amsmathbb{H}}_{n}(x-\vert s)}
        {1-H_{n,\floor{ns}}(x-)}
        -1
    \Big|
    =
     \mathcal{O}_{\amsmathbb{P}}\Big(\frac{1}{k_{n}h_{n}} + \frac{1}{\sqrt{k_{n}h_{n}}} \Big)
     =
     \mathcal{O}_{\amsmathbb{P}}\Big(\frac{1}{\sqrt{k_{n}h_{n}}} \Big)
    \end{align*}

    Lemma~\ref{lemma_Z_n_k_n_eventually_contained} yields that the interval $[0, Z_{n:n-k_{n}}]$ can be covered by $[0, T_{s}^{(n)}] $ with probability going to one. This finishes the proof.
    \end{proof}
    As noted in Section~\ref{section_main_results}, this result is directly comparable to the rate $\mathcal{O}_{\amsmathbb{P}}(1/\sqrt{k_{n}})$ in Lemma 1 of~\cite{Csorgo_96} for $\gamma = 0$; however, we pay an additional factor of $1/\sqrt{h_{n}}$ in the rate since we use local kernel estimators.

\newpage
\section{Additional simulation results}\label{appendix_simulations}
In this section we provide additional details on the simulation results in Section~\ref{section_simulation_study} and provide some complementary simulation studies.

\subsection{Additional details of Section~\ref{section_simulation_study}}
In Section~\ref{section_simulation_study} we consider the heteroscedastic Lomax, Fr\'echet and Burr models. Below we restate the constructions and argue why they admit heteroscedastic distribution functions in the sense of~\eqref{eq_Scedasis_F_and_G}.
\begin{enumerate}
   \item Heteroscedastic Lomax (Pareto II): for $\alpha_{F}=1$ and $\alpha_{G}=2$ 
    \begin{align*}
        X_{i}^{(n)} := c_{F}(i/n)^{1/\alpha_{F}}((1-U_{i})^{-1/\alpha_{F}} -1),
        \qquad
        C_{i}^{(n)} := c_{G}(i/n)^{1/\alpha_{G}}((1-V_{i})^{-1/\alpha_{G}} -1).
    \end{align*}
    Then $1-F_{n,i}(x) = (1+x/c_{F}(i/n)^{1/\alpha_{F}})^{-\alpha_{F}}$ and $1-G_{n,i}(x) = (1+x/c_{G}(i/n)^{1/\alpha_{G}})^{-\alpha_{G}}$ and so the choices of baseline functions $1-F(x):=x^{-\alpha_{F}}$ and $1-G(x):=x^{-\alpha_{G}}$ ensure that~\eqref{eq_Scedasis_F_and_G} holds.
    \item Heteroscedastic Fr\'echet: for $\alpha_{F}=1$ and $\alpha_{G}=2$,
    \begin{align*}
        X_{i}^{(n)} := (c_{F}(i/n)/\log(1/U_{i}))^{1/\alpha_{F}},
        \qquad
        C_{i}^{(n)} := (c_{G}(i/n)/\log(1/V_{i}))^{1/\alpha_{G}}.
    \end{align*}
    Then $F_{n,i}(x) = \exp(-c_{F}(i/n)x^{-\alpha_{F}})$ and $G_{n,i}(x) = \exp(-c_{G}(i/n)x^{-\alpha_{G}})$, and the usual Taylor expansion of the exponential function yields that~\eqref{eq_Scedasis_F_and_G} holds for the choices $1-F(x):=x^{-\alpha_{F}}$ and $1-G(x):=x^{-\alpha_{G}}$.
    \item Heteroscedastic Burr: for $a_{F}=a_{G}=2$, $b_{F}=1/2$ and $b_{G}=1$,
    \begin{align*}
        &X_{i}^{(n)} := c_{F}(i/n)^{1/(a_{F}b_{F})}((1-U_{i})^{-1/b_{F}} -1)^{1/a_{F}},\\
        &C_{i}^{(n)} := c_{G}(i/n)^{1/(a_{G}b_{G})}((1-V_{i})^{-1/b_{G}} -1)^{1/a_{G}},
    \end{align*}
    for positive constants $a_{F},b_{F},a_{G},b_{G}$. In this case, 
    \begin{align*}
        1-F_{n,i}(x) &= 
        (1+x^{a_{F}}c_{F}(i/n)^{-1/b_{F}})^{-b_{F}}\\
        &=
        c_{F}(i/n)x^{-a_{F}b_{F}}\big(1 + c_{F}(i/n)^{1/b_{F}}x^{-a_{F}} \big)^{-b_{F}}
        =
        c_{F}(i/n)x^{-a_{F}b_{F}}(1 +o(1)),
    \end{align*}
    as $x \to \infty$ using the usual Taylor expansion, and hence~\eqref{eq_Scedasis_F_and_G} holds for the choices $1-F(x):=x^{-a_{F}b_{F}}$ and $1-G(x):=x^{-a_{G}b_{G}}$.
\end{enumerate}

\begin{remark}
\label{remark_heteroscedastic_constructions_satisfy_condition}
For each of the three constructions, define the extended distribution functions $F(x,u)$ and $G(x,u)$ by replacing $c_{F}(i/n)$ and $c_{G}(i/n)$ by $c_{F}(u)$ and $c_{G}(u)$, respectively. Because $c_{F}$ and $c_{G}$ are positive on $(0,1)$ and continuously differentiable with locally Lipschitz first derivatives, these extensions satisfy Condition~\ref{cond_H_H1_C1_F_G_Lip}. Indeed, fix $s\in(0,1)$ and consider a small neighborhood $I_{s}\subset(0,1)$. On $I_{s}$, the functions $c_{F}$ and $c_{G}$ are bounded, non-zero, and their first-order derivatives are bounded and Lipschitz. Differentiation of the three distribution functions and their densities shows that each $\partial G(x,u)/\partial u$ is uniformly bounded in $x$ and Lipschitz in $u$ with a Lipschitz constant independent of $x$. Similarly, each $\partial f(x,u)/\partial u$ admits an integrable envelope in $x$ and is Lipschitz in $u$ with an integrable Lipschitz modulus. Therefore, all the requirements of Remark~\ref{remark_sufficient_condition} are satisfied, and hence so is Condition~\ref{cond_H_H1_C1_F_G_Lip}.
\end{remark}

Denote by $\check{c}_{F}^{k_{n}}(s\vert s_{0})$ a $k_{n}$-dependent scedasis estimator normalized at a fixed $s_{0}\in (0,1)$, based on $n$ observations. Denote by $B$ the number of batches of $n$ realizations, and denote the estimator computed from batch $b$ by $\check{c}_{F}^{k_{n},b}(s\vert s_{0})$. Using this notation, we estimate the pointwise mean squared error (MSE) at a fixed $s$, as well as the pathwise uniform error by
\begin{align*}
    \widehat{\text{MSE}}(\check{c}_{F}^{k_{n}}(s\vert s_{0}))
    &:=
    \frac{1}{B}\sum_{b=1}^{B}\Big(\check{c}_{F}^{k_{n},b}(s\vert s_{0}) - \frac{c_{F}(s)}{c_{F}(s_{0})}\Big)^{2}, 
    \\
    \widehat{E}^{\infty}(\check{c}_{F}^{k_{n}}(\cdot\vert s_{0}))
    &:=
    \frac{1}{B}\sum_{b=1}^{B}\Big\lVert\check{c}_{F}^{k_{n},b}(\cdot\vert s_{0}) - \frac{c_{F}(\cdot)}{c_{F}(s_{0})}\Big\rVert_{\infty},
\end{align*}
where the estimator $\check{c}_{F}^{k_{n}}(s\vert s_{0})$ is either one of the two estimators 
\begin{align*}
    \hat{c}_{F}(s \vert s_{0})
    =
    \frac{1-\amsmathbb{F}^{(n)}(Z_{n:n-k_{n}} \vert s)}{1-\amsmathbb{F}^{(n)}(Z_{n:n-k_{n}} \vert s_{0})},
    \qquad
    \hat{c}_{F}^{\text{EHZ}}(s\vert s_{0})
    =
    \frac{\hat{c}_{F}^{\text{EHZ}}(s)}{\hat{c}_{F}^{\text{EHZ}}(s_{0})}.
\end{align*}
Recall that we fix $s_{0}=1/2$ and $s=2/5$, draw $B=200$ batches of $n=2{,}000$ observations for each of the distributions, and use the uniform kernel.\\
As noted in Section~\ref{section_simulation_study}, the estimated values are only stored when the denominators are non-zero. To quantify the rate at which these estimators are discarded we report in Table~\ref{tab_denominator_failure_rates} the maximum failure rates for each bandwidth and estimator across all batches, distributions and the chosen range of $k_{n}/n$. 
\begin{table}[hbt!]
\centering
\begin{tabular}{lcc}
\toprule
Bandwidth & Einmahl et al. & Proposal \\
\midrule
\(h=0.08\) & 5.0\% & 1.5\% \\
\(h=0.15\) & 0.0\% & 0.0\% \\
\(h=0.22\) & 0.0\% & 0.0\% \\
\bottomrule
\end{tabular}
\caption{Maximum denominator failure rates over the displayed threshold range $k_{n}/n \in [0.005, 0.5]$, based on $B={200}$ batches of $n=2{,}000$ observations. For each bandwidth and method, the maximum is taken across the Lomax, Fr\'echet and Burr heteroscedastic designs.}
\label{tab_denominator_failure_rates}
\end{table}
This table shows that even in the worst cases only a negligible proportion of the estimated values are discarded because their denominators are zero. Hence the displayed simulations are representative of the underlying properties of the estimators.

\subsection{Quantifying bias and integrated squared error}
Using the same distributions and parameter choices, we also quantify the corresponding pointwise bias and pathwise integrated squared error (ISE) for our estimator and the relative Einmahl et al.-benchmark as above
\begin{align*}
    \widehat{\text{Bias}}(\check{c}_{F}^{k_{n}}(s\vert s_{0}))
    &:=
    \frac{1}{B}\sum_{b=1}^{B}\Big(\check{c}_{F}^{k_{n},b}(s\vert s_{0}) - \frac{c_{F}(s)}{c_{F}(s_{0})}\Big),
    \\
    \widehat{\text{ISE}}(\check{c}_{F}^{k_{n}}(\cdot\vert s_{0}))
    &:=
    \frac{1}{B}\sum_{b=1}^{B}\int_{h_{n}}^{1-h_{n}} \Big(\check{c}_{F}^{k_{n},b}(s\vert s_{0}) - \frac{c_{F}(s)}{c_{F}(s_{0})}\Big)^{2} \mathrm{d} s,
\end{align*}

The corresponding errors are depicted in Figure~\ref{fig_simulation_bias_ise}. The ISE is computed over the interior grid $s \in [h_{n},1-h_{n}]$, to retain correct sample sizes.
\begin{figure}[hbt!] 
    \centering
    \includegraphics[width=.75\textwidth]{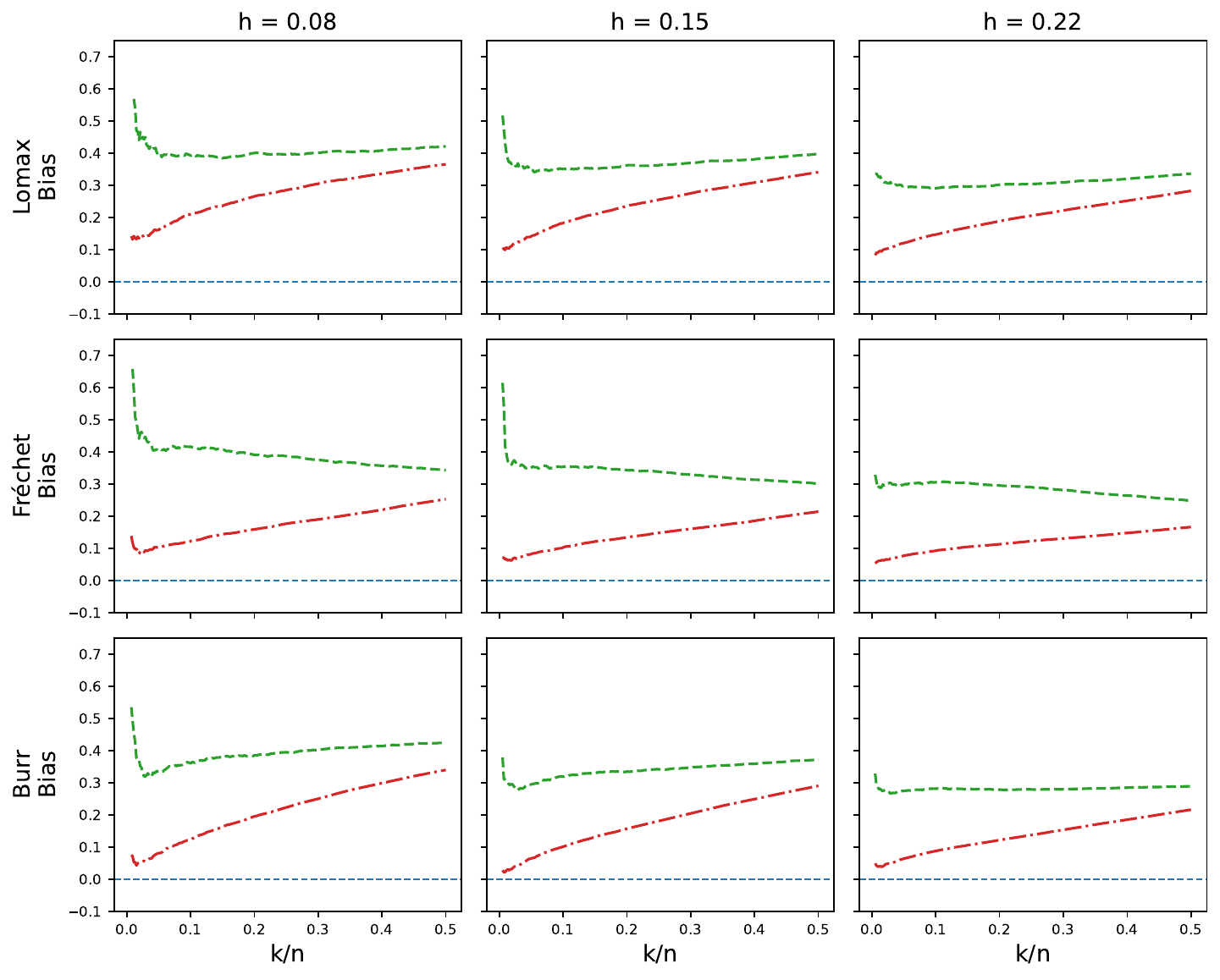} 
    \includegraphics[width=.75\textwidth]{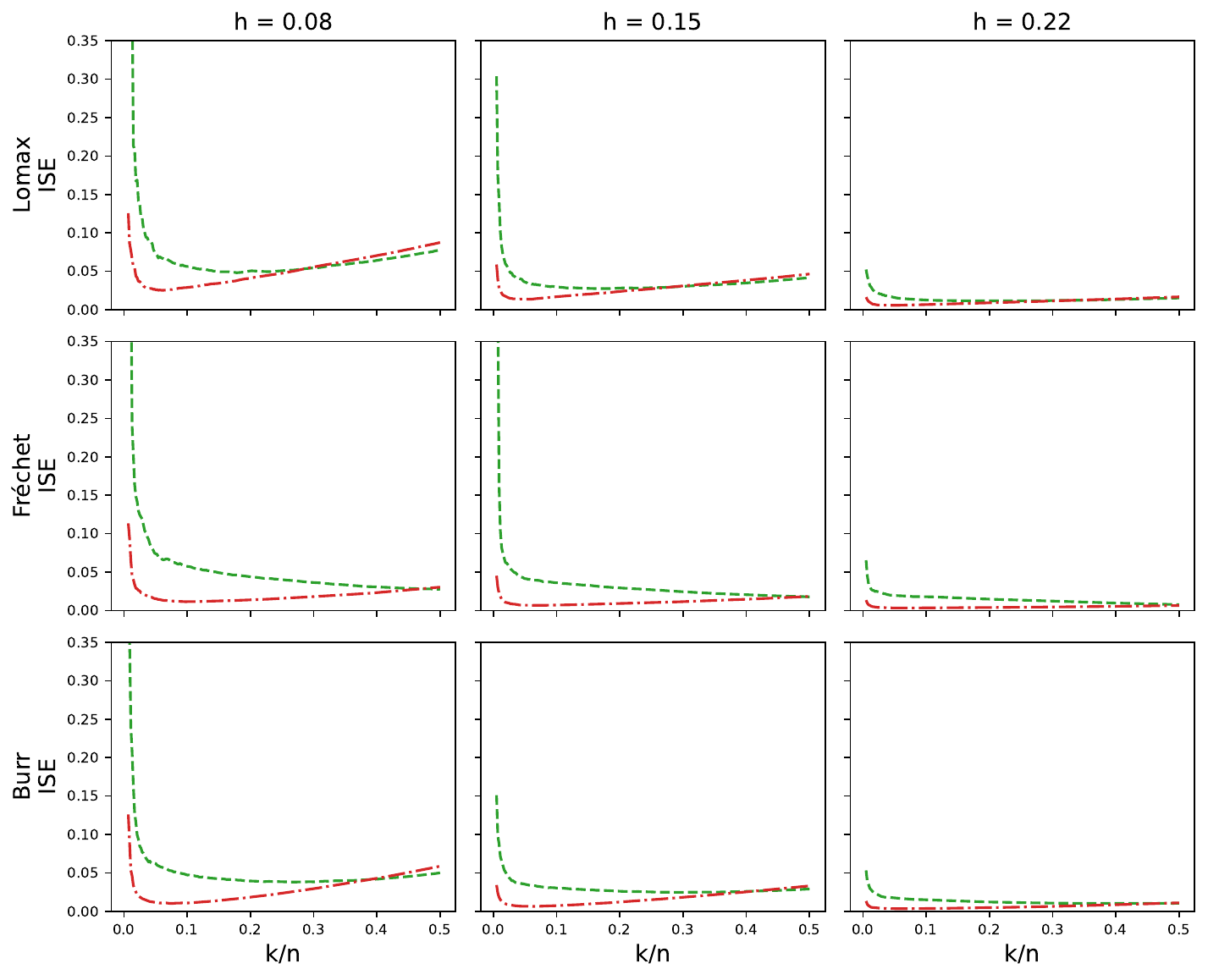} 
    \caption{Simulation study: Estimated pointwise bias (top panel) for $s=2/5$, and pathwise integrated squared error (bottom panel) for $\hat{c}_{F}(\cdot \vert 1/2)$ and $  \hat{c}_{F}^{\mathrm{EHZ}}(\cdot \vert 1/2)$ in red dot-dashed and green dashed lines respectively as functions of $k_{n}/n \in [0.005,1/2]$, where the errors are computed from $B=200$ batches of $n=2{,}000$ simulated observations from heteroscedastic Lomax, Fr\'echet and Burr (in first, second and third row respectively for bias and fourth, fifth and sixth respectively for integrated squared error) for bandwidth values $h_{n}\in \{0.08, 0.15, 0.22\}$ (from first to last column respectively) corresponding to the approximate values of $n^{-1/3}, n^{-1/4}, n^{-1/5}$.}
    \label{fig_simulation_bias_ise}
\end{figure}
Across all three simulation designs, the proposed estimator yields substantially smaller bias than the relative Einmahl et al.-benchmark over the displayed threshold range. The improvement is most pronounced for small and moderate values of $k_{n}/n$, where heterogeneous censoring has the largest effect on the observed tail. For larger bandwidths, the loss curves become smoother and the threshold sensitivity is reduced.
The same conclusions hold for the integrated squared error; however, the proposed estimator only consistently outperforms the relative Einmahl et al.-benchmark in the tail region for $k_{n}/n$ approximately smaller than $0.25$. 
\newpage
{\color{white}.}
\newpage
\section{Additional real data analysis results}\label{appendix_real_data}
In this appendix we provide some additional details on how the proposed estimator can be used for insurance claims. We proceed to explain how the claim sizes have been inflation-adjusted and finally we provide tail diagnostics on the used data.
\subsection{Actuarial considerations}
Insurance companies often cover long-tailed risks subject to random right-censoring, such as workers' compensation, fire insurance and automobile liability insurance; see~\cite{mcneil1997} and~\cite{embrechts1997}. The reporting delay between occurrence and reporting date and the subsequent delay between reporting and settlement induce incomplete observation of the full, or “ultimate,” claim amount. Consequently, actuaries must model claims that are incurred but not yet reported (IBNR) and claims that are reported but not yet settled (RBNS), which are central concerns in claims reserving.\\
As our theoretical results hold under proportional tails without assuming the domain of attraction, our estimator can in principle also be applied to short-tailed claims, but we focus on heavy-tailed claims. \\\\
While classical reserving methods model claim severities on aggregate claims data of homogeneous portfolios, see~\citep{bornhuetter1972, mack1994stochastic}, other papers have proposed modeling individual claims, see~\citep{arjas1989claims, norberg1993prediction} and~\cite{Wuthrich03072018} for a more modern treatment. Our analysis belongs to the latter category.
The present random-censorship framework assumes continuous target and censoring distributions and therefore does not accommodate a point mass at zero in the observed cumulative payments. We consequently retain only claims with strictly positive cumulative payments at the valuation date and restrict attention to RBNS claims for which payment has already begun.\\
In this case, $Z_{i}^{(n)}$ denotes the (strictly positive) cumulative payment of claim $i$ and $\delta_{i}^{(n)}$ indicates whether the claim is open or closed. This allows for the schematic example in Figure~\ref{fig_reserving_triangle}. 

\begin{figure}[hbt!]
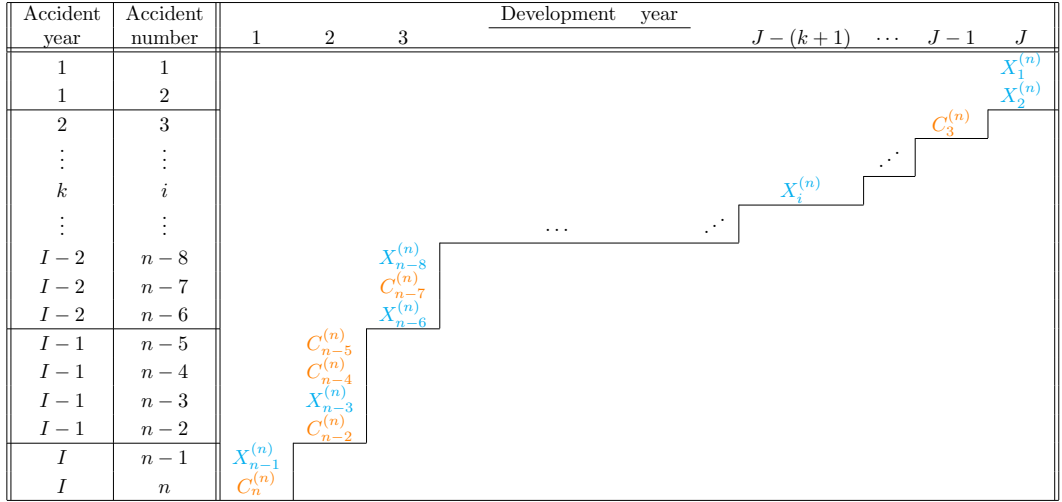

\centering
         \resizebox{14cm}{!}{%
     \begin{tabular}{|| c | c || c c c c c c c c c c c c||}
    \hline
    Accident&Accident&   &&&&&Development&year&&&&     &    \\ \cline{8-9}
    year    &number  &$1$&$2$&$3$&           &    &&&&$J-(k+1)$& $\cdots$ &$J-1$& $J$  \\  [-0.5ex] 
    \hline\hline 
     $1$    &$1$     &   &&&           &    &&&&&&     &{\color{cyan}$X_{1}^{(n)}$}\\ 
     $1$    &$2$     &   &&&           &    &&&&&&     &{\color{cyan}$X_{2}^{(n)}$}\\ \cline{1-2} \cline{14-14} 
     $2$    &$3$     &   &&&           &    &&&&&& \multicolumn{1}{c|}{{\color{orange}$C_{3}^{(n)}$}} & \\   \cline{13-13} 
     \vdots &\vdots  &   &&&           &    &&&&& \multicolumn{1}{c|}{\reflectbox{$\ddots$}}        && \\ \cline{12-12} 
     $k$ &$i$      &   &&&           &    &&&& \multicolumn{1}{c|}{{\color{cyan}$X_{i}^{(n)}$}}  &&& \\ \cline{11-11} 
     \vdots &\vdots  &   &&&           &    &$\cdots$&& \multicolumn{1}{c|}{\reflectbox{$\ddots$}}         &&&&\\ \cline{6-10} 
     $I-2$  & $n-8$  &&& \multicolumn{1}{c|}{{\color{cyan}$X_{n-8}^{(n)}$}}  &&&&&&&&&\\
     $I-2$  & $n-7$  &&& \multicolumn{1}{c|}{{\color{orange}$C_{n-7}^{(n)}$}}  &&&&&&&&&\\
     $I-2$  & $n-6$  &&& \multicolumn{1}{c|}{{\color{cyan}$X_{n-6}^{(n)}$}}  &&&&&&&&&\\ \cline{1-2} \cline{5-5} 
     $I-1$  & $n-5$  &&  \multicolumn{1}{c|}{{\color{orange}$C_{n-5}^{(n)}$}}  &&&&&&&&&&   \\
     $I-1$  & $n-4$  &&  \multicolumn{1}{c|}{{\color{orange}$C_{n-4}^{(n)}$}}  &&&&&&&&&&   \\
     $I-1$  & $n-3$  &&  \multicolumn{1}{c|}{{\color{cyan}$X_{n-3}^{(n)}$}}  &&&&&&&&&&   \\
     $I-1$  & $n-2$  &&  \multicolumn{1}{c|}{{\color{orange}$C_{n-2}^{(n)}$}}  &&&&&&&&&&   \\ \cline{1-2} \cline{4-4} 
     $I$    & $n-1$  &  \multicolumn{1}{c|}{{\color{cyan}$X_{n-1}^{(n)}$}}   &&&&&&&&&&&   \\
     $I$    & $n$    &  \multicolumn{1}{c|}{{\color{orange}$C_{n}^{(n)}$}}     &&&&&&&&&&&   \\
    \hline
    \end{tabular}
}
\caption{Visual representation of a claims reserving triangle. At accident/occurrence year $I$, the $n$ reported cumulative payments are available on the slightly curved diagonal, where payments denoted $X_{i}^{(n)}$ are closed claims $(\delta_{i}^{(n)}=1)$, while $C_{j}^{(n)}$ are open claims $(\delta_{j}^{(n)}=0)$.}
\label{fig_reserving_triangle}
\end{figure}

\begin{remark}
    While the representation of Figure~\ref{fig_reserving_triangle} 
    is sensible for any given accident year $I$, we remark that comparing the triangle with the corresponding triangle at year $I+1$ might reorder the data. Suppose for instance that in year $I+1$ a claim having occurred in accident year $I-1$ (or any other former year) is reported and payments have begun. In this case the triangle at year $I+1$ will have an extra row in accident year $I-1$, thus reordering the $i$-indices of all cumulative claims with occurrence date after that of the newly reported claim. Consequently the reporting delay is a potential source of perturbation when comparing estimated scedasis functions across different time horizons.
\end{remark}

\begin{remark}
    Rather than modeling outstanding claims through the cumulative payments, it is sometimes useful to include the so-called case reserve, which is a best guess on the outstanding payments made by a case worker or AI model shortly after the reporting date. Adding the case reserve to the cumulative payment yields the so-called incurred payment and this quantity is often used in standard reserving methods. While the incurred payment is right-censored until a claim is closed, in the sense that the case reserve might be inadequate or excessive, it should reflect the expected ultimate claim size and thus often behaves like the observation of interest $X_{i}^{(n)}$ rather than its properly right-censored version $Z_{i}^{(n)}$. By construction we require $Z_{i}^{(n)} \leq X_{i}^{(n)}$ and consequently modeling the scedasis directly on the incurred claims may substantially overestimate the scedasis in areas with a high degree of censoring.
\end{remark}

\subsection{Data preprocessing}
All claim amounts in our dataset are converted to constant currency units using the French construction cost index \emph{Indice du coût de la construction des immeubles à usage d'habitation} (ICC), published by~\cite{inseeICC}. Let $J_q$ denote the value of the ICC in quarter $q$, let $q_{i}$ be the quarter containing the occurrence date of claim $i$, and let $q_{0}$ denote the chosen base quarter. We use the adjustment
\begin{align*}
    P_i^{(n),\mathrm{adj}}
    =
    P_i^{(n)} \frac{J_{q_0}}{J_{q_i}}.
\end{align*}
Observations with $Z_{i}^{(n)}=0$ are removed, since zero observations arise for reported claims with cumulative payments below the deductible or with no positive cumulative payment at the valuation date.\\
In Section~\ref{section_real_data_analysis} we fit our proposed estimator to sequentially nested periods. The following numbers of retained observations and censored claims are reported in Table~\ref{tab_real_data_valuation_sample_sizes}.
\begin{table}[hbt!]
\centering
\begin{tabular}{cccccc}
\toprule
Dataset & Valuation year & Raw claims & Retained & Censored & Censored fraction \\
\midrule
$Y0$ & 2007 & 109,992 & 107,728 & 2,654 & 2.5\% \\
$Y2$ & 2005 & 97,046 & 93,229 & 2,012 & 2.2\% \\
$Y5$ & 2002 & 78,350 & 75,444 & 1,733 & 2.3\% \\
$Y8$ & 1999 & 56,756 & 53,037 & 1,874 & 3.5\% \\
\bottomrule
\end{tabular}
\caption{Sample sizes for the retrospective valuation datasets.}
\label{tab_real_data_valuation_sample_sizes}
\end{table}

\subsection{Diagnostics}
By~\eqref{eq_Scedasis_F_and_G} and~\eqref{eq_H_scedasis_dens}, we require that the tail index is constant across time. We investigate this by dividing the observations into two and four equally sized temporal periods plotting a censoring-adjusted version of the Hill estimator for each period. Denote by $\delta_{[n:i]}$ the concomitant corresponding to the order statistic $Z_{n:i}$. The censoring-adjusted Hill estimator is then given by
\begin{align*}
    \hat{\gamma}_{k_{n}} = \Big(\sum_{i=1}^{k_{n}}\log(Z_{n:n-i+1}/Z_{n:n-k_{n}})\Big)/
    \Big(\sum_{i=1}^{k_{n}}\delta_{[n:n-i+1]}\Big).
\end{align*}
The corresponding Hill estimators $k_{n} \mapsto \hat{\gamma}_{k_{n}}$ are depicted in Figure~\ref{fig_freclaim_tail_index_unif_periods} for $k_{n}/n \in [0.0025, 0.2]$.
\begin{figure}[hbt!] 
    \centering
    \includegraphics[width=1\textwidth]{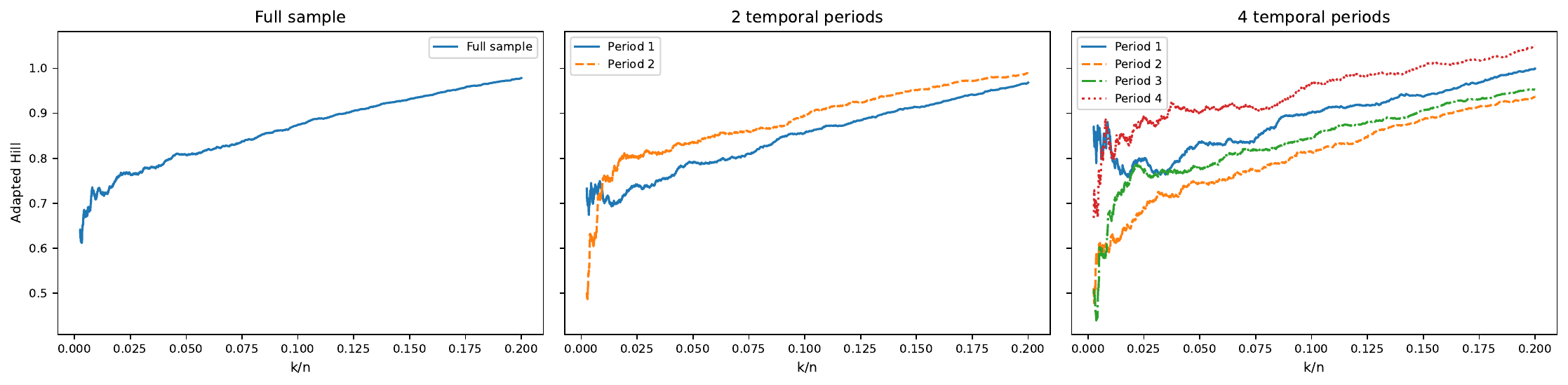}
    \caption{Adapted Hill estimators on sub-divided data sets for $k_{n}/n \in [0.0025, 0.2]$:
    We split the data into one, two and four disjoint and equally sized sub-datasets (left, middle, right respectively) according to the occurrence dates and plot $\hat{\gamma}_{k_{n}}$ against $k_{n}/n$, so that the left plot is based on the full sample, the mid plot is based on the claims with occurrence date in the first and second half of 1992--2007, the right plot based on claims with occurrence date in the first, second, third or fourth quarter of 1992--2007.} 
    \label{fig_freclaim_tail_index_unif_periods}
\end{figure}
While there is not much difference in the estimated tail index between the first and second half of the data we do see some but not severe differentiation when going from two to four periods with highest tail indices in period 1 and 4. The differences are however close in the sense that they vary from 0.6 to 0.9 in the tail region $k_{n}/n \in [0.01, 0.025]$ and thus would give rise to the same number of finite moments if the data were Pareto distributed.\\
Below we conduct an analysis similar to that above, but now we depict the adapted Hill estimates for the nested time horizons used in the valuation of the scedasis density estimator in Section~\ref{section_real_data_analysis}. We note that the tail index of the first period (corresponding to the blue solid line) does appear to be substantially higher than the subsequent nested periods which might challenge our assumption of proportional tails.
\begin{figure}[hbt!] 
    \centering
    \includegraphics[width=.75\textwidth]{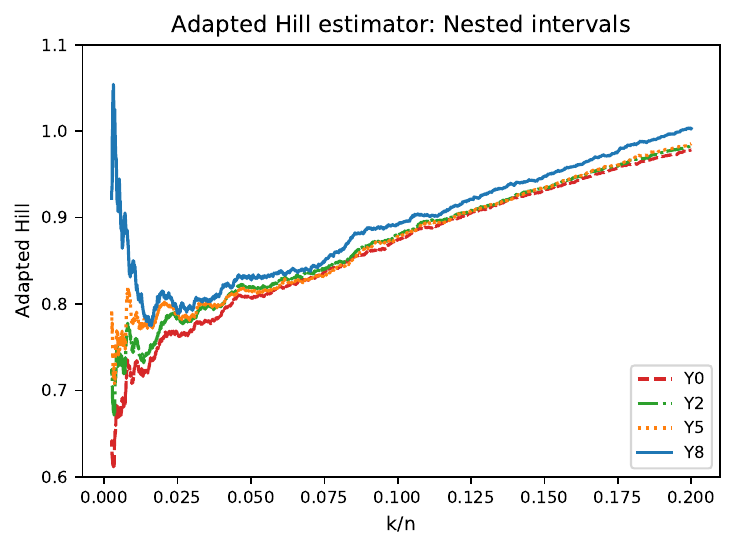}
    \caption{Adapted Hill estimators corresponding to the nested time intervals used in the scedasis valuation for $k_{n}/n \in [0.0025, 0.2]$: Here \texttt{Y0} denotes the full dataset (red, dashed), \texttt{Y2} excludes the two most recent years (green, dot-dashed), \texttt{Y5} excludes the five most recent years  (orange, dotted) and \texttt{Y8} excludes the eight most recent years  (blue, solid).}
    \label{fig_freclaim_tail_valuation_plots}
\end{figure}
\FloatBarrier
\end{appendix}

\end{document}